\documentclass [a4paper,11pt]{amsart}
\usepackage{longtable}
\usepackage{hyperref}
\usepackage[T1]{fontenc}
\usepackage[utf8]{inputenc}
\usepackage[english]{babel}
\usepackage{textcomp}
\usepackage{dsfont}
\usepackage{latexsym}
\usepackage{amssymb}
\usepackage{amsthm}
\usepackage{amsmath}
\DeclareMathAlphabet{\mathpzc}{OT1}{pzc}{m}{en}
\usepackage{yfonts}
\usepackage{xfrac}
\usepackage{newlfont}
\usepackage{graphicx}
\usepackage{mathtools}
\usepackage{comment}
\usepackage{indentfirst}
\usepackage{braket}
\usepackage{mathrsfs}

\newcommand{\atomic}{(L)}
\newcommand{\atomics}{(L')}

\usepackage{etoolbox}

\usepackage{scalerel}[2014/03/10]
\usepackage[usestackEOL]{stackengine}
\newcommand{\dashint}{\,\ThisStyle{\ensurestackMath{%
\stackinset{c}{.2\LMpt}{c}{.5\LMpt}{\SavedStyle-}{\SavedStyle\phantom{\int}}}%
\setbox0=\hbox{$\SavedStyle\int\,$}\kern-\wd0}\int}

\DeclareMathOperator{\card}{Card}

\DeclareMathOperator{\supp}{Supp}

\DeclareMathOperator{\tr}{Tr}

\DeclareMathOperator{\Hol}{Hol}

\DeclareMathOperator{\Ima}{Im}
\DeclareMathOperator{\Rea}{Re}

\newcommand{\Supp}[1]{\supp\left( #1\right) }

\newcommand{\ee}{\mathrm{e}}

\newcommand{\vect}[1]{\mathbf{{#1}}}
\newcommand{\dd}{\mathrm{d}}

\DeclarePairedDelimiter{\abs}{\lvert}{\rvert}

\DeclarePairedDelimiter{\norm}{\lVert}{\rVert}

\let\originalleft\left
\let\originalright\right
\renewcommand{\left}{\mathopen{}\mathclose\bgroup\originalleft}
\renewcommand{\right}{\aftergroup\egroup\originalright}

\newcommand{\N}{\mathds{N}}
\newcommand{\Z}{\mathds{Z}}

\newcommand{\C}{\mathds{C}}

\newcommand{\R}{\mathds{R}}

\newcommand{\Ff}{\mathfrak{F}}

\newcommand{\Bc}{\mathcal{B}}

\newcommand{\Ec}{\mathcal{E}}
\newcommand{\Fc}{\mathcal{F}}

\newcommand{\Hc}{\mathcal{H}}

\newcommand{\Lc}{\mathcal{L}}
\newcommand{\cM}{\mathcal{M}}
\newcommand{\Nc}{\mathcal{N}}

\newcommand{\Sc}{\mathcal{S}}
\newcommand{\Tc}{\mathcal{T}}

\newcommand{\Ms}{\mathscr{M}}

\newcommand{\meg}{\leqslant}
\newcommand{\Meg}{\geqslant}
\newcommand{\eps}{\varepsilon}
\renewcommand{\phi}{\varphi}
\newcommand{\mi}{\mu}

\newcommand{\Lin}{\mathscr{L}}

\newcommand{\half}{\textstyle{\frac12}}

\newcommand{\itb}{\item[{\tiny $\bullet$}]}
\newcommand{\vb}{\vect b}
\newcommand{\vd}{\vect d}
\newcommand{\vm}{\vect m}
\newcommand{\vs}{\vect s}

\begin{document}

\numberwithin{equation}{section} 
\theoremstyle{definition}
\newtheorem{deff}{Definition}[section]

\newtheorem{oss}[deff]{Remark}

\newtheorem{ass}[deff]{Assumptions}

\newtheorem{nott}[deff]{Notation}

\theoremstyle{plain}
\newtheorem{teo}[deff]{Theorem}

\newtheorem{lem}[deff]{Lemma}

\newtheorem{prop}[deff]{Proposition}

\newtheorem{cor}[deff]{Corollary}

\title[Carleson measures on Siegel domains]{Carleson and reverse
  Carleson measures\\ on homogeneous Siegel domains} 
\author[M. Calzi, M. M. Peloso]{Mattia Calzi, 
Marco M. Peloso}

\address{Dipartimento di Matematica, Universit\`a degli Studi di
  Milano, Via C. Saldini 50, 20133 Milano, Italy}
\email{{\tt mattia.calzi@unimi.it}}
\email{{\tt marco.peloso@unimi.it}}

\thanks{{\em Math Subject Classification 2020:} Primary: 32A36; Secondary: 32A10, 32M10. }
\thanks{Both authors are members of the
Gruppo Nazionale per l'Analisi Matematica, la Probabilit\`a e le
loro 
Applicazioni (GNAMPA) of the Istituto Nazionale di Alta Matematica
(INdAM) and are  partially supported by the 2020
GNAMPA grant {\em Fractional Laplacians and subLaplacians on Lie groups and trees}.}

\begin{abstract}
In this paper we study Carleson and reverse Carleson measures on
holomorphic function spaces on a homogeneous Siegel domain of Type
II.  We prove several necessary conditions and sufficient conditions
in order for a measure $\mi$ to be Carleson and reverse Carleson on
mixed-normed weighted Bergman spaces.  
\end{abstract}
\maketitle

\section{Introduction}

Carleson measures were first introduced by L.\
Carleson~\cite{Carleson1,Carleson2}  in order to study the corona
problem in the classical Hardy spaces on the unit disc. The study of
these measures have flourished since then, and has been generalized to
several different settings.
Form a general perspective, they can be effectively defined as follows.
Given a locally compact space $\Tc$   and a (quasi-)Banach space
$X$ of (say) continuous functions on $\Tc$, a positive Radon measure
$\mi$ on $\Tc$ is said to be a $p$-Carleson measure for $X$ ($p\in
]0,\infty[$) if there is a constant $C>0$ such that   
\[
\norm{f}_{L^p(\mi)}\meg C \norm{f}_X
\]
for every $f\in X$. Switching the roles of $X$ and $L^p(\mi)$ in the
above inequality, one then gets the definition of a reverse
$p$-Carleson measure for $X$. This kind of measures was essentially
introduced in a series of papers by D.\ H.\
Luecking~\cite{Luecking,Luecking3,Luecking4,Luecking2} for weighted
Bergman spaces on the unit disc or more general domains, and has been
later generalized to different settings by several authors
(cf.~\cite{FricainHR} and the references therein). 

In this paper we study Carleson and reverse Carleson measures on
weighted Bergman spaces on homogeneous Siegel domain, which we are now
about to introduce. 

Let $E$  be a complex hilbertian space of dimension $n$, $F$ a real
hilbertian space of dimension $m>0$ and $\Omega$ a  convex cone
in $F$ not containing any straight lines.\footnote{We shall generally
denote by $\langle\,\cdot\,,\,\cdot\,\rangle$ bilinear pairings and
real scalar products, and by $\langle\,\cdot\,\vert\,\cdot\,\rangle$
sesquilinear pairings and complex scalar products, without
specifying the involved spaces.} 
We also assume that $\Omega$ is (affine-)homogeneous, that is the group $G(\Omega)$ of linear transformation of $F$ that preserve $\Omega$ acts transitively
on $\Omega$. 

Let $\Phi\colon E\times E\to F_\C $   be a non-degenerate hermitian
mapping such that $\Phi(\zeta)\coloneqq \Phi(\zeta,\zeta)\in
\overline\Omega$ for all $\zeta\in E$.  Define the  Siegel domain of
type II associated with the cone $\Omega$ and the mapping $\Phi$ as  
\[
D \coloneqq \Set{ (\zeta, z) \in E \times F_\C \colon
\rho(\zeta,z)\coloneqq \Ima z -  \Phi(\zeta) \in \Omega } .
\]
In the case $n=0$, i.e. $E=\Set{0}$, $D$ is said to be of Type I, and
is
often called a {\em  tubular domain} over the cone $\Omega$.  
The domain $D$ is said to be homogeneous if for every
$(\zeta,z),(\zeta',z')\in D$ there is a biholomorphism $\phi$ of $D$
such that $\phi(\zeta,z)=(\zeta',z')$. It turns out that one may then
assume $\phi$ to be affine, cf., e.g.,~\cite[Theorem 2.3]{Xu}. More
precisely, $D$ is homogeneous if and only if for every $h,h'\in
\Omega$ there are $t\in GL(F)$ and $g\in GL(E)$ such that $th=h'$ and
such that $t\Phi=\Phi(g\times g)$ (so that $g\times t$ preserves
$D$), cf., e.g.,~\cite[Propositions 2.1 and 2.2]{Murakami}. In
particular, if $D$ is homogeneous, then $\Omega$ is homogeneous.  

The domain $D$ is said to be symmetric if it is homogeneous and admits
an involutive biholomorphism with an isolated fixed point. If $D$ is
symmetric, then $\Omega$ is symmetric, that is, homogeneous and
self-dual. Conversely, if $\Omega$ is symmetric \emph{and $D$ is a
tubular domain}, then  $D$ is symmetric (cf.,
e.g.,~\cite[Theorem]{Satake} for more details on various
characterizations of symmetric Siegel domains).

The \v{S}ilov boundary  of $D$ is the set 
\[
bD \coloneqq \Set{ (\zeta, z) \in E \times F_\C \colon\rho(\zeta,z) =0 } ,
\]
and it can be identified with $E\times F$ via the mapping  $ (\zeta,
x+i\Phi(\zeta,\zeta))\mapsto(\zeta,x)$. In addition, $bD$ admits a
step-2 nilpotent Lie group 
structure,  whose product can be described as follows: 
\[
(\zeta,x)(\zeta',x')=(\zeta+\zeta',x+x'+2\Ima
\Phi(\zeta,\zeta')),
\]
for $(\zeta,x),(\zeta',x')\in E\times F$, 
see e.g.~\cite[Section 1.1]{CalziPeloso}. We denote by
$\Nc$ the set $E\times F$ endowed with this group structure. 

Observe that, by definition, $\rho$ maps $ D$ into $\Omega$, and that
the fibres of $\rho$, namely the sets $b D+(0, i h)$ for $h\in
\Omega$, give rise to a foliation of $D$. Given a function $f$
defined on $D$, we shall often denote by $f_h$ the restriction of $f$
to $b D+(0,i h)$, interpreted as a function on $\Nc$ for the sake of
convenience, so that 
\[
f_h (\zeta,x)=f(\zeta,x+i\Phi(\zeta)+i h)
\]
for every $h\in \Omega$ and for every $(\zeta,x)\in \Nc$. Observe
that, identifying  $b D+(0,i h)$ with $\Nc$ as above for every $h\in
\Omega$, we get a left action of $b D$ on $D$ by affine
biholomorphisms.

For $p,q\in ]0,\infty]$ and $\vs\in\R^r$, the
weighted Bergman spaces are defined as\footnote{The definitions of the
rank $r$ of $\Omega$,  of the `generalized power functions'
$\Delta_\Omega^{\vect s}$ and of $\vd\in\R^r$   are deferred to Section~\ref{sec:2}.} 
\begin{equation}\label{eq:1}
A^{p,q}_\vs (D)\coloneqq \Set{  f\in\Hol(D)\colon\int_\Omega \Big(
\int_\Nc\abs{f_h(\zeta,x)}^p \,  \dd (\zeta,x) \Big)^{q/p}
\Delta^{q \vs}_\Omega (h)  \, \dd \nu_\Omega(h)<\infty\,  } 
\end{equation}
(modification if $\max(p,q)=\infty$), where $\dd (\zeta,x)$ denotes a
Haar measure on $\Nc$ and $\nu_\Omega$ denotes a positive
$G(\Omega)$-invariant measure   on $\Omega$, both  fixed and  unique
up to a 
multiplicative constant. We shall often simply write $A^{p,q}_\vs$ instead of $A^{p,q}_\vs (D)$. 

We remark that the spaces $A^{p,p}_\vs$ are the weighted Bergman spaces,  
the unweighted case corresponding to the value $\vs=-\vd/p$, while the
spaces $A^{p,\infty}_{\vect 0}$ are the classical Hardy spaces.  

\medskip

We denote by  $\cM_+(D)$ the space of positive Radon measures.

\begin{deff}\label{def:dom-set}
Given  a Banach space $X$ of functions on $D$ and $p\in ]0,\infty[$, a
measure $\mi\in\cM_+(D)$  is said to be $p$-Carleson for $X$ if there
exists a constant $C>0$ such that  
\[
\norm{f}_{L^p(\mi)}\le C \norm{f}_X
\]
for every $f\in X$.
If, in addition, the canonical mapping $X\to L^p(\mi)$ is compact,
then $\mi$ is said to be a  compact or   vanishing $p$-Carleson for $X$. 

A measure $\mu\in\cM_+(D)$ is said to be a reverse $p$-Carleson
measure for $X$ if there exists $C>0$ such that  
\[
\norm{f}_X\le C \norm{f}_{L^p(\mi)}.
\]
for every $f\in X$.

Finally, a measure $\mu\in\cM_+(D)$ is said to be a sampling, or dominant, measure for $X$ 
if it is both Carleson and reverse Carleson for $X$, in other words,
if $X$ embeds as a closed subspace of $L^p(\mi)$.
\end{deff}

Notice that, even though the definition of a $p$-Carleson and of a
reverse $p$-Carleson measure is quite standard, the definition of a
vanishing $p$-Carleson measure is slightly more
problematic. Originally, Carleson measures were characterized
requiring a suitable function  $F$ to be bounded, and vanishing
Carleson measures were characterized requiring $F$ to vanish `at
infinity'. As noted in~\cite{Power}, this corresponds to the
compactness of the canonical mapping $X\to L^p(\mi)$. In the
literature also appear apparently different, but still equivalent,
definitions of a vanishing Carleson measure
(cf.,e.g.,~\cite{HuLvZhu}).

In this paper we consider the  problem of obtaining necessary and
sufficients conditions for the measure $\mu$ to be a $p$-Carleson,  a vanishing
$p$-Carleson, or a reverse $p$-Carleson measure for the spaces $A^{p,q}_\vs$.
\medskip

There  exists a vast literature on Carleson measures on holomorphic
function  spaces, in one and several variables, and it is impossible
to give a proper and complete account of it.  Hence, we limit
ourselves to the papers most relevant for this work, and extend our
apologies to all other
authors. In~\cite{Carleson1,Carleson2,Duren,Luecking8}, Carleson
measures for the Hardy or Hardy--Sobolev spaces on the unit disc are
studied. In in~\cite{Stegenga,Wu} Carleson measures for the Dirichlet
and weighted Direchlet spaces still on the unit disc are
characterized. In higher dimensions, we mention~\cite{CimaWogen},
which deals with Hardy spaces in the unit ball,~\cite{Luecking7},
which deals with Bergman spaces on general
domains,~\cite{ArcozziRochbergSawyer3,ArcozziRochbergSawyer1,ArcozziRochbergSawyer5,ArcozziRochbergSawyer2,ArcozziRochbergSawyer4},
which deal with holomorphic Sobolev--Besov spaces on the unit
ball,~\cite{AbateSaracco,HuLvZhu,Abate,AbateRaissy,AbateMongodiRaissy},
which deals with Bergman spaces on strongly 
pseudoconvex domains,
and~\cite{NanaSehba,BekolleSehbaTchoundja,BekolleSehba}, which deal
with various function spaces  on Siegel domains.

Carleson measures still draw a lot of attention, in connection with
many problems in analysis such 
as operator theory, interpolation, and boundary behavior problems,
just to mention a few.  The results in this paper are applied in the
study of Toeplitz and Ces\`aro operators on weighted Bergman  spaces on
Siegel domains of Type II, see~\cite{T&C}.

On the other hand, reverse Carleson measures appeared at a later time,
and much less is known about them.  They have been studied by D.\ H.\
Luecking in the case of Bergman spaces on the unit disc
\cite{Luecking,Luecking2,Luecking3,Luecking4,Luecking5}. See also the recent
survey~\cite{FricainHR} and references therein.

Among reverse Carleson measures, those which are also Carleson
measures form a particularly tractable subclass, for which better
results are generally available. Such measures are often called
\emph{sampling measures}, since they can be considered as
generalizations of sampling sequences.  
A complete characterization of sampling measures is available on the
unit disc (cf.~\cite{Luecking5}), but not on more general
domains. Nonetheless, the techniques developed
in~\cite{Luecking,Luecking3,Luecking4} to characterize dominant (or
sampling) sets can be effectively extended to general weighted Bergman
spaces on homogeneous Siegel domains, as already noted
in~\cite{Luecking4}. A subset $G$ of $D$ is said to be \emph{dominant} for
$A^{p,p}_\vs$ if the measure
$\chi_G(\Delta_\Omega^{p\vs+\vd}\circ \rho)\cdot \Hc^{2n+2m}$. As it turns out, this notion is
independent of both $\vs$ and $p$.

In this paper we consider weighted Bergman spaces on a homogeneous
Siegel domain $D$ of Type II. 
We  provide necessary and sufficient conditions   for a measure
$\mi\in\cM_+(D)$ to be a $p$-Carleson measure for $A^{p_1,q_1}_\vs(D)$ 
when $p_1,q_1\le p$  (Theorems~\ref{main:1} and~\ref{teo:4}), as well
as in the general case,  under suitable additional assumptions
(Theorems~\ref{main:2} and~\ref{teo:4}).  
We also provide a sufficient condition for a measure to be a  reverse
Carleson measure (Theorem~\ref{prop:29}), as well as necessary
conditions and sufficient conditions for a measure $\mi\in\cM_+(D)$ to
be a sampling measure for $A^{p,p}_\vs(D)$
(Theorems~\ref{prop:31},~\ref{prop:30}, and~\ref{prop:27}).   
Finally, we remark that analysis of the theory of holomorphic
function spaces, in particular on homogeneous Siegel domains of Type
II, is a quite active area of research, see e.g. \cite{ACMPS,AMPS,MPS} and
references therein for some recent results.

The paper is organized as follows. 
In Section~\ref{sec:C+} we review our main results in the particular
case $D=\C_+$ for the ease of the reader.   
In Section~\ref{sec:2} we give the
definition of homogeneous Siegel domains of Type  II, while in
Section~\ref{sec:fnct-sp} we give the definitions of several function
spaces on $D$ and recall some of their most relevant properties which
are involved in our analysis.    
In Section~\ref{sec:Carleson}, we prove necessary conditions and sufficient
conditions for a measure $\mi\in\cM_+(D)$ to be a Carleson measure for
the weighted Bergman space $A^{p,q}_\vs$.  Section~\ref{sec:5} is
devoted to an extension of a result by Hardy and Littlewood about
embedding the Hardy space into a weighted Bergman space. Finally, in
Section~\ref{sec:6}, we prove our results concerning reverse Carleson and
sampling measures.

\section{Carleson and Reverse Carleson Measures on $\C_+$}\label{sec:C+}

In this brief section we present our main results for weighted Bergman spaces on $\C_+$, where they become particularly simple.

In this case, for $p,q\in ]0,\infty]$ and $s\in\R$,
\[
A^{p,q}_s=\Set{f\in \Hol(\C_+)\colon  \int_0^\infty y^{qs} \left(\int_\R \abs{f(x+iy)}^p\,\dd y\right)^{q/p}\,\frac{\dd y}{y}<\infty}
\] 
(modification when $\max(p,q)=\infty$). Then, $A^{p,q}_s=0$ if $s<0$ or $s=0$ and $q<\infty$. In addition, $A^{p,\infty}_0$ is the Hardy space $H^p$.

\begin{teo}
Take $p_1,q_1,p\in ]0,\infty]$ with $p<\infty$, and $s>0$. Take a compact neighbourhood $K$ of $i$ in $\C_+$, $\mi\in \cM_+(D)$, and define
\[
M_K(\mi)\colon \C_+ \ni z\mapsto \mi(\Rea z+(\Ima z)K)\in \R_+. 
\]
Then, the following conditions are equivalent:
\begin{enumerate}
\item[(1)] $A^{p_1,q_1}_s$ embeds continuously into $L^p(\mi)$;

\item[(2)] setting $p^*\coloneqq \max(1,p_1/p)'$, $q^*\coloneqq \max(1,q_1/p)'$, and $s^*\coloneqq -\max(1,p/p_1)-p s$,
\[
\int_0^\infty y^{q^*s^*} \left(\int_\R \abs{M_K(\mi)(x+iy)}^{p^*}\,\dd y\right)^{q^*/p^*}\,\frac{\dd y}{y}<\infty
\]
(modification when $\max(p^*,q^*)=\infty$).
\end{enumerate}
\end{teo}

(Cf.~Theorem~\ref{main:2} and Propositions~\ref{prop:23bis} and~\ref{prop:25} below.)

Even though Carleson measures for the Hardy spaces
$A^{p,\infty}_0(\C_+)$ have been characterized in the literature, it
is known that the problem of determining the Carleson measures for
Hardy spaces on higher rank irreducible Siegel domains  is highly non-trivial.

Concerning sampling measures, Theorem~\ref{prop:27} becomes:

\begin{teo}
Take $p,s>0$, and a $p$-Carleson measure $\mi$ for $A^{p,p}_s$. Assume that the support of every vague cluster point of every sequence of measures of the form $ y^{p s+1}\, (x+y\,\cdot)_*(\mi) $  ($x\in \R, y>0$) is a set of uniqueness for $A^{q,q}_{(p/q)s}$ for some (fixed) $q\in ]0,p[$.

Then, $\mi$ is a $p$-sampling measure for $A^{p,p}_s$.
\end{teo}

Observe that the above sufficient condition is likely to be necessary
as well, \emph{in this case}, since the corresponding assertion on the
unit disc is true (cf.~\cite{Luecking5}).  

Nonetheless, since the above sufficient condition need not be easily managed, other sufficient conditions can be provided.

\begin{prop}\label{prop:20}
Take $p,s,\eps,C'>0$, and a compact neighbourhood $K$ of $i$ in
$\C_+$. Then, there is a compact neighbourhood $K'$ of $i$ in $\C_+$
such that, for every $\mi\in \cM_+(\C_+)$ such that 
\[
N(\mi)\coloneqq \sup_{x\in \R,y>0} y^{-1-p s}\mi(x+ y K')<\infty
\]
and such that
\[
\Hc^2\left(\Set{z\in x+y K\colon \mi(\Rea z+ (\Ima z) K')\Meg \eps (\Ima z)^{1+p s} N(\mi)  }\right)\Meg C' y^2
\]
for every $x\in \R$ and for every $y>0$, is a $p$-sampling measure for $A^{p,p}_s$.
\end{prop}

The above result is a consequence of a characterization of dominant
subsets of $\C_+$ (cf.~Theorem~\ref{prop:31} below) which is very
close to what one obtains applying Proposition~\ref{prop:20} to the
measure $\chi_G (\Ima\,\cdot\,)^{p s-1}\cdot \Hc^2$ and to some
$\eps\in ]0,1[$, where $G$ is a measurable subset of $\C_+$. 
Here and in what follows, we denote by $\Hc^m$ the $m$-dimensional Hausdorff
measure.     

\begin{prop}\label{prop:21}
A subset $G$ of $\C_+$ is dominant if and only if there are a constant $C>0$ and a compact neighbourhood $K$ of $i$ such that
\[
\Hc^2(G\cap (x+y K))\Meg C y^2
\]
for every $x\in \R$ and for every $y>0$.
\end{prop}

Finally, the following result provides a sufficient condition for a
(not necessarily Carleson) measure to be reverse Carleson  (see
Theorem~\ref{prop:29}). 

\begin{teo}  
Take $p_1,q_1 ,p\in ]0,\infty]$ with $p<\infty$, and $s>0$ or $s\Meg0$ if $q_1=\infty$. Define $p^*\coloneqq
p_1/(p-p_1)_+$ and $ q^*\coloneqq q_1/(p-q_1)_+$. 

Then, there are $\delta>0$ and a constant $C>0$ such that every  $\mi\in\cM_+(D)$ such that
\[
\norm*{  \bigg( \frac{2^{k p[s+1/p_1]}}{\mi([\delta 2^k j,\delta 2^k (j+1) ]+i [\delta 2^k,\delta2^{k+1}]  )} \bigg)_{j,k}
}_{\ell^{p^*, q^*}(\Z,\Z)}  <\infty,
\]
is  a $p$-reverse Carleson for $A^{p_1,q_1}_s$.    
\end{teo}

\section{Homogeneous Siegel domains of Type  II}
\label{sec:2}

For the general theory of homogeneous Siegel domains of Type II and
further details we refer the reader  to~\cite{Vinberg,Gindikin,CalziPeloso}.  
Throughout the paper we adopt the notation of~\cite{CalziPeloso}.

\subsection{The structure of homogeneous cones and a description of $D$}

\begin{deff}\label{def:A}
Take $r\in\N$ and let $A=\bigoplus_{j,k1}^r A_{j,k}$ be a graded algebra
endowed with a linear involution~$^*$ such that the following hold:  
\begin{itemize}
\itb $A_{j,k}A_{k',\ell}\subseteq A_{j,\ell}$ if $k=k'$ and
$A_{j,k}A_{k',\ell}=\Set{0}$ if $k\neq k'$;
\itb $A_0=\Set{0}$ and $\dim A_{j,j}=1$ for every $j=1,\dots,r$;
\itb $A_0=\Set{0}$ and $\dim A_{j,j}=1$ for every $j=1,\dots,r$;
\itb for every $j=1,\dots, r$ there exists $e_j\in A_{j,j}$ such
that the left multiplication by $e_j$ is the identity on
$A_{j,k}$ and the  right multiplication by $e_j$ is the
identity on $A_{k,j}$, for every $k=1,\dots, r$; 
\itb if we define $\tr a=\sum_{j=1}^r \langle e_j', a\rangle$ for
every $a\in A$, where $e'_j$ is the unique graded linear functional on
$A$ which takes the value $1$ at $e_j$ ($j=1,\dots,r$), then the
mapping $(a,b)\mapsto \tr(a b)$ is symmetric; 
\itb for every $a,b,c\in A$, $\tr(a(bc))=\tr((ab)c)$;
\itb  $A^*_{j,k}=A_{k,j}$ for every $j,k=1,\dots,r$ and $(ab)^*=b^* a^*$ for every $a,b\in A$;
\itb the symmetric bilinear mapping $(a,b)\mapsto \tr(a^*b)$ is positive and non-degenerate;
\itb setting $T\coloneqq\bigoplus_{j\meg k} A_{j,k}$, one has $t(u w)=(t u)w$ and $t( u u^*)=(t u) u^*$ for every $t,u,w\in T$.
\end{itemize}
Then, we say that $A$ is a $T$-algebra. 
\end{deff}

\begin{deff}
Given a $T$-algebra $A$, define 
\begin{equation}\label{def:T,H,C(A)}
\begin{aligned}
T_+&\coloneqq \Set{a\in T\colon \langle e_j',a\rangle>0 \quad \forall j=1,\dots, r }, &
H(A) &\coloneqq\Set{a\in A\colon a=a^*},\\
C(A)&\coloneqq \Set{t t^*\colon t\in T_+}, &
C'(A)&\coloneqq \Set{t^* t\colon t\in T_+}.
\end{aligned}
\end{equation}
The cones $C(A)$ and $C'(A)$ are said to have rank $r$.\index{Rank!of a homogeneous cone}
\end{deff}

The following theorem was proved in~\cite{Vinberg}.

\begin{teo}\label{teo:Vinberg}
Let $A$ be a $T$-algebra. Then, the following hold: 
\begin{itemize}
\item[(i)] $T_+$, endowed with the product induced by $A$, is a Lie group;

\item[(ii)] $C(A)$ and $C'(A)$ are   homogeneous   cones, and are dual to
one another with respect to the scalar product on  $H(A)$ ;\footnote{  In other words, $C(A)=\Set{x\in H(A)\colon \forall \lambda\in \overline{C'(A)}\setminus \Set{0}\quad \tr(\lambda x) >0 }$.}
\item[(iii)] $(t x) t^*= t(x t^*)$ and $(t^* x) t = t^* (x t)$ for every $t\in T_+$ and for every $x\in   H(A)$ ;
\item[(iv)] the mappings 
\[
(t,x)\mapsto t x t^* \qquad \text{and}\qquad (t,x)\mapsto t^* x t
\]
are simply transitive left and right actions of $T_+$ on $C(A)$ and
$C'(A)$, respectively, which are dual to one another with respect to
the scalar product on   $H(A)$ . 
\end{itemize}

In addition, if $\Omega$ is a homogeneous cone in a finite-dimensional  vector space  $F$ over $\R$, then, there exist a $T$-algebra $A'$  and an isomorphism $\Psi\colon H(A')\to F'$ such that $\Psi(C(A'))=\Omega$.
\end{teo} 

By the previous results, given a homogeneous cone $\Omega$, we may select a  subgroup $T_+$ of $G(\Omega)$ which acts simply transitively on $\Omega$ and,  by transposition, on the dual con $\Omega'$.  
In order to avoid some notational inconveniences, we shall write $t\cdot h$
instead of $t(h)$ or $t h$, and   $\lambda\cdot t$ instead of $\lambda
\circ t$ or $\lambda t$, for every $t\in T_+$, for every $h\in
\Omega$, and for every $\lambda\in \Omega'$. 
Then, the following hold (cf., e.g.,~\cite[Lemma~2.9]{CalziPeloso}):
\begin{itemize}
\itb for every $t\in T_+$ there exists $g\in GL(E)$ such that
$t\cdot\Phi = \Phi (g\times g)$;\footnote{This follows from the
requirement that $D$ be homogeneous (for a suitable choice of
$T_+$).} 
\itb for every $\vs\in\C^r$ the mapping $\Delta^\vs\colon T_+\to \C^*$ defined  by
$\Delta^\vs (t) \coloneqq\prod_{j=1}^r \langle e_j', t\rangle^{2 s_j}$  is a group
homomorphism;
\itb there is $\vb\in \R^r$ such that
$\Delta^{-\vb}(t)=\det_{\R}(g)=\abs{\det_\C(g)}^2$ for every $t\in T_+$ and for every $g\in GL(E)$ such that $t\cdot \Phi=\Phi(g\times g)$.
\end{itemize}
We recall that $D$ is a tube domain, that is, $E=\Set{0}$, if and only
if $\vb=\vect 0$, cf.~\cite[Remark 2.13]{CalziPeloso}.  

\subsection{The generalized  power and Gamma functions}
Using the characters $\Delta^\vs$ we define the {\em
generalized power functions} $\Delta_\Omega^\vs$ and
$\Delta_{\Omega'}^\vs$ on $\Omega$ and $\Omega'$,  respectively.

\begin{deff}
Fix base points $e_\Omega \in \Omega$ and  $e_{\Omega'} \in \Omega'$, 
and  define the {\em  generalized power functions} $\Delta_\Omega^\vs$
and $\Delta_{\Omega'}^\vs$, for $\vs\in\C^r$, by 
\[
\Delta_\Omega^\vs (h) = \Delta^\vs(t) \qquad \text{and} \qquad \Delta_\Omega^\vs (\lambda) = \Delta^\vs(t')
\]
where $t,t'$ are the \emph{unique} elements of $T_+$ such that
$h=t\cdot e_\Omega$ and $\lambda=e_\Omega\cdot t'$, respectively.  
We also set $m_{j,k}\coloneqq\dim A_{j,k}$ and define  
\begin{equation}\label{def:m,m'}
\vm\coloneqq\Big(\sum_{k>j} m_{j,k}\Big)_{j=1,\dots,r},\qquad \vm' \coloneqq\Big(\sum_{k<j} m_{j,k}\Big)_{j=1,\dots,r} .
\end{equation}
The generalized Gamma functions on $\Omega$ and $\Omega'$ are defined for
$\Rea\vs\in \frac{1}{2}{\vect{m}}+(\R_+^*)^r$ as
\[
\Gamma_\Omega(\vs)  \coloneqq\int_{\Omega} \ee^{-\langle
e_{\Omega'},h\rangle}\Delta^{\vs}_\Omega(h)\,\dd \nu_\Omega(h)=
c  \prod_{j=1}^r \Gamma\left( s_j-\frac{m_j}{2} \right) 
\]
and, for $\Rea\vs\in \frac{1}{2}\vect{m}'+(\R_+^*)^r$,
\[
\Gamma_{\Omega'}(\vs)\coloneqq \int_{\Omega'} \ee^{-\langle \lambda,e_{\Omega}\rangle}\Delta^{\vs}_{\Omega'}(\lambda)\,\dd\nu_{\Omega'}(\lambda) = c  \prod_{j=1}^r \Gamma\left( s_j-\frac{m'_j}{2} \right),
\]
where $c>0$ is a suitable constant.

Furthermore, we set $  \vd\coloneqq -\left(\vect{1}_r+\half
\vect{m}+\half \vect{m}'\right)$  and define  
\begin{equation}\label{def:d}
\nu_\Omega\coloneqq  \Delta_\Omega^{\vect{d}}\cdot \Hc^m \qquad\text{and} \qquad \nu_{\Omega'} \coloneqq \Delta^{\vect{d}}_{\Omega'}\cdot \Hc^m ,
\end{equation}
where $\Hc^m$ denotes the $m$-dimensional Hausdorff measure on
$\Omega$ and $\Omega'$, respectively.  
\end{deff}

We point out  that  $\nu_\Omega$ and $\nu_{\Omega'}$ are two
$G(\Omega)$-invariant measures on $\Omega$ and $\Omega'$,
respectively:   cf.~\cite[Lemma~2.18]{CalziPeloso} for
$T_+$-invariance, and argue as in the proof of~\cite[Proposition
I.3.1]{FarautKoranyi} for $G(\Omega)$-invariance.

\begin{deff}
We denote by $(I^{\vect s}_\Omega)_{\vs\in \C^r}$ the unique holomorphic family of tempered distributions on $F$ such that $I^{\vect s}_\Omega= \frac{1}{\Gamma_\Omega(\vs)}\Delta^{\vect s}_\Omega \cdot \nu_\Omega$ for $\Rea\vs\in \frac 1 2 \vect m+(\R_+^*)^r$ (cf.~\cite[Lemma 2.26, Definition 2.27, and Proposition 2.28]{CalziPeloso}). We call `Riemann--Liouville operators' the operators of convolution by the $I^{\vs}_\Omega$.
\end{deff}

\subsection{Lattices on $D$}

We shall make extensive use of the notion of  a  lattice on the domain
$D$.
We first introduce some Riemannian metrics on $\Omega, \Omega'$, and $D$.

\begin{deff}
We endow $D$ with the Bergman metric, that is, with the (complete K\"ahler) metric $\vect k$ defined by
\[
\vect k_{(\zeta,z)} v w\coloneqq \partial_v \overline{\partial_w} \log (\Delta^{\vect b+2 \vect d}\circ \rho)(\zeta,z)
\]
for every $(\zeta,z)\in D$ and for every $v,w\in E\times F_\C$ (cf.~\cite[\S 2.5]{CalziPeloso}). We denote by $B((\zeta,z),R)$ the corresponding open ball of centre $(\zeta,z)$ and radius $R$. 
We denote by $\nu_D$ the corresponding invariant measure $(\Delta^{\vect b+2\vect d}\circ \rho) \cdot \Hc^{2 n+2m}$ on $D$ (cf., e.g.,~\cite[Proposition 2.44]{CalziPeloso}).

We endow $\Omega$ with the quotient metric induced by the submersion $\rho\colon D\to \Omega$, and $\Omega'$ with the Riemannian metric induced by the correspondence $\Omega\ni t\cdot e_\Omega \mapsto e_{\Omega'}\cdot t\in \Omega'$ (cf.~\cite[Definition 2.45 and Lemma 2.46]{CalziPeloso}). We denote by $B_\Omega(h,R)$ and $B_{\Omega'}(\lambda,R)$ the corresponding open balls of centre $h$ and $\lambda$, respectively, and radius $R$. 
\end{deff}

\begin{deff}
Given $\delta>0$ and   $R>1$,
we say that a family $(h_{k})_{k\in K}$ of elements of $\Omega$  is a
$(\delta,R)$-lattice if the following hold:
\begin{itemize}
\itb the balls $B_\Omega(h_k,\delta)$ are pairwise disjoint;
\itb the balls  $B_\Omega(h_k,R\delta)$ cover $D$.
\end{itemize}
We define $(\delta,R)$-lattices on $\Omega'$ in an analogous fashion.  

Furthermore, we say that a family
$(\zeta_{j,k},z_{j,k})_{j\in J, k\in  K}$ of elements of $D$ is a $(\delta,R)$-lattice if the following
hold: 
\begin{itemize}
\itb  there is a $(\delta,R)$-lattice $(h_k)_{k\in K}$   on $\Omega$ such that
$h_k \coloneqq \rho(\zeta_{j,k},z_{j,k})$  for every  $j\in J$ and every $k\in K$; 
\itb the balls $B((\zeta_{j,k},z_{j,k}),\delta)$ are pairwise disjoint;
\itb the balls  $B((\zeta_{j,k},z_{j,k}),R\delta)$ cover $D$.
\end{itemize}
\end{deff}

The following  result  is~\cite[Lemma~2.55]{CalziPeloso}, and  guarantees the
existence of lattices as above on $D$.
\begin{lem}\label{lem:32}
Take $\delta>0$. Then, there is a $(\delta,4)$-lattice on $D$. 
\end{lem}

\subsection{Fourier transform on $\Nc$}

We now recall the definition and a few properties of the Fourier
transform on $\Nc$, cf.~\cite[Section 1.2]{CalziPeloso} for  further  details.

Define $W\coloneqq\Set{\lambda\in F'\colon \exists \zeta\neq 0 \ \text{such
that}\ 
\langle \lambda, \Ima \Phi(\zeta, \,\cdot\,)\rangle=0 }$.
Observe that, for every $\lambda\in F'\setminus W$, the quotient
$\Nc/\ker \lambda$ is a $(2 n+1)$-dimensional Heisenberg group, so
that the Stone--von Neumann theorem (cf., e.g.,~\cite[Theorem
1.50]{Folland}) implies that there is up to unitary equivalence)  a unique  irreducible unitary
representation $\pi_\lambda$ of $\Nc$ in some hilbertian space
$H_\lambda$ such that $\pi_\lambda(0,i
x)=\ee^{-i\langle \lambda,x\rangle}$ for every $x\in F$. Using the
Plancherel formula on the quotients $\Nc/\ker \lambda$ and integrating
in $\lambda$, one may then find the Plancherel formula for $\Nc$
(cf.,e.g.,~\cite[Section 2]{AstengoCowlingDiBlasioSundari}). Namely,
there is a constant $c>0$ such that 
\[
\norm{f}^2_{L^2(\Nc)}=c \int_{F'\setminus W} \norm{\pi_\lambda(f)}_{\Lin^2(H_\lambda)}^2 \Delta_{\Omega'}^{-\vect b}(\lambda)\,\dd \lambda
\] 
for every $f\in L^1(\Nc)\cap L^2(\Nc)$ (cf.~\cite[Corollary 1.17 and Proposition 2.30]{CalziPeloso}), where $\Lin^2(H_\lambda)$ denotes the space of Hilbert--Schmidt endomorphisms of $H_\lambda$. 

Notice, though, that $\pi_\lambda(f_h)=0$ for almost every $\lambda\in F'\setminus (W\cup \Omega')$, for every $h\in \Omega$, and for every $f$ in the space $ A^{p,q}_{\vs}$ to be defined below, $p\in ]0,2]$ (cf.~\cite[Corollares 1.37 and 3.3, and Proposition 3.2]{CalziPeloso}). For this reason, we shall only describe $\pi_\lambda$ for $\lambda\in \Omega'$ (`Bargmann representation'). We define $H_\lambda\coloneqq \Hol(E)\cap L^2(\nu_\lambda)$, where $\nu_\lambda= \ee^{-2\langle\lambda, \Phi(\,\cdot\,)\rangle}\cdot \Hc^{2 n}$, and  
\begin{equation}\label{eq:pi-lambda}
\pi_\lambda(\zeta,x)  \psi  (\omega)
\coloneqq \ee^{\langle \lambda_\C, -i x +2 \Phi(\omega, \zeta)-\Phi(\zeta)\rangle}  \psi (\omega- \zeta),
\end{equation} 
for every $\psi \in H_\lambda$, for every $\omega\in E$, and for every $(\zeta,x)\in \Nc$.
Denote by $P_{\lambda,0}$ the self-adjoint projector of $H_\lambda$
onto the space of constant functions, for every $\lambda\in \Omega'$,
and define $P_{\lambda,0}=0$ for $\lambda\in F'\setminus (W\cup
\Omega')$. The introduction of $P_{\lambda,0}$ is justified by the
observation that $\pi_\lambda(f_h)=\pi_\lambda(f_h)P_{\lambda,0}$ for
almost every $\lambda\in F'\setminus W$, for every $h\in \Omega$, and
for every $f$ in the space $ A^{p,q}_{\vs}$ to be defined below, $p\in
]0,2]$ (cf.~\cite[Corollares 1.37 and 3.3, and Proposition
3.2]{CalziPeloso}).

\section{Function spaces}\label{sec:fnct-sp}
In this section we introduce the functions spaces on $D$ and $\Nc$ that we  shall consider  in this paper.  We refer
the reader to~\cite{CalziPeloso} for  further  details and references.

\begin{deff}\label{def:Lpqs}
Take $\vs\in \R^r$ and $p,q\in ]0,\infty]$. 
We define 
\[
L^{p,q}_\vs(D)\coloneqq\Set{ f\colon D\to \C\colon  f \text{ is measurable}, \int_\Omega  \Big(    \Delta_\Omega^{\vs}(h) \norm{f_h}_{L^p(\Nc)}  \Big) ^q\,\dd  \nu_\Omega(h)<\infty  }  
\]
(modification when $q=\infty$), and  $L^{p,q}_{\vs,0}(D)$ as the
closure of $C_c( D)$ in $L^{p,q}_{\vs}(D)$.   Constistently with~\eqref{eq:1}, we define
\[
A^{p,q}_\vs(D)= L^{p,q}_\vs(D)\cap \Hol(D), \qquad\text{and}\qquad
A^{p,q}_{\vs,0}(D)
= L^{p,q}_{\vs,0}(D) \cap \Hol(D).
\]
\end{deff}
Notice that $L^{p,q}_{\vs,0}(D)=L^{p,q}_{\vs}(D)$  and $A^{p,q}_{\vs,0}(D)=A^{p,q}_{\vs}(D)$
if (and only
if) $p,q<\infty$.

It is possible to characterize the values of $p,q, \vs$ for which
$A^{p,q}_{\vs}(D)$ is non-trivial. The following is~\cite[Proposition~3.5]{CalziPeloso}.
\begin{prop}
Take $\vs\in \R^r$ and $p,q\in ]0,\infty]$. Then,   $A^{p,q}_{\vs,0}(D)\neq \Set{0}$ (resp.\   $A^{p,q}_{\vs}(D)\neq \Set{0}$) 
if and only if  $\vs\in \frac{1}{2 q}\vect m+(\R_+^*)^r$ (resp.  $\vs\in \R_+^r$ if $q=\infty$). 
\end{prop}

\begin{deff}
For $\vs\in\C^r$ we define the kernel function $B^{\vs}$ as
\[
B_{(\zeta',z')}^{\vs} (\zeta,z)\coloneqq \Delta_\Omega^{\vs}\left( \frac{z-\overline{z'}}{2i}
-\Phi(\zeta,\zeta') \right)
\]
for every $((\zeta,z),(\zeta',z'))\in (D\times \overline D)\cup (\overline D \times D)$. 
\end{deff}

\begin{oss}
It turns out that, when $\vs \in \frac14 {\vect m}+ (\R^*_+)^r$, 
the
reproducing kernel of $A^{2,2}_\vs$ is  given by 
\[
K_\vs ((\zeta,z),(\zeta',z'))
\coloneqq c\frac{\Gamma_{\Omega'}(2 \vs-\vb-\vd)}{\Gamma_\Omega(2\vs)}
B_{(\zeta',z')}^{\vd+\vb-2\vs} (\zeta,z)
\]
for a suitable constant $c>0$, cf.~\cite[Remark 3.12]{CalziPeloso}.

In addition, we can  also  describe the {\em Cauchy--Szeg\H o} kernel,
that is, the reproducing kernel of the Hardy space
$A^{2,\infty}_{\vect 0}$. Indeed, for every $(\zeta,z)\in D$ and for every
$(\zeta',x')\in \Nc$, set 
\[
S_{(\zeta,z)}(\zeta',x') \coloneqq
c' \big( B^{\vb+\vd}_{(\zeta,z)}\big)_0(\zeta',x')
\]
for a suitable constant $c'>0$, cf.~\cite[Lemma 5.1]{CalziPeloso}.  Then,  for every $f\in A^{2,\infty}_{\vect 0}$, 
\[
f(\zeta,z)=\langle f_0| S_{(\zeta,z)}\rangle
\]
for every $(\zeta,z)\in D$, where $f_0\coloneqq\lim\limits_{h\to 0} f_h$ in $L^2(\Nc)$. 
\end{oss}

For $p,q\in ]0,\infty]$, and two sets $J$ and $K$, define
\[
\ell^{p,q}(J,K)\coloneqq \Set{\lambda\in \C^{J\times K}\colon ((\lambda_{j,k})_{j\in J})_{k\in K}\in \ell^q(K;\ell^p(J))},
\]
endowed with the corresponding quasi-norm, and define
$\ell^{p,q}_0(J,K)$ as the closure of $\C^{(J\times K)}$ in
$\ell^{p,q}(J,K)$.

\begin{deff}\label{def:L-property}
We say that property $\atomic^{p,q}_{\vs,\vect{s'},0}$ (resp.\
$\atomic^{p,q}_{\vs,\vect{s'}}$) holds if for every $\delta_0>0$
there is a $(\delta,4)$-lattice $(\zeta_{j,k},z_{j,k})_{j\in J, k\in K}$, with $\delta\in ]0,\delta_0]$,  such that, defining
$h_k\coloneqq \rho(\zeta_{j,k},z_{j,k})$ for every $k\in K$ and
for some (hence every) $j\in J$, the mapping 
\[
\Psi\colon\lambda \mapsto \sum_{j,k} \lambda_{j,k}
B_{(\zeta_{j,k},z_{j,k})}^{\vect{s'}} \Delta_\Omega^{(\vb+\vect
d)/p-\vs-\vect{s'}}(h_k) 
\]
is well defined (with locally uniform convergence of the sum) and maps
$\ell^{p,q}_0(J,K) $ into $A^{p,q}_{\vs,0}(D)$ continuously
(resp.\ maps $\ell^{p,q}(J,K) $ into $A^{p,q}_{\vs}(D)$
continuously). 

If we may take $(\zeta_{j,k},z_{j,k})_{j\in J, k\in K}$, for every
$\delta_0>0$ as above, in such a way that the corresponding mapping
$\Psi$ is onto, then we say that property $\atomics^{p,q}_{\vect
s,\vect{s'},0}$ (resp.\ $\atomics^{p,q}_{\vs,\vect{s'}}$) holds.  
\end{deff}

We shall present later (cf.~Proposition~\ref{prop:25})  some sufficient conditions for property
$\atomics$ to hold. See also~\cite{CalziPeloso} for a more thorough
discussion and some necessary conditions.

In this paper we  are  also interested in a family of spaces of holomorphic
functions on $D$,  denoted by $\widetilde A^{p,q}_{\vs}(D)$,
which is  defined in connection with  the boundary values  of  the elements of 
$A^{p,q}_\vs(D)$. We begin by introducing some 
Besov-type spaces defined on the \v{S}ilov boundary  $bD$, that we
identify with $\Nc$.

We define a space of test functions on $\Nc$ by setting
\[
\Sc_\Omega(\Nc)\coloneqq\Set{ \psi  \in \Sc(\Nc)\colon \exists \varphi\in C^\infty_c(\Omega')\;\forall\lambda\in F'\setminus W\quad \pi_\lambda(\psi)= \varphi(\lambda)P_{\lambda,0} } .
\]
Then, define $\Sc_{\Omega,L}(\Nc)\coloneqq
\Sc(\Nc)*\Sc_{\Omega}(\Nc)$, endowed with the inductive limit of the
topologies induced by $\Sc(\Nc)$ on its subspaces $\Sc(\Nc)*\psi$,
$\psi\in \Sc_{\Omega}(\Nc)$.\footnote{It is not hard to see that this  definition is equivalent to~\cite[Definition 4.4]{CalziPeloso}.}
Denote by  $\Sc_{\Omega,L}'(\Nc)$ the dual of $\Sc_{\Omega,L}(\Nc)$.
See~\cite[Propositions 4.2 and 4.5, and Lemma~4.14]{CalziPeloso} for a proof of the following result. 
\begin{prop}\label{prop:4.2}
The following hold:
\begin{enumerate}
\item[\em(1)] the mapping $\Fc_\Nc\colon \varphi\mapsto [\lambda
  \mapsto \tr(\pi_\lambda(\varphi) ) ] $ induces an isomorphism of
  $\Sc_\Omega(\Nc)$ onto $C^\infty_c(\Omega')$; 

\item[\em(2)]  given two $(\delta,R)$-lattices $(\lambda_k)_{k\in K}$ and $(\lambda'_{k'})_{k'\in K'}$ on $\Omega'$, and two families $(\psi_k)_{k\in K}, (\psi'_{k'})_{k'\in K'}$ of elements of $\Sc_\Omega(\Nc)$ such that $((\Fc_\Nc \psi_k)(\,\cdot\, t_k))$ and $((\Fc_\Nc \psi'_{k'})(\,\cdot\, t'_{k'}))$ are bounded families of positive elements of $C^\infty_c(\Omega')$, where $t_k,t'_{k'}\in T_+$ are such that $\lambda_k=e_{\Omega'}\cdot t_k$ and $\lambda'_{k'}=e_{\Omega'}\cdot t'_{k'}$, and
\[
\sum_k \Fc_\Nc\psi_k, \sum_{k'} \Fc_\Nc \psi_{k'}\Meg 1
\]
on $\Omega'$,  one has  
\[
\Big\| \Delta_{\Omega'}^{\vs}(\lambda_{k'}) \big\|u*\psi'_{k'}
\big\|_{L^p(\Nc)}  \Big\|_{\ell^q(K')}
\approx\Big\| \Delta_{\Omega'}^{\vs}(\lambda_k) \big\| u*\psi_k
\big\|_{L^p(\Nc)}  \Big\|_{\ell^q(K)},
\]
for every $u\in \Sc_{\Omega,L}'(\Nc)$.  
\end{enumerate} 
\end{prop}

\begin{deff}
Let $\vs\in \R^r$, $p,q\in ]0,\infty]$. Take
$(\lambda_k)_{k\in K}$ and $(\psi_k)$ as in
Proposition~\ref{prop:4.2}.
Then, we define $B^{\vs}_{p,q}(\Nc,\Omega)$ as the space of $u\in \Sc_{\Omega,L}'(\Nc)$ such that
\[
(\Delta^{\vs}_{\Omega'}(\lambda_k) (u*\psi_k))_k\in  \ell^q (K;L^p(\Nc)),
\]
endowed with the corresponding topology. We denote by  $\mathring{B}^{\vs}_{p,q}(\Nc,\Omega)$ the closure of $\Sc_{\Omega,L}(\Nc)$ in $ B^{\vs}_{p,q}(\Nc,\Omega)$, which is  the space of $u\in
\Sc_{\Omega,L}'(\Nc)$ such that  
\[
(\Delta^{\vs}_{\Omega'}(\lambda_k) (u*\psi_k))_k\in \ell^q_0(K;L^p_0(\Nc)).
\]
\end{deff}

We are now able  to define  an extension operator  from the Besov-type spaces $B^{-\vs}_{p,q}(\Nc,\Omega)$ on the \v{S}ilov
boundary $\Nc$ to $D$ to the function  spaces $\widetilde
A^{p,q}_\vs(D)$ and $\widetilde A^{p,q}_{\vs,0}(D)$ on $D$. 

\begin{deff}\label{def:tilde-spaces}
Take $p,q\in]0,\infty]$ and  $\vs\in
\frac{1}{p}(\vb+\vd)+\frac{1}{2 q'}\vect{m'}+(\R_+^*)^r$, so that, 
for every $(\zeta,z)\in D$, $S_{(\zeta,z)}\in \mathring
B^{\vs-(1/p-1)_+(\vect b+\vect d)}_{p',q'}(\Nc,\Omega)$
(cf.~\cite[Lemma 5.1]{CalziPeloso}).  
Then, define,  for every  $u\in B^{-\vs}_{p,q}(\Nc,\Omega)$ and for every $(\zeta,z)\in D$,
\[
\Ec u  (\zeta,z) \coloneqq\langle u| S_{(\zeta,z)}\rangle,
\]
so that $\Ec$ maps $B^{-\vs}_{p,q}(\Nc,\Omega)$ into $A^{\infty,\infty}_{\vs-(\vect b+\vect d)/p}(D)$ continuously (cf.~\cite[Theorem 5.2]{CalziPeloso}).
We define 
\[
\widetilde A^{p,q}_\vs(D) \coloneqq\Ec(B^{-\vs}_{p,q}(\Nc,\Omega))\qquad
\text{and} \qquad\widetilde A^{p,q}_{\vs,0}(D)\coloneqq \Ec(\mathring B^{-\vs}_{p,q}(\Nc,\Omega)),
\]
and endow both spaces with the corresponding (direct image) topology.

In addition, we define the {\em boundary value operator}
$\Bc \colon A^{2,\infty}_{\vect 0}(D)\to B^{\vect 0}_{2,2} (\Nc,\Omega) $ 
\[
\Bc\colon A^{2,\infty}_{\vect 0}(D)\ni f \mapsto \lim_{h\to 0} f_h\in L^2(\Nc)
\]
\end{deff}
We remark that,  by~\cite[Corollary 1.37]{CalziPeloso}, $\Bc$  is
well defined and continuous, and that $\Ec \Bc=I$. 

It is possible to describe the boundary values of functions in
$A^{p,q}_\vs(D)$, that is, the the limits of $f_h$, as $h\to 0,h\in \Omega$, for $f\in
A^{p,q}_\vs(D)$, as elements  of  $B^{-\vs}_{p,q}(\Nc,\Omega)$.
Conversely, under  suitable  assumptions, every element of
$B^{-\vs}_{p,q}(\Nc,\Omega)$ is the boundary value of a unique element
of $A^{p,q}_{\vs}(D)$.  Precisely, the following holds,
cf.~\cite[Theorem 5.2, Proposition 5.4 and its proof, and Corollary 5.11]{CalziPeloso}.

\begin{prop}\label{prop:23bis}
Take $p,q\in ]0,\infty]$, and $\vs\in \sup\Big( \frac{1}{2 q}\vect m,
\frac{1}{p}(\vb+\vd)+\frac{1}{2 q'}\vect{m'}\Big)+(\R_+^*)^r $. 
Then, the following hold:
\begin{enumerate}
\item[\em(1)] $(\Ec u)_h$ converges to $u$ in
$B^{-\vs}_{p,q}(\Nc,\Omega)$ (resp.\ in $\Sc'_{\Omega,L}(\Nc)$)
for everu $u\in \mathring B^{-\vs}_{p,q}(\Nc,\Omega)$ (resp.\ for
every $u\in B^{-\vs}_{p,q}(\Nc,\Omega)$);  

\item[\em(2)] the operator $\Bc$ induces a continuous linear mapping
\[
\Bc\colon A^{p,q}_{\vs}(D)\to B^{-\vs}_{p,q}(\Nc,\Omega)
\qquad \text{(resp.\ } \Bc\colon A^{p,q}_{\vs,0}(D)\to \mathring{B}^{-\vs}_{p,q}(\Nc,\Omega))
\]
such that $\Ec \Bc=I$;

\item[\em(3)] there are continuous inclusions
\[
\Ec(\Sc_{\Omega,L}(\Nc))\subseteq A^{p,q}_{\vs}(D) \subseteq
\widetilde A^{p,q}_{\vs}(D)
\qquad
(\text{resp.\ } \Ec(\Sc_{\Omega,L}(\Nc))\subseteq A^{p,q}_{\vs,0}(D) \subseteq \widetilde A^{p,q}_{\vs,0}(D))  ;
\]

\item[\em(4)] if, further,
$\vs\in\sup\Big(\frac{1}{2 q}\vect m+\big( \frac{1}{2 \min(p,p')}-\frac{1}{2q} \big)_+\vect{m'} , \frac{1}{p}(\vb+\vd)+\frac{1}{2 q'}\vect{m'}\Big) +(\R_+^*)^r,
$
then,
\[
A^{p,q}_{\vs}(D)=\widetilde A^{p,q}_{\vs}(D)  \qquad\text{and} \qquad
A^{p,q}_{\vs,0}(D)=\widetilde A^{p,q}_{\vs,0}(D).
\]
\end{enumerate}
\end{prop}

We now present some sufficient conditions for property $\atomics$. See~\cite[Corollary 5.14]{CalziPeloso} for a proof of the following result.

\begin{prop}\label{prop:25}
Take $p,q\in ]0,\infty]$, $\vs\in \sup\Big( \frac{1}{2 q}\vect m,
\frac{1}{p}(\vb+\vd)+\frac{1}{2 q'}\vect{m'}\Big)+(\R_+^*)^r $, and
$\vs'\in \frac{1}{\min(1,p)}(\vect b+\vect d)-\frac{1}{2 q}\vect{m'} -
\left(\frac{1}{2 \min(1,p)}-\frac{1}{2q}  \right)_+\vect m -
\vs-(\R_+^*)^r$. If $A^{p,q}_{\vect s,0}(D)=\widetilde A^{p,q}_{\vs,0}(D)$
(resp.\ $A^{p,q}_{\vect s}(D)=\widetilde A^{p,q}_{\vs}(D)$), then property
$\atomics^{p,q}_{\vs, \vs',0}$ (resp.\ $\atomics^{p,q}_{\vs,\vs'}$)
holds. 
\end{prop}

\medskip

We conclude this section with a  technical lemma that will be very useful later on. 
\begin{lem}\label{lem:5}
Take $\vs\in \R^r$, $p,q\in ]0,\infty]$, $t\in T_+$, and $g\in
GL(E)$ such that $t\cdot \Phi=\Phi\circ (g\times g)$. Then, 
\[
\norm{f\circ (g \times t)}_{L^{p,q}_{\vs}(D)}=\Delta_\Omega^{(\vb+\vd)/p-\vs}(t)\norm{f}_{L^{p,q}_{\vs}(D)}
\]
for every $\nu_D$-measurable function $f$ on $D$.
\end{lem}

\begin{proof}
Observe that 
\[
(f\circ (g\times t))_h(\zeta,x)= f_{t\cdot h}(g\zeta,t\cdot x)
\]
for every $(\zeta,x)\in D$ and for every $h\in\Omega$,  so that
\[
\norm{(f\circ (g\times t))_h}_{L^p(\Nc)}= \Delta^{(\vb+\vd)/p}(t) \norm{f_{t\cdot h}}_{L^p(\Nc)}
\]
for every $h\in \Omega$, thanks to~\cite[Lemmas 2.9 and 2.18]{CalziPeloso}.
Then,
\[
\norm{f\circ (g\times t)}_{L^{p,q}_{\vs}(D)}=  \Delta^{(\vb+\vd)/p-\vs}(t)\norm{f}_{L^{p,q}_{\vs}(D)},
\]
whence the result.
\end{proof}

\section{Carleson Measures}\label{sec:Carleson}

In this section we prove our main results concerning Carleson measures
for $A^{p,q}_\vs$.
For every  $\mi\in\cM_+(D)$  and $R>0$, we set
\begin{equation}\label{eq:MR-mi}
M_R(\mi)\colon D\ni (\zeta,z)\mapsto \mi(B((\zeta,z),R))\in [0,\infty[.
\end{equation}

\begin{lem}\label{lem:1}
Take $\vs\in \R^r$, $p,q\in ]0,\infty]$, and $\mi\in\cM_+(D)$. Then, the following conditions are equivalent:
\begin{enumerate}
\item[{\em(1)}]  there is $R>0$ such that $M_R(\mi)\in L^{p,q}_{\vs}(D)$;

\item[{\em(2)}]  $M_R(\mi)\in L^{p,q}_{\vs}(D)$ for every $R>0$;

\item[{\em(3)}] there is a $(\delta,R)$-lattice $(\zeta_{j,k},z_{j,k})_{j\in J,k\in K}$ on $D$, with $\delta>0$ and $R>1$, such that 
\[
\Big(\Delta_\Omega^{\vs-(\vb+\vect
d)/p}(\rho(\zeta_{j,k},z_{j,k}))M_{R\delta}(\mi)(\zeta_{j,k},z_{j,k})
\Big)\in\ell^{p,q}(J,K);
\]

\item[{\em(4)}] for every $(\delta,R)$-lattice $(\zeta_{j,k},z_{j,k})_{j\in J,k\in K}$ on $D$, with $\delta>0$ and $R>1$, 
\[
\big(\Delta_\Omega^{\vs-(\vb+\vd)/p}(\rho(\zeta_{j,k},z_{j,k}))M_{R\delta}(\mi)(\zeta_{j,k},z_{j,k}) \big)\in\ell^{p,q}(J,K).
\]
\end{enumerate}
The same holds if one replaces $L^{p,q}_{\vs}(D)$ and
$\ell^{p,q}(J,K)$ with $L^{p,q}_{\vs,0}(D)$ and
$\ell^{p,q}_0(J,K)$, respectively. 
\end{lem}

This extends~\cite[Lemmas 2.9 and 2.12]{NanaSehba}, where the case in
which $p=q\in [1,\infty]$, $D$ is an irreducible symmetric tube
domain, and $\vect{s}\in \R\vect 1_r$, is considered.\footnote{Notice,
though, that in~\cite[Lemmas 2.9 and 2.12]{NanaSehba} the use of the
`pure-norm' spaces $L^p$ and $\ell^p$ allows for the use of more
general lattices. } 

\begin{proof}
We shall only prove the first assertion. The second assertion is clear
if $\mi$ has compact support, and then follows by approximation, since
the equivalence of conditions (1)--(4) is quantitative by the closed
graph theorem, or by the proof below. 

(1) $\implies$ (2). Take $R'>0$. Observe that, since $\overline B((0,i
e_\Omega),R')$ is compact by~\cite[Proposition 2.44]{CalziPeloso},
there is a finite family $(\zeta_j,z_j)_{j\in J}$ of elements of $D$
such that $B((0,i e_\Omega),R')\subseteq \bigcup_{j\in
J}B((\zeta_j,z_j),R)$.  
For every $j\in J$, choose $t_j\in T_+$ and $g_j\in GL(E)$ in such a
way that $t_j\cdot \Phi=\Phi\circ (g_j\times g_j)$ and $t_j\cdot
e_\Omega= \rho(\zeta_j,z_j)$. If we define  
\[
\phi_j\colon (\zeta,z)\mapsto (\zeta_j,\Rea z_j+i\Phi(\zeta_j))\cdot (g_j \zeta,t_j\cdot z),
\]
then $\phi_j$ is an affine automorphism of $D$ and
$B((\zeta,z),R')\subseteq \bigcup_{j\in J} B(\phi_j(\zeta,z),R)$ for
every $(\zeta,z)\in D$. Therefore, in order to prove that
$M_{R'}(\mi)\in L^{p,q}_{\vs}(D)$, it will suffice to prove that
$M_R(\mi)\circ \phi_j\in L^{p,q}_{\vs}(D)$ for every $j\in J$.
Since this follows easily from Lemma~\ref{lem:5}, this proves (2). 

(2) $\implies$ (4). Let $(\zeta_{j,k},z_{j,k})_{j\in J,k\in K}$ be a
$(\delta,R)$-lattice on $D$ for some $\delta>0$ and $R>1$. Observe
that 
\[
\sum_{j,k} M_{R\delta}(\zeta_{j,k},z_{j,k}) \chi_{B((\zeta_{j,k},z_{j,k}),\delta)}\meg M_{(R+1)\delta}(\mi)(\zeta_{j,k},z_{j,k})
\]
on $D$, so that 
\[
\sum_{j,k}  M_{R\delta}(\zeta_{j,k},z_{j,k})\chi_{B((\zeta_{j,k},z_{j,k}),\delta)}\in L^{p,q}_{\vs}(D).
\]
Now, observe that there is a constant $C_1>0$ such that
\[
\norm{(\chi_{B((0,i e_\Omega),\delta)})_h}_{L^p(\Nc)}\Meg C_1
\]
for every $h\in B_\Omega(e_\Omega,\delta/2)$,\footnote{Observe that,
since $B((0,i e_\Omega),\delta)$ is an open set, the mapping
$h\mapsto \norm{(\chi_{B((0,i e_\Omega),\delta)})_h}_{L^p(\Nc)}$ is
lower semi-continuous, and that it vanishes nowhere on the compact
set $\overline B_\Omega(e_\Omega,\delta/2)$, thanks to~\cite[Lemma
2.46]{CalziPeloso}.} so that, by homogeneity, 
\[
\norm{(\chi_{B((\zeta_{j,k},z_{j,k}),\delta)})_h}_{L^p(\Nc)}\Meg C_1\Delta_\Omega^{-(\vb+\vd)/p}(h_k)
\]
for every $h\in B_\Omega(h_k,\delta/2)$ and for every $(j,k)\in J\times K$, where $h_k\coloneqq\rho(\zeta_{j,k},z_{j,k})$ for every $(j,k)\in J\times K$.
It then follows that
\[
\begin{split}
&\norm*{\sum_{j,k} \left( \chi_{B((\zeta_{j,k},z_{j,k}),\delta)}\right)_h M_{R\delta}(\zeta_{j,k},z_{j,k})}_{L^p(\Nc)}\\
&\qquad\Meg C_1 \norm*{ \big(\Delta_\Omega^{-(\vb+\vd)/p}(h_k)M_{R\delta}(\zeta_{j,k},z_{j,k}) \chi_{B_\Omega(h_k,\delta/2)}(h)\big)_{j,k}}_{\ell^p(J\times K)}
\end{split}
\]
for every $h\in\Omega$, whence
\[
\left(\Delta_\Omega^{\vs-(\vb+\vd)/p}(h_k)M_{R\delta}(\zeta_{j,k},z_{j,k})\right)\in \ell^{p,q}(J,K)
\]
thanks to~\cite[Corollary 2.49]{CalziPeloso}.

(4) $\implies$ (3). Obvious by~\cite[Lemma 2.55]{CalziPeloso}.

(3) $\implies$ (1). Take a $(\delta,R)$-lattice
$(\zeta_{j,k},z_{j,k})_{j\in J, k\in K}$ as in the statement, define
$h_k\coloneqq \rho(\zeta_{j,k},z_{j,k})$ for every $(j,k)\in J\times
K$, and let us prove that 
\[
\left(\Delta_\Omega^{\vs-(\vb+\vd)/p}(h_k)M_{R' \delta}(\zeta_{j,k},z_{j,k}) \right)\in\ell^{p,q}(J,K)
\] 
for every $R'\Meg R$. Indeed, for every $(j,k)\in J\times K$, define
$V_{j,k}$ as the set of $(j',k')\in J\times K$ such that
$d((\zeta_{j,k},z_{j,k}),(\zeta_{j',k'},z_{j',k'}))<(R+R')\delta $,
and observe that $N\coloneqq \sup\limits_{j,k} \card(V_{j,k})$ is finite
by~\cite[Proposition 2.56]{CalziPeloso}. In addition,
by~\cite[Corollary 2.49]{CalziPeloso} we may find a constant $C_2>0$
such that 
\[
\frac{1}{C_2}\Delta_\Omega^{\vs-(\vb+\vd)/p}(h_{k'}) \meg
\Delta_\Omega^{\vs-(\vb+\vd)/p}(h_k)\meg
C_2\Delta_\Omega^{\vs-(\vb+\vd)/p}(h_{k'}) 
\] 
for every $(j,k)\in J\times K$ and for every $(j',k')\in V_{j,k}$. Then,
\[
\Delta_\Omega^{\vs-(\vb+\vd)/p}(h_k)M_{R'\delta}(\zeta_{j,k},z_{j,k})\meg
C_2 \sum_{(j',k')\in V_{j,k}}
\Delta_\Omega^{\vs-(\vb+\vd)/p}(h_{k'})M_{R\delta}(\mi)(\zeta_{j',k'},z_{j',k'}) 
\]
for every $(j,k)\in J\times K$. Define $V'_k$ as the set of $k'\in K$
such that $d(h_k,h_{k'})\meg (R+R')\delta$, so that $k'\in V'_k$ for
every $(j',k')\in V_{j,k}$ and for every $j\in J$ by~\cite[Lemma
2.46]{CalziPeloso}. In addition, $N'\coloneqq  \sup\limits_{k} \card(V'_{k})$
is finite by~\cite[Proposition 2.56]{CalziPeloso}. 
Therefore,
\[
\begin{split}
&\norm*{\big(\Delta_\Omega^{\vs-(\vb+\vd)/p}(h_k)M_{R'\delta}(\zeta_{j,k},z_{j,k})\big)_j}_{\ell^p(J)}\\
&\qquad\qquad\qquad\qquad
\meg C_2 N^{\max(1,1/p)} \norm*{\big(\Delta_\Omega^{\vs-(\vb+\vd)/p}(h_{k'})M_{R\delta}(\zeta_{j',k'},z_{j',k'})\big)_{j',k'}}_{\ell^p(J\times V'_k)}
\end{split}
\]
for every $k\in K$, so that
\[
\begin{split}
&\norm*{\left( \Delta_\Omega^{\vs-(\vb+\vd)/p}(h_k)M_{R'\delta}(\zeta_{j,k},z_{j,k})\right) _{j,k}}_{\ell^{p,q}(J,K)}\\
&\qquad\qquad\qquad\meg C_2 N^{\max(1,1/p)} N'^{\max(1/p,1/q)}
\norm*{\left( \Delta_\Omega^{\vs-(\vb+\vect
d)/p}(h_{k'})M_{R\delta}(\mi)(\zeta_{j',k'},z_{j',k'})\right)
_{j',k'}}_{\ell^{p,q}(J,K)}. 
\end{split}
\]
whence our claim.

Now, let $(B_{j,k})_{j,k}$ be a Borel partition of $D$ such that $B_{j,k}\subseteq B((\zeta_{j,k},z_{j,k}),R\delta)$ for every $(j,k)\in J\times K$, and observe that
\[
M_\delta(\mi)\meg \sum_{j,k} \chi_{B_{j,k}} M_{(R+1)\delta}(\zeta_{j,k},z_{j,k})
\] 
on $D$. In addition, arguing as in the proof of~\cite[Theorem
3.23]{CalziPeloso}, we see that there is a constant $C_3>0$ such that 
\[
\norm{(\chi_{B((\zeta,z),R\delta)})_h}_{L^p(\Nc)}\meg C_3
\Delta^{-(\vb+\vd)/p}_\Omega(\rho(\zeta,z))
\chi_{B_\Omega(\rho(\zeta,z),R\delta)}(h) 
\]
for every $(\zeta,z)\in D$ and for every $h\in \Omega$.
Hence,
\[
\norm*{M_\delta(\mi)_h}_{L^p(\Nc)}\meg C_3\norm*{\big(
\chi_{B(h_k,R\delta)}(h) \Delta^{-(\vb+\vect
d)/p}_\Omega(h_k)M_{(R+1)\delta}(\zeta_{j,k},z_{j,k})
\big)_{j,k}}_{\ell^p(J\times K)}. 
\]
Now, by~\cite[Corollary 2.49 and Proposition 2.56]{CalziPeloso} there
are two constants $C_4>0$ and $N''\in \N$ such that 
\[
\frac{1}{C_4}\Delta^{\vs}_\Omega(h)\meg  \Delta^{\vs}_\Omega(h_k)\meg C_4 \Delta^{\vs}_\Omega(h)
\]
for every $h\in B_\Omega(h_k, R\delta)$ and for every $k\in K$, and
such that $\sum_k \chi_{B_\Omega(h_k,R\delta)}\meg N''
\chi_\Omega$. Then, 
\[
\norm*{M_\delta(\mi)}_{L^{p,q}_{\vs}(D)}\meg C_5
\norm*{\big(\Delta^{\vs-(\vb+\vd)/p}_\Omega(h_k)
M_{(R+1)\delta}(\zeta_{j,k},z_{j,k}) \big)_{j,k}
}_{\ell^{p,q}(J,K)}, 
\]
where $C_5\coloneqq C_3 C_4
\nu_\Omega(B_\Omega(e_\Omega,R\delta))^{1/q}
N''^{\max(1/p,1/q)}$. Thus,  (1) follows.  
\end{proof}

The next proposition provides a first sufficient condition for a
measure $\mi\in\cM_+(D)$ to  be a $p$-Carleson measure or a  compact
$p$-Carleson measure. 
It extends~\cite[Proposition 3.5 and Theorem
3.8 (i)]{NanaSehba}, where the case in which $p_1=q_1,p\in
[1,\infty[$, $D$ is an irreducible symmetric tube domain, and
$\vect{s}\in \R\vect 1_r$, is considered. 

We shall often use  the following notation. Given $p_1,q_1,p\in ]0,\infty]$, with $p<\infty$ and $\vs\in
\R^r$, we set\footnote{Recall that $t'\coloneqq\max(1,t)'$ for every $t\in ]0,\infty]$.}
\begin{equation}\label{p*q*s'}
p^*\coloneqq (p_1/p)',\quad
q^*\coloneqq(q_1/p)',\quad \text{and}\quad
\vs^*\coloneqq \max(1,p/p_1)(\vb+\vd)-p\vs.
\end{equation}

\begin{prop}\label{prop:23}
Take $p_1,q_1,p\in ]0,\infty]$, with $p<\infty$, and take $\vs\in
\R^r$ such that $\vs\in \frac{1}{2 q_1}\vect m+(\R_+^*)^r$ or
$\vs\in \R_+^r$ if $q_1=\infty$. Let $p^*,q^*$, and $\vs^*$ be as in~\eqref{p*q*s'}. 
Let $\mi\in\cM_+(D)$, and take $R>0$. 
Assume that 
\[
M_R(\mi)\in L^{p^*,q^*}_{\vs^*}(D).
\]
Then, $A^{p_1,q_1}_{\vs}(D)$  embeds continuously into $L^{p}(\mi)$.

If, in addition, 
\[
M_R(\mi)\in L^{p^*, q^*}_{\vs^*,0}(D),
\]
then $A^{p_1,q_1}_{\vs}(D)$  embeds compactly into $L^{p}(\mi)$.
\end{prop}

\begin{proof}
Let $(\zeta_{j,k},z_{j,k})_{j\in J,k\in K}$ be an $(R/4,4)$-lattice on
$D$ (cf.~\cite[Lemma 2.55]{CalziPeloso}), and define
$S_+\colon \Hol(D)\to \C^{J\times K}$ so that 
\[
(S_+ f)_{j,k}\coloneqq \Delta^{\vs-(\vb+\vect
d)/p_1}_\Omega(h_k) \max_{\overline B((\zeta_{j,k},z_{j,k}), R)}
\abs{f} 
\]
for every $(j,k)\in J\times K$, where $h_k\coloneqq
\rho(\zeta_{j,k},z_{j,k})$, so that there is a constant $C_1>0$ such
that 
\[
\frac{1}{C_1} \norm{f}_{A^{p_1,q_1}_{\vs}(D)}\meg \norm{S_+
f}_{\ell^{p_1,q_1}(J,K)}\meg C_1\norm{f}_{A^{p_1,q_1}_{\vs}(D)} 
\] 
for every $f\in \Hol(D)$, thanks to~\cite[Theorem 3.23]{CalziPeloso}.
In addition, choose a Borel partition $(B_{j,k})$ of $D$ such that
$B_{j,k}\subseteq B((\zeta_{j,k},z_{j,k}), R)$ for every $(j,k)\in
J\times K$. Then, for every $f\in A^{p_1,q_1}_{\vs}(D)$, 
\[
\begin{split}
\norm{f}_{L^{p}(\mi)}&\meg \norm*{\sum_{(j,k)\in J\times K} (S_+
f)_{j,k}  \Delta^{(\vb+\vd)/p_1-\vect
s}_\Omega(h_k)\chi_{B_{j,k}}}_{L^{p}(\mi)}\\ 
&\meg \norm*{\big((S_+ f)_{j,k}  \Delta^{(\vb+\vd)/p_1-\vect
s}_\Omega(h_k)
M_R(\mi)(\zeta_{j,k},z_{j,k})^{1/p}\big)_{j,k}}_{\ell^{p}(J\times
K)}\\ 
&\meg C_2\norm*{S_+ f}_{\ell^{p_1,q_1}(J,K)}\\ 
&\meg C_1 C_2 \norm{f}_{A^{p_1,q_1}_{\vs}(D)}
\end{split}
\]
where $C_2\coloneqq \norm*{\Delta^{p[(\vb+\vd)/p_1-\vs]}_\Omega(h_k) M_R(\mi)(\zeta_{j,k},z_{j,k}) }_{\ell^{p^*,q^*}(J,K)}^{1/p}$. 
Then, the first assertion follows from Lemma~\ref{lem:1}.

In order to prove the second assertion, it will suffice to show that,
if $\mi$ has compact support, then the canonical mapping
$A^{p_1,q_1}_{\vs}(D)\to L^{p}(\mi)$ is compact. Since
$A^{p_1,q_1}_{\vs}(D)$ embeds continuously, hence compactly, into
the Fr\'echet--Montel space $\Hol(D)$,  which in turn embeds
continuously into $L^p(\mi)$, the assertion follows easily. 
\end{proof}

\begin{prop}\label{prop:26}
Take $p_1,q_1,p\in ]0,\infty]$, with $p<\infty$, and take $\vs\in \R^r$ such that $\vs\in \frac{1}{2 q_1}\vect m+(\R_+^*)^r$ (resp.\ $\vs\in \R_+^r$ if $q_1=\infty$). 
Let $\mi\in\cM_+(D)$ and  $R>0$. Assume that $A^{p_1,q_1}_{\vs,0}(D)$ (resp.\ $A^{p_1,q_1}_{\vs}(D)$) embeds continuously into $L^{p}(\mi)$.
Then,
\[
M_R(\mi)\in L^{\infty,\infty}_{p(\vb+\vd)/p_1-p\vs}(D).
\]
If, in addition, $A^{p_1,q_1}_{\vs,0}(D)$ (resp.\ $A^{p_1,q_1}_{\vs}(D)$) embeds compactly into $L^{p}(\mi)$ and $\vect{s_1}\in (\R_+^*)^r$ if $p_1=q_1=\infty$, then
\[
M_R(\mi)\in L^{\infty,\infty}_{p(\vb+\vd)/p_1-p\vs,0}(D).
\]
\end{prop}

\begin{proof}
Take $\vect{s'}\in \R^r$ such that 
\[
\vect{s'}\in \frac{1}{p_1} (\vb+\vd)- \frac{1}{2 p_1}\vect{m'}-(\R_+^*)^r \qquad \text{and} \qquad
\vs+\vect{s'}\in \frac{1}{p_1} (\vb+\vd)- \frac{1}{2 q_1}\vect{m'}-(\R_+^*)^r.
\]
Then, $B^{\vect{s'}}_{(\zeta',z')}\in A^{p_1,q_1}_{\vs,0}(D)$ (resp.\ $B^{\vect{s'}}_{(\zeta',z')}\in A^{p_1,q_1}_{\vs}(D)$) for every $(\zeta',z')\in D$, and
\[
\norm{B^{\vect{s'}}_{(\zeta',z')}}_{A^{p_1,q_1}_{\vs}(D)}= C_1 \Delta_\Omega^{\vs+\vs^*-(\vb+\vd)/p_1}(\rho(\zeta',z'))
\]
for a suitable constant $C_1>0$, thanks to~\cite[Proposition 2.41]{CalziPeloso}. Hence, there is a constant $C_2>0$ such that
\[
\norm{B^{\vect{s'}}_{(\zeta',z')}}_{L^{p}(\mi)}\meg C_2 \Delta_\Omega^{\vs+\vs^*-(\vb+\vd)/p_1}(\rho(\zeta',z'))
\]
for every $(\zeta',z')\in D$.
In addition, there is a constant $C_3>0$  such that
\[
\frac{1}{C_3}\abs{B^{\vect{s'}}_{(\zeta',z')}}\meg \abs{B^{\vect{s'}}_{(\zeta,z)}}\meg C_3\abs{B^{\vect{s'}}_{(\zeta',z')}}
\]
on $D$, for every $(\zeta',z')\in D$ and for every $(\zeta,z)\in B((\zeta',z'),r)$, thanks to~\cite[Theorem 2.47]{CalziPeloso}. Therefore,
\[
\begin{split}
\norm{B^{\vect{s'}}_{(\zeta',z')}}_{L^{p}(\mi)}&\Meg \frac{1}{C_3} \Delta^{\vect{s'}}_{\Omega}(\rho(\zeta',z'))\norm{ \chi_{B((\zeta',z'),R)} }_{L^{p}(\mi)}\\
&= \frac{1}{C_3} \Delta^{\vect{s'}}_{\Omega}(\rho(\zeta',z')) M_R(\mi)(\zeta',z')^{1/p}
\end{split}
\]
for every $(\zeta',z')\in D$. It then follows that
\[
M_R(\mi)(\zeta',z')\meg C_2^{p} C_3^{ p} \Delta_\Omega^{p[\vs-(\vb+\vd)/p_1]  }(\rho(\zeta',z'))
\]
for every $(\zeta',z')\in D$.

Now, assume that $A^{p_1,q_1}_{\vs,0}(D)$ (resp.\ $A^{p_1,q_1}_{\vs}(D)$) embeds compactly in $L^{p}(\mi)$, and that $\vs\in (\R_+^*)^r$ if $p_1=q_1=\infty$.
Define $b_{(\zeta,z)}\coloneqq B^{\vect{s'}}_{(\zeta,z)}\Delta_\Omega^{(\vb+\vd)/p_1-\vs-\vect{s'}}(\rho(\zeta,z))$ for every $(\zeta,z)\in D$, so that the family $(b_{(\zeta,z)})_{(\zeta,z)\in D}$ is uniformly bounded in $A^{p_1,q_1}_{\vs}(D)$. In addition, since $\vs\in \frac{\vb+\vd}{p_1}+(\R_+^*)^r$,~\cite[Proposition 2.41]{CalziPeloso} implies that $B_{(\zeta,z)}^{\vect{s'}}\in A^{\infty,\infty}_{(\vb+\vd)/p_1-\vs-\vect{s'},0}(D)$, so that
\[
\lim_{(\zeta',z')\to \infty} \abs{b_{(\zeta',z')}(\zeta,z)}=\lim_{(\zeta',z')\to \infty} \Delta_\Omega^{(\vb+\vd)/p_1-\vs-\vect{s'}}(\rho(\zeta',z')) \abs{B^{\vect{s'}}_{(\zeta,z)}(\zeta',z')}=0
\]
for every $(\zeta,z)\in D$.  Let $b$ be a cluster point of $(b_{(\zeta,z)})$ in $L^p(\mi)$ for $(\zeta,z)\to \infty$, and observe that $b=0$ since $b_{(\zeta,z)}$ converges pointwise (and also locally uniformly, thanks to~\cite[Theorem 2.47]{CalziPeloso}) to $0$ as $(\zeta,z)\to 0$. Since $A^{p_1,q_1}_{\vs}(D)$ embeds compactly into $L^{p}(\mi)$, this suffices to prove that  $b_{(\zeta,z)}$  converges to $0$ in $L^p(\mi)$ as $(\zeta,z)\to \infty$.
Therefore, the computations of the proof of the implication (1) $\implies$ (2) show that
\[
\lim_{(\zeta,z)\to \infty}\Delta^{p[(\vb+\vd)/p_1-\vs]}_\Omega(\rho(\zeta,z))  M_R(\mi)(\zeta,z)=0,
\]
whence the result.
\end{proof}

We are now ready to prove our first main result.  It provides a
necessary and sufficient condition for a measure
$\mi\in\cM_+(D)$ to be a $p$-Carleson measure for $A^{p_1,q_1}_{\vs}$.
This  result  is analogous to the classical characterization in the
case  of the unit ball, see e.g.~\cite[Theorem 2.25]{Zhu2}, and it 
extends~\cite[Theorem 3.5]{NanaSehba}, where the case in which 
$p_1=q_1,p\in [1,\infty[$, $D$ is an irreducible symmetric tube
domain, and $\vs\in \R\vect 1_r$, is considered.

\begin{teo}\label{main:1}
Take $p_1,q_1,p\in ]0,\infty[$ with $p_1,q_1\meg p$, and take
$\vs\in \R^r$ such that $\vs\in \frac{1}{2 q_1}\vect
m+(\R_+^*)^r$. Let $\mi\in\cM_+(D)$ and  $R>0$. Then, the following
conditions are equivalent: 
\begin{enumerate}
\item[{\em (1)}] $A^{p_1,q_1}_{\vs}(D)$  embeds continuously (resp.\ compactly) into $L^{p}(\mi)$;
\item[{\em (2)}]  $M_R(\mi)\in L^{\infty,\infty}_{p[(\vb+\vect
d)/p_1-\vs]}(D)$ (resp.\  $M_R(\mi)\in
L^{\infty,\infty}_{p[(\vb+\vd)/p_1-\vs],0}(D)$). 
\end{enumerate}
\end{teo}

\begin{proof}
(2) $\implies$ (1). This follows from Proposition~\ref{prop:23}.

(1) $\implies$ (2). This follows from Proposition~\ref{prop:26}.
\end{proof}

We now deal with Carleson measures {\em admiting a loss}, that is, with the case $p_1,q_1\not \meg p$.  Note that when dealing with the mixed norm case, this complementary condition is not as simple as in the pure norm one. 
We provide a characterization of $p$-Carleson measures on
$A^{p_1,q_1}_\vs$ in two situations: when $A^{p_1,q_1}_\vs$ admits a
suitable atomic decomposition (Theorem~\ref{main:2}); when
$A^{p_1,q_1}_\vs$ embeds into $A^{p,p}_{\vs+(1/p-1/p_1)(\vect b+\vect
d)}$ (Theorem~\ref{teo:4}). 

The  next theorem extends~\cite[Theorem 3.8 (ii)]{NanaSehba},   where the case in which $p_1=q_1,p\in [1,\infty[$, $D$ is an irreducible symmetric tube
domain, and $\vect{s}\in \R\vect 1_r$, is considered. 
Recall that property $\atomic^{p_1,q_1}_{\vs, \vect{s'},0}$ is
defined   in Definition~\ref{def:L-property}.

\begin{teo}\label{main:2}
Take $p_1,q_1,p\in ]0,\infty]$, with $p<\infty$, and $\vs\in \frac{1}{2 q_1}\vect m+(\R_+^*)^r$.
Assume that property $\atomic^{p_1,q_1}_{\vs,\vect{s'},0}$ holds 
for some $\vect{s'}\in \R^r$, and let $\mi\in\cM_+(D)$.  
Then, the following conditions are equivalent: 
\begin{enumerate}
\item[{\em(1)}] $A^{p_1,q_1}_{\vs}(D)$  embeds continuously  into $L^{p}(\mi)$;
\item[{\em(2)}] $A^{p_1,q_1}_{\vs,0}(D)$  embeds continuously  into $L^{p}(\mi)$;
\item[{\em(3)}] for some (or, equivalently, every) $R>0$,
$M_R(\mi)\in L^{p^*,q^*}_{\vs^*}(D)$, where
$p^*,q^*,\vs^*$ are as in~\eqref{p*q*s'}.
\end{enumerate}
If, in addition, $p< q_1$, then the following conditions are equivalent:
\begin{enumerate}
\item[{\em(1$'$)}] $A^{p_1,q_1}_{\vs}(D)$  embeds compactly  into $L^{p}(\mi)$;
\item[{\em(2$'$)}] $A^{p_1,q_1}_{\vs,0}(D)$  embeds compactly  into $L^{p}(\mi)$;
\item[{\em(3$'$)}]for some (or, equivalently, every) $R>0$,
$M_R(\mi)\in L^{p^*,q^*}_{\vs^*,0}(D)$, where $p^*,q^*,\vs^*$ are
as in~\eqref{p*q*s'}.
\end{enumerate}
\end{teo}

Before we pass to the proof, we need a lemma.

\begin{lem}\label{lem:9}
Take $p_1,q_1,p_2,q_2\in ]0,\infty]$, $\vs\in\R^r$, $R>0$, and
$\mi\in\cM_+(D)$. Assume that $p_1\meg p_2$ and that $q_1\meg q_2$. Then, the following hold: 
\begin{enumerate}
\item[{\em(1)}] if $M_R(\mi)\in L^{p_1,q_1}_{\vs}(D)$, then $M_R(\mi)\in L^{p_2,q_2}_{\vs+(1/p_2-1/p_1)(\vb+\vd)}(D)$;
\item[{\em(2)}] if $M_R(\mi)\in L^{p_1,q_1}_{\vs,0}(D)$, then $M_R(\mi)\in L^{p_2,q_2}_{\vs+(1/p_2-1/p_1)(\vb+\vd),0}(D)$;
\item[{\em(3)}] if $q_1<\infty$, $M_R(\mi)\in L^{p_1,q_1}_{\vs}(D)\cap
L^{\infty,\infty}_{\vs-(\vb+\vd)/p_1,0}(D)$, then $M_R(\mi)\in
L^{p_1,q_1}_{\vs,0}(D)$. 
\end{enumerate}
\end{lem}

\begin{proof}
All the assertions follow easily from Lemma~\ref{lem:1} and the
elementary inclusions $\ell^{p_1,q_1}(J,K)\subseteq \ell^{p_2,q_2}(J,K)$,
$\ell^{p_1,q_1}_0(J,K)\subseteq \ell^{p_2,q_2}_0(J,K)$, and
$\ell^{p_1,q_1}(J,K)\cap \ell^{\infty,\infty}_0(J,K)\subseteq
\ell^{p_1,q_1}_0(J,K)$ (when $q_1<\infty$) for every two sets $J$ and
$K$.  
\end{proof}

\begin{proof}[Proof of Theorem~\ref{main:2}.] 
Clearly, (1) $\implies$ (2), while  the implication (3) $\implies$  (1) follows from  Proposition~\ref{prop:23}. 
Thus, it only remains to show that 
(2) $\implies$ (3). 

By assumption, there is a $(\delta,4)$-lattice
$(\zeta_{j,k},z_{j,k})_{j\in J,k\in K}$ for some $\delta\meg R/4$ such
that the mapping 
\[
\Psi\colon \ell^{p_1,q_1}_0(J,K)\ni \lambda \mapsto \sum_{j,k}
\lambda_{j,k} B^{ \vect{s'}}_{(\zeta_{j,k},z_{j,k})} \Delta_\Omega^{(\vect
b+\vd)/p_1-\vs- \vect{s'}}(h_k)\in A^{p_1,q_1}_{\vect
s,0}(D) 
\]
is well defined and continuous, where $h_k\coloneqq
\rho(\zeta_{j,k},z_{j,k})$ for every $j\in J$ and for every $k\in K$. 

Now, take a probability space $(X,\nu)$ and a sequence
$(r_\beta)_{\beta\in \N}$ of $\nu$-measurable functions on $X$ such
that 
\[
\Big(\bigotimes_{\beta\in B} r_\beta   \Big)(\nu)=\frac{1}{2^{\card(B)}}\sum_{\eps\in \Set{-1,1}^B} \delta_\eps
\]
for every finite subset $B$ of $\N$ (cf.~\cite[C.1]{Grafakos}). By
Khintchine's inequality, there is a constant $C_1>0$ such that 
\[
\frac{1}{C_1}\Big( \sum_{\beta\in\N} \abs{a_\beta}^2
\Big)^{1/2}\meg \bigg\| \sum_{\beta\in B} a_\beta r_\beta
\bigg\|_{L^{p}(\nu)}\meg C_1 \Big( \sum_{\beta\in\N} \abs{a_\beta}^2
\Big)^{1/2} 
\]
for every $(a_\beta)\in \C^{(\N)}$ (cf.~\cite[C.2]{Grafakos}). Let
$\iota\colon J\times K\to \N$ be a bijection and define $\vect{r}\cdot
\lambda\coloneqq ( r_{\gamma(j,k)} \lambda_{j,k})_{j,k}$ for every
$\lambda\in \C^{(J\times K)}$. By the assumptions and the continuity
of $\Psi$, there is a constant $C_2>0$ such that, for every
$\lambda\in \C^{(J\times K)}$, 
\[
\norm{\Psi(\vect r\cdot \lambda)}_{L^{p}(\mi)}\meg C_2  \norm{\vect
r\cdot \lambda}_{\ell^{p_1,q_1}(J,K)}= C_2
\norm{\lambda}_{\ell^{p_1,q_1}(J,K)} 
\]
$\nu$-almost everywhere.
Therefore, by means of Tonelli's theorem we see that
\[
\norm*{\left( \sum_{j,k} \abs*{\lambda_{j,k}  B^{ \vect{s'}}_{(\zeta_{j,k},z_{j,k})}}^2    \Delta_\Omega^{2[(\vb+\vd)/p_1-\vs-\vect{s'}]}(h_k)
\right)^{1/2}}_{L^{p}(\mi)}\meg C_1 C_2\norm{\lambda}_{\ell^{p_1,q_1}(J,K)} 
\]
for every $\lambda\in \C^{(J\times K)}$. Now, observe that there is
$N\in\N$ such that $\sum_{j,k} \chi_{B((\zeta_{j,k},z_{j,k}),R)}\meg
N$ on $D$, thanks to~\cite[Proposition 2.56]{CalziPeloso}. In
addition, there is a constant $C_3>0$ such that  
\[
\frac{1}{C_3}\Delta_\Omega^{ \vect{s'}}(h_k)\meg
\abs*{B^{ \vect{s'}}_{(\zeta_{j,k},z_{j,k})}}\meg C_3
\Delta_\Omega^{ \vect{s'}}(h_k) 
\]
on $B((\zeta_{j,k},z_{j,k}),R)$ for every $(j,k)\in J\times K$, thanks
to~\cite[Theorem 2.47]{CalziPeloso}. 
Then,
\[
\begin{split}
&\norm*{\left( \sum_{j,k} \abs*{\lambda_{j,k}
B^{ \vect{s'}}_{(\zeta_{j,k},z_{j,k})}}^2
\Delta_\Omega^{2[(\vb+\vd)/p_1-\vs- \vect{s'}]}(h_k)
\right)^{1/2}}_{L^{p}(\mi)}\\ 
&\qquad\Meg \frac{1}{C_3}
\norm*{\left(
\sum_{j,k}\chi_{B((\zeta_{j,k},z_{j,k}),R)}\abs*{\lambda_{j,k}}^2 \Delta_\Omega^{2[(\vb+\vd)/p_1-\vs]}(h_k) \right)^{1/2}}_{L^{p}(\mi)}\\ 
&\qquad\Meg \frac{N^{-(1-p/2)_+}}{C_3} \left(  \sum_{j,k}   \abs*{\lambda_{j,k}}^{p} \Delta_\Omega^{p[(\vb+\vd)/p_1-\vs]}(h_k) M_R(\mi)(\zeta_{j,k},z_{j,k}) \right)^{1/p}
\end{split}
\]
for every $\lambda\in \C^{(J\times K)}$. Using the natural duality between $\ell^{p_1/p,q_1/p}_0(J,K)$ and $\ell^{p^*,q^*}(J,K)$, we then see that
\[
\big(\Delta_\Omega^{p[(\vb+\vd)/p_1-\vs]}(h_k) M_R(\mi)(\zeta_{j,k},z_{j,k})\big)\in \ell^{p^*,q^*}(J,K)  .
\]
Then,  (3)   follows from Lemma~\ref{lem:1}.

Next, we turn to the second part of the statement.
The fact that (3$'$) $\implies$ (1$'$) follows from
Proposition~\ref{prop:23},   while it is   obvious that
(1$'$) $\implies$ (2$'$).  Finally, Proposition~\ref{prop:26},
Lemma~\ref{lem:9},   and the implication (2) $\implies$ (3) prove that
(2$'$) $\implies$ (3$'$). 
\end{proof}

With the same techniques, but using~\cite[Corollary 5.14]{CalziPeloso}
instead of property $\atomic^{p_1,q_1}_{\vs,\vect{s'},0}$, one
also proves the following result.  We recall that the spaces
$\widetilde A^{p,q}_{\vs}(D)$ and $\widetilde A^{p,q}_{\vs,0}(D)$ are
the holomorphic extensions of  the analytic-type Besov spaces $B^{-\vs}_{p,q}(\Nc,\Omega)$
and $\mathring B^{-\vs}_{p,q}(\Nc,\Omega)$, respectively,   see Definition~\ref{def:tilde-spaces}.

\begin{prop}\label{prop:1}
Take $p_1,q_1,p\in ]0,\infty]$, with $p<\infty$, and take  $\vs\in
\frac{1}{p_1}(\vb+\vd)+\frac{1}{2
q_1'}\vect{m'}+(\R_+^*)^r$.   Take $\mi\in\cM_+(D)$ so   that  $\widetilde A^{p_1,q_1}_{\vs,0}(D)$  embeds continuously into $L^{p}(\mi)$, and   take $p^*,q^*,\vs^*$    as in~\eqref{p*q*s'}. Then, 
\[
M_R(\mi)\in L^{p^*,q^*}_{\vs^*}(D)
\]
for every $R>0$.
\end{prop}

Notice, though, that an analogue of Proposition~\ref{prop:23} for the
spaces $ \widetilde A^{p_1,q_1}_{\vs,0}(D)$ is false, in general,
as the following Proposition shows.   We denote by $\Lc \nu$  the Laplace
transform of $\nu\in\cM_+(\Omega)$, that is,
\[
(\Lc \nu)(\lambda)= \int_\Omega \ee^{-\langle \lambda, h\rangle}\,\dd \nu(h)
\] 
for every $\lambda\in F'$.

\begin{prop}\label{prop:24}
Take $\vs\in \frac{\vb+\vd}{2}+\frac{1}{4}\vect{m'}+(\R_+^*)^r$   and a $bD$-invariant   $\mi\in\cM_+(D)$, and let $\nu$ be the unique Radon measure $\nu$ on
$\Omega$ such that $\int_D f\,\dd \mi= \int_\Omega \int_{\Nc}
f_h(\zeta,x)\,\dd (\zeta,x)\,\dd \nu(h)$ for every $f\in C_c(D)$. Take
$\vect{s'}\in -\frac{1}{4}\vect{m'}-\vs-(\R_+^*)^r$. Then, the
following conditions are equivalent: 
\begin{enumerate}
\item[{\em(1)}] $\widetilde A^{2,2}_{\vs}(D)$ embeds continuously (resp.\ compactly) into $L^2(\mi)$;

\item[{\em(2)}] the function $\Delta^{2\vs}_{\Omega'} \Lc\nu$ is
bounded (resp.\ is bounded and vanishes at $\infty$) on $\Omega'$; 

\item[{\em(3)}] the function
$h\mapsto\Delta_\Omega^{-\vs-\vect{s'}}(h)
\norm{\Delta^{\vect{s'}}(\,\cdot\,+h)}_{L^2(\nu)}$ is bounded
(resp.\ is bounded and vanishes at $\infty$) on $\Omega$. 
\end{enumerate}
\end{prop}

This extends~\cite[Theorem 3.1]{BekolleSehba}, where the case in which
$\vs=0$, $D$ is an irreducible symmetric tube domain, and $\nu$ is
absolutely continuous with respect to Lebesgue measure, is
considered.

Notice that this shows that the necessary conditions of
Proposition~\ref{prop:1} are not sufficient, in general. For example,
when $\vs=0$, then a $bD$-invariant positive Radon measure $\mi$ on
$D$ induces a continuous embedding of the Hardy space $\widetilde
A^{2,2}_{\vect 0}(D)  = A^{2,\infty}_{\vect 0}(D)$ into $L^2(\mi)$ if and only if the measure $\nu$
defined as in the statement is integrable on $\Omega$, while
Proposition~\ref{prop:1} only requires the mapping $h\mapsto
\nu(B_\Omega(h,R))$ to be bounded for some (every) $R>0$. 

\begin{proof}
(2) $\implies$ (1).  Take $f\in \widetilde A^{2,2}_{\vs}(D)\cap
A^{2,2}_{-\vd/2}(D)$,  and observe that   there is a constant $c>0$ such that , defining
$\tau_f(\lambda)\coloneqq \pi_\lambda(f_{e_\Omega}) \ee^{\langle
\lambda, e_\Omega\rangle}$ for almost every $\lambda\in \Omega'$, 
\[
\norm{f}_{L^2(\mi)}^2=c
\int_{\Omega'}  \norm{\tau_f(\lambda)}_{\Lin^2(H_\lambda)}^2 (\Lc\nu)(2 \lambda)\Delta^{-\vb}_{\Omega'}(\lambda)\,\dd \lambda,
\]
arguing as in the proof of~\cite[Corollary 1.38]{CalziPeloso}. In addition, we may choose a norm on $\widetilde A^{2,2}_{\vs}(D)$ so that
\[
\norm{f}_{\widetilde A^{2,2}_{\vs}(D)}^2=\int_{\Omega'}  \norm{\tau_f(\lambda)}_{\Lin^2(H_\lambda)}^2 \Delta^{-2\vs-\vb}_{\Omega'}(\lambda)\,\dd \lambda
\]
thanks to~\cite[Lemma 2.26, Proposition 3.11 and Proposition 5.13]{CalziPeloso}. 
Now, define 
\[
\Lc^2_{\vs}(\Omega')\coloneqq \Set{\tau\in \int_{\Omega'}^\oplus \Lin^2(H_\lambda) \Delta_{\Omega'}^{-\vect b-2\vect s}(\lambda)\,\dd \lambda\colon \tau=\tau P_{\,\cdot\,,0} }
\]
(cf.~\cite[Definition 3.10]{CalziPeloso}), and observe that the image of the mapping $\tau_f$, extended to $\widetilde A^{2,2}_{\vs}(D)$ by continuity, is the whole of $\Lc^2_{\vs}(\Omega')$ (cf.~\cite[Propositions 3.15, 3.17, 5.4, 5.13, and Corollary 5.11]{CalziPeloso}). Therefore,   $\widetilde A^{2,2}_{\vs}(D)$ embeds into $L^2(\mi)$ if and only if there is a constant $C_1>0$ such that
\[
\Lc\nu(2 \lambda)\meg C_1\Delta^{-2\vs}_{\Omega'}(\lambda)
\]
for every $\lambda\in \Omega'$.

Now, if $\Delta^{2\vs}_{\Omega'} \Lc\nu$ is bounded and vanishes at infinity, then it is clear that the mapping $\tau \mapsto \Delta^{2\vs}_{\Omega'} \Lc\nu \,\tau$ is a compact endomorphism of $\Lc^2_{\vs}(\Omega')$, so that the preceding remarks show that  $\widetilde A^{2,2}_{\vs}(D)$ embeds compactly into $L^2(\mi)$.

(1) $\implies$ (3). Observe first that $\norm{B^{\vect{s'}+(\vb+\vd)/2}_{(\zeta,z)}}_{\widetilde A^{2,2}_{\vs}(D)}= \Delta_\Omega^{\vs+\vect{s'}}(\rho(\zeta,z))$ for every $(\zeta,z)\in D$, for a suitable choice of a norm on $\widetilde A^{2,2}_{\vs}(D)$ (cf.~\cite[Proposition 2.41 and Lemma 5.15]{CalziPeloso}). In addition, there is a constant $C_2>0$ such that 
\[
\norm*{\left(B^{\vect{s'}+(\vb+\vd)/2}_{(0, i h)}\right)_{h'}}_{L^2(\Nc)}=C_2 \Delta_\Omega^{\vect{s'}}(h+h')
\]
for every $h,h'\in \Omega$, thanks to~\cite[Lemma 2.39]{CalziPeloso}. Therefore, (3) follows arguing as in the proof of Proposition~\ref{prop:26}.

(3) $\implies$ (2). Observe that there is a constant $C_3>0$ such that
\[
\norm{\Delta^{\vect{s'}}(\,\cdot\,+h)  }_{L^2(\nu_t)}\meg C_3\Delta_\Omega^{\vs+\vect{s'}}(h) \quad \forall t\in T_+ \qquad\text{(resp.\ and   $\lim\limits_{t\to \infty }  \norm{\Delta^{\vect{s'}}(\,\cdot\,+h)  }_{L^2(\nu_t)}=0$)}
\]
for  every $h\in \Omega$, where $\nu_t= \Delta^{2\vs}(t)
(t\,\cdot\,)_* \nu$, and that (2) is equivalent to saying that 
\[
\sup_{t\in t_+}\Lc \nu_t(e_{\Omega}')<\infty \qquad \text{(resp.\ $\lim\limits_{t\to\infty}\Lc \nu_t(e_{\Omega}')=0  $).}
\]

Now,~\cite[Corollary 2.36]{CalziPeloso} implies that
\[
\Delta^{2\vect{s'}}(h+h')\Meg \Delta^{2 \vect{s'}}(2 h)=2^{2\vect{s'}} \Delta^{2 \vect{s'}}(h)
\]
for every $h\in \Omega$ and for every $h'\in \Omega \cap (h-\Omega)$,
so that
\[
\nu_t(\Omega\cap (h-\Omega))\meg 2^{-2\vect{s'}}\Delta^{-2
\vect{s'}}(h)\int_{\Omega\cap (h-\Omega)}
\Delta^{2\vect{s'}}(h+h')\,\dd \nu_t(h') \meg 2^{-2\vect{s'}} C_3^2
\Delta^{2\vs}(h) 
\]
for every $h\in\Omega$ and for every $t\in T_+$.
In addition, fix $R>0$ so that $\langle e_{\Omega'},h\rangle\Meg 1$
for every $h\in \Omega \setminus (Re_\Omega-\Omega)$.  Then,  
\[
\begin{split}
\Lc\nu_t(e_{\Omega'})&=\int_\Omega \ee^{-\langle e_{\Omega'}, h\rangle}\,\dd \nu_t(h)\\
&\meg \nu_t(\Omega\cap (R e_\Omega-\Omega))+\sum_{k\in\N}  \ee^{-2^k R
} \nu_t(\Omega\cap (2^{k+1} R e_\Omega-\Omega)\setminus (2^{k} R
e_\Omega-\Omega))\\ 
&\meg 2^{-2\vect{s'}} C_3^2\left( R^{2\vs}+ \sum_{k\in\N} (2^{k+1}
R )^{2\vs} \ee^{- 2^k R}\right) 
\end{split}
\]
for every $t\in T_+$. If, in addition, $\lim\limits_{t\to \infty }
\norm{\Delta^{\vect{s'}}(\,\cdot\,+h)  }_{L^2(\nu_t)}=0$ for every
$h\in\Omega$, then the preceding computations show that $
\nu_t(\Omega\cap (2^{k+1} R e_\Omega-\Omega)\setminus (2^{k} R
e_\Omega-\Omega))\to 0$ for $t\to+\infty$ for every $k\in \N$, so that
it is readily seen that $\Lc \nu_t(e_{\Omega'})\to 0$ for
$t\to\infty$. 
\end{proof}

We now come to the last  main result of this section.  It provides a
necessary and sufficient condition for a measure $\mi\in\cM_+(D)$  to be a
$p$-Carleson  measure for $A^{p_1,q_1}_\vs$ also in terms of
integrability properties of Bergman kernels.

\begin{teo} \label{teo:4}
Take    $p_1,q_1,p\in ]0,\infty]$, with $p<\infty$, $\mi\in\cM_+(D)$,
$\vs\in \R^r$, and $\vect{s'}\in \C^r$ such 
that the following hold: 
\begin{itemize}
\item[{\em(i)}] $\vs\in \frac{1}{2 q_1}\vect m+(\R_+^*)^r$ (resp.\
$\vs \in \R_+^r$ if $q_1=\infty$); 
\item[{\em(ii)}]  $\Rea\vect{s'}\in \frac{1}{p_1}(\vb+\vect
d)-\frac{1}{2 p_1}\vect{m'}-(\R_+^*)^r $ (resp.\ $\Rea\vect{s'}\in
-\R_+^r $ if $p_1=\infty$); 

\item[{\em(iii)}]  $\vs+\Rea\vect{s'}\in \frac{1}{p_1}(\vect
b+\vd)-\frac{1}{2 q_1}\vect{m'}-(\R_+^*)^r $  (resp.\ $\vect
s+\Rea\vect{s'}\in  \frac{1}{p_1}(\vb+\vd)-\R_+^r $ if
$q_1=\infty$); 
\item[{\em(iv)}]  $A^{p_1,q_1}_{\vs,0}(D)$ (resp.\
$A^{p_1,q_1}_{\vs}(D)$) embeds continuously into
$A^{p,p}_{\vect{s''}}(D)$, where $\vect{s''}  \coloneqq \vect
s+\left(\frac{1}{p}-\frac{1}{p_1} \right)(\vb+\vd)$. 
\end{itemize}
Then, the following conditions are equivalent for every $R>0$:
\begin{enumerate}
\item[{\em(1)}] $A^{p_1,q_1}_{\vs,0}(D)$ (resp.\
$A^{p_1,q_1}_{\vs}(D)$) embeds continuously into $L^{p}(\mi)$;

\item[{\em(2)}]  there is a constant $C>0$ such that
\[
\norm{B^{\vect{s'}}_{(\zeta,z)}}_{L^{p}(\mi)}\meg C
\Delta^{\vect{s}+\Rea\vect{s'}-(\vb+\vd)/p_1}(\rho(\zeta,z)) 
\]
for every $(\zeta,z)\in D$;
\item[{\em(3)}]  $M_R(\mi)\in L^{\infty,\infty}_{p[(\vb+\vd)/p_1-\vs]}(D)$.
\end{enumerate}

If, in addition, $\vs,-\vs-\Rea\vect{s'}\in (\R_+^*)^r$ when $p_1=q_1=\infty$, then the following conditions are equivalent for every $R>0$:
\begin{enumerate}
\item[{\em(1$'$)}] $A^{p_1,q_1}_{\vs,0}(D)$ (resp.\  $A^{p_1,q_1}_{\vs}(D)$) embeds compactly into $L^{p}(\mi)$;
\item[{\em(2$'$)}]  the function $(\zeta,z)\mapsto\Delta^{(\vb+\vd)/p_1-\vect{s}-\Rea\vect{s'}}(\rho(\zeta,z)) \norm{B^{\vect{s'}}_{(\zeta,z)}}_{L^{p}(\mi)}$ is bounded and vanishes at $\infty$ on $D$;
\item[{\em(3$'$)}]  $M_R(\mi)\in L^{\infty,\infty}_{p[(\vb+\vd)/p_1-\vs],0}(D)$.
\end{enumerate}
\end{teo}

This extends~\cite[Theorem 1.1]{BekolleSehbaTchoundja}, where the case
in which $\vs=0$,  $\vect{s'}\in \R\vect 1_r$, $q=\infty$, and $D$
is an irreducible symmetric tube domain, is considered.

Notice that, if (i) to (iv) hold, then (i), (ii), and (iii) hold with
$\vs$, $p_1$, and $q_1$ replaced by $\vect{s''}$, $p$, and $p$,
respectively, thanks to~\cite[Proposition 2.41]{CalziPeloso}.  

Conversely, if (2) implies (1) and (i), (ii), and (iii) hold with
$\vs$, $p_1$, and $q_1$ replaced by $\vect{s''}$, $p$, and $p$,
respectively, then (iv) holds thanks to~\cite[Proposition 2.41]{CalziPeloso} again. 

Observe that $A^{2,\infty}_{\vect 0}(D)=\widetilde A^{2,2}_{\vect 0}(D)\subseteq \widetilde A^{p,p}_{(1/p-1/2)(\vb+\vd)}(D)$ for every $p\Meg 2$, thanks to~\cite[Proposition 4.19]{CalziPeloso}, so that the assumptions are satisfied whenever $\vect{s''}=(1/p-1/2)(\vb+\vd)$ and $\widetilde A^{p,p}_{\vect{s''}}(D)= A^{p,p}_{\vect{s''}}(D)$. This is the case, for example, when (cf.~\cite[Corollary 5.11]{CalziPeloso})
\[
\left( \frac{1}{p}-\frac 1 2 \right)(\vb+\vd)\in \sup \left(\frac{1}{2 p}\vect m+\left(\frac{1}{2}-\frac{1}{p}\right)\vect{m'}, \frac{1}{p}(\vb+\vd)+\frac{1}{2 p'}\vect{m'}    \right)+(\R_+^*)^r.
\]
\begin{proof}
We prove only the equivalence of conditions (1)--(3). The equivalence of conditions (1$'$)--(3$'$) is proved similarly, using the techniques employed in the proof of Proposition~\ref{prop:26}.

(1) $\implies$ (2). This follows from~\cite[Proposition 2.41]{CalziPeloso}.

(2) $\implies$ (3). This follows easily from~\cite[Theorem 2.47]{CalziPeloso}.

(3) $\implies$ (1). By Theorem~\ref{main:1}, $A^{p,p}_{\vs}(D)$ embeds continuously into $L^{p}(\mi)$, so that the assertion
follows. 
\end{proof}

With a similar proof, but using~\cite[Lemma 5.15]{CalziPeloso} instead
of~\cite[Proposition 2.41]{CalziPeloso}, one may prove the following
result, which  deals with $p$-Carleson measures for the spaces
$\widetilde A^{p_1,q_1}_{\vs}(D)$.  

\begin{teo}
Take $p_1,q_1,p\in ]0,\infty[$, $\mi\in\cM_+(D)$, $\vs\in \R^r$, and $\vect{s'}\in \C^r$ such that the following hold:
\begin{itemize}
\item[{\em(i)}] $\vs\in \frac{1}{ p_1}(\vb+\vd)+\frac{1}{2 q_1'}\vect{m'}+(\R_+^*)^r$; 
\item[{\em(ii)}]  $\vs+\Rea\vect{s'}\in \frac{1}{p_1}(\vb+\vd)-\frac{1}{2 q_1}\vect{m'}-(\R_+^*)^r $  (resp.\ $\vs+\Rea\vect{s'}\in  \frac{1}{p_1}(\vb+\vd)-\R_+^r $ if $q_1=\infty$);
\item[{\em(iii)}]  $\widetilde A^{p_1,q_1}_{\vs,0}(D)$ (resp.\ $\widetilde A^{p_1,q_1}_{\vs}(D)$) embeds continuously into
$A^{p,p}_{\vect{s''}}(D)$, where $\vect{s''}\coloneqq\vs+\left(\frac{1}{p}-\frac{1}{p_1} \right)(\vb+\vd)$. 
\end{itemize}
Then, the following conditions are equivalent for every $R>0$:
\begin{enumerate}
\item[{\em(1)}] $\widetilde A^{p_1,q_1}_{\vs,0}(D)$ (resp.\  $\widetilde A^{p_1,q_1}_{\vs}(D)$) embeds continuously into $L^{p}(\mi)$;

\item[{\em(2)}]  there is a constant $C>0$ such that
\[
\norm{B^{\vect{s'}}_{(\zeta,z)}}_{L^{p}(\mi)}\meg C \Delta^{\vect{s}+\Rea\vect{s'}-(\vb+\vd)/p_1}(\rho(\zeta,z))
\]
for every $(\zeta,z)\in D$;

\item[{\em(3)}]  $M_R(\mi)\in L^{\infty,\infty}_{p[(\vb+\vd)/p_1-\vs]}(D)$.
\end{enumerate}
If, in addition, $\vs+\Rea\vect{s'}\in -(\R_+^*)^r$ when $p_1=q_1=\infty$, then the following conditions are equivalent for every $R>0$:
\begin{enumerate}
\item[{\em(1$'$)}] $\widetilde A^{p_1,q_1}_{\vs,0}(D)$ (resp.\  $\widetilde A^{p_1,q_1}_{\vs}(D)$) embeds compactly into $L^{p}(\mi)$;

\item[{\em(2$'$)}]  the function
$(\zeta,z)\mapsto\Delta^{(\vb+\vd)/p_1-\vect{s}-\Rea\vect{s'}}(\rho(\zeta,z))
\norm{B^{\vect{s'}}_{(\zeta,z)}}_{L^{p}(\mi)}$ is bounded and
vanishes at $\infty$ on $D$; 

\item[{\em(3$'$)}]  $M_R(\mi)\in L^{\infty,\infty}_{p[(\vb+\vd)/p_1-\vs],0}(D)$.
\end{enumerate}
\end{teo}

Notice that, by Corollary~\ref{cor:1} below, if (iii) holds, then $p_1,q_1\meg p $.

\section{Inclusion Between the Spaces $\widetilde A^{p,q}_{\vs}(D)$
and  $A^{p,q}_{\vs}(D)$}\label{sec:5}

A classical result by  Hardy and Littlewood in dimension 1  shows that  
the Hardy space $H^2(\C_+)$  embeds continuously into
$A^{p,p}_{\frac12-\frac1p}(\C_+)$ when $p\ge4$.  In this section we
address a similar question. The main result is Theorem~\ref{prop:2},
in which we restrict ourselves to the setting of Siegel domains of Type I.
In addition to that, we also prove some necessary conditions for the spaces $A$ and $\widetilde A$ to embed into one another. For sufficient condition, see~\cite[Propositions 3.2, 3.7, and 4.19]{CalziPeloso}. 

\begin{lem}\label{lem:2}
Take $p_1,p_2,q_1,q_2\in ]0,\infty]$ and $\vect{s_1},\vect{s_2}\in \R^r$, and assume that $\vect{s_1}\in \frac{1}{2 q_1}\vect m+(\R_+^*)^r$ (resp.\ $\vect{s_1}\in \R_+^r$ if $q_1=\infty$) and that $A^{p_1,q_1}_{\vect{s_1},0}(D)\subseteq A^{p_2,q_2}_{\vect{s_2}}(D)$ (resp.\ $A^{p_1,q_1}_{\vect{s_1}}(D)\subseteq A^{p_2,q_2}_{\vect{s_2}}(D)$).
Then, 
\[
p_1\meg p_2\qquad \text{and} \qquad \vect{s_2}=\vect{s_1}+\left(\frac{1}{p_2}-\frac{1}{p_1}\right)(\vb+\vd).
\]
\end{lem}

As Theorem~\ref{prop:2} below shows, it is not possible to deduce $q_1\meg q_2$ \emph{in this generality}. Cf.~Corollary~\ref{cor:1} below, though.

\begin{proof}
By assumption, there is a constant $C>0$ such that
\[
\norm{f}_{A^{p_2,q_2}_{\vect{s_2}}(D)}\meg C\norm{f}_{A^{p_1,q_1}_{\vect{s_1}}(D)}
\]
for every $f\in A^{p_1,q_1}_{\vect{s_1},0}(D)$ (resp.\ for every $f\in A^{p_1,q_1}_{\vect{s_1}}(D)$). We assume that $C$ is the least possible constant for which such an inequality holds. 
Now, for every $t\in T_+$ choose $g_t\in GL(E)$ so that $t\cdot \Phi=\Phi\circ (g_t\times g_t)$. Then,
\[
\norm{f\circ (g_t\times t)}_{A^{p_1,q_1}_{\vect{s_1}}(D)}=\Delta^{-\vect{s_1}+(\vb+\vd)/p_1}(t)\norm{f}_{A^{p_1,q_1}_{\vect{s_1}}(D)}
\]
and
\[
\norm{f\circ (g_t\times t)}_{A^{p_2,q_2}_{\vect{s_2}}(D)}=\Delta^{-\vect{s_2}+(\vb+\vd)/p_2}(t)\norm{f}_{A^{p_2,q_2}_{\vect{s_2}}(D)}
\]
for every $f\in \Hol(D)$, thanks to Lemma~\ref{lem:5}. Therefore, the arbitrariness of $t\in T_+$ and the minimality of $C$ imply that
\[
\vect{s_2}=\vect{s_1}+(1/p_2-1/p_1)(\vb+\vd).
\]
Next, assume that $A^{p_1,q_1}_{\vect{s_1},0}(D)\neq 0$, that is, that $\vect{s_1}\in \frac{1}{2 q_1}\vect{m}+(\R_+^*)^r$. 
Then,
\[
\begin{split}
2^{1/p_2}\norm{f}_{A^{p_2,q_2}_{\vect{s_2}}(D)}&=\lim_{(\zeta,x)\to \infty}\norm{f-f((\zeta,x+i\Phi(\zeta))\,\cdot\,)}_{A^{p_2,q_2}_{\vect{s_2}}(D)}\\
&\meg C\lim_{(\zeta,x)\to \infty}\norm{f-f((\zeta,x+i\Phi(\zeta))\,\cdot\,)}_{A^{p_1,q_1}_{\vect{s_1}}(D)}\\
&=2^{1/p_1} C\norm{f}_{A^{p_1,q_1}_{\vect{s_1}}(D)}
\end{split}
\]
for every $f\in A^{p_1,q_1}_{\vect{s_1},0}(D)$, so that $1/p_1\Meg 1/p_2$, that is, $p_1\meg p_2$, by the minimality of $C$.

Finally, assume that $A^{p_1,q_1}_{\vect{s_1},0}(D)=0$, so that
$q_1=\infty$ and  $(s_1)_j=0$ for some $j\in \Set{1,\dots,r}$. Then,
for every $\eps>0$ we may find $\vect{s'}\in \R^r$ such that
$s'_j=\frac{1}{p_1}(b_j+d_j)-\frac{1}{2 p_1}m'_j-\eps$ and such that
$B^{\vect{s'}}_{(\zeta,z)}\in A^{p_1,q_1}_{\vect{s_1}}(D)$ for every
$(\zeta,z)\in D$ (cf.~\cite[Proposition 2.41]{CalziPeloso}), so that
$B^{\vect{s'}}_{(\zeta,z)}\in
A^{p_2,q_2}_{\vect{s_2}}(D)$. Then,~\cite[Proposition
2.41]{CalziPeloso} implies that $s'_j\meg
\frac{1}{p_2}(b_j+d_j)-\frac{1}{2 p_2} m'_j$. Since $-b_j,m'_j\Meg 0$
and $-d_j>0$, the arbitrariness of $\eps$ implies that $1/p_1\Meg
1/p_2$, that is, $p_1\meg p_2$. 
\end{proof}

\begin{lem}\label{lem:3}
Take $p_1,p_2,q_1,q_2\in ]0,\infty]$ and $\vect{s_1},\vect{s_2}\in
\R^r$, and assume that the canonical mapping $\Sc_{\Omega,L}(\Nc)\to
B^{\vect{s_2}}_{p_2,q_2}(\Nc,\Omega)$ induces a continuous linear
mapping $\mathring B^{\vect{s_1}}_{p_1,q_1}(\Nc,\Omega)\to
B^{\vect{s_2}}_{p_2,q_2}(\Nc,\Omega)$. Then, $p_1\meg p_2$, $q_1\meg
q_2$, and
$\vect{s_2}=\vect{s_1}+\left(\frac{1}{p_1}-\frac{1}{p_2}\right)(\vb+\vd)$. 
\end{lem}

\begin{proof}
Applying the operator $u\mapsto u*I^{\vect{s'}}_\Omega$ for a suitable $\vect{s'}\in \R^r$, we may assume that $\vect{s_1},\vect{s_2} $ are sufficiently small so as to ensure that the mapping $\Ec$ is defined and induces isomorphisms of $\mathring B^{\vect{s_1}}_{p_1,q_1}(\Nc,\Omega)$ and $B^{\vect{s_2}}_{p_2,q_2}(\Nc,\Omega)$ onto $A^{p_1,q_1}_{-\vect{s_1},0}(D)$ and $A^{p_2,q_2}_{-\vect{s_2}}(D)$, respectively (cf.~\cite[Theorem 4.26 and Corollary 5.11]{CalziPeloso}). Then, Lemma~\ref{lem:2} implies that $p_1\meg p_2$ and that $\vect{s_2}=\vect{s_1}+\left(\frac{1}{p_1}-\frac{1}{p_2}\right)(\vb+\vd)$.
Therefore, it only remains to prove that $q_1\meg q_2$. Then, take
$\vect{s'}\in \R^r$ and observe that~\cite[Lemma 5.15]{CalziPeloso}
implies that $B^{\vect{s'}}_{(\zeta,z)}\in
A^{p_1,q_1}_{\vect{s_1},0}(D)$ for some (or, equivalently, every)
$(\zeta,z)\in D$ if and only if $\vect{s_1}+\vect{s'}\in
\frac{1}{p_1}(\vb+\vd)-\frac{1}{2 q_1}\vect{m'}-(\R_+^*)^r$, and that
$B^{\vect{s'}}_{(\zeta,z)}\in A^{p_2,q_2}_{\vect{s_2}}(D)$ for some
(or, equivalently, every) $(\zeta,z)\in D$ if and only if
$\vect{s_2}+\vect{s'}\in \frac{1}{p_2}(\vb+\vd)-\frac{1}{2
q_2}\vect{m'}-(\R_+^*)^r$ and $q_2<\infty$ or $\vect{s_2}+\vect{s'}\in
\frac{1}{p_2}(\vb+\vd)-\R_+^r$ and $q_2=\infty$. Therefore, $q_1\meg
q_2$ \emph{provided that $\vect{m'}\neq \vect 0$}.   

Thus, we only need to consider the case $\vect{m'}=\vect 0$, that is,
$r=m$. In this case, we may assume that $\Omega=\Omega'=(\R_+^*)^r$,
identifying $\R^r$ with its dual. 
Define $K=\Z^r$ and $t_k\coloneqq\lambda_k\coloneqq
2^k=(2^{k_1},\dots, 2^{k_r})$ for every $k\in K$. Notice that the
distance $(\lambda, \lambda')\mapsto \sum_{j=1}^r
\abs{\log(\lambda_j/\lambda'_j)}$ on $\Omega'$ is
$G(\Omega')$-invariant and clearly locally bi-Lipschitz equivalent to
$d_{\Omega'}$, hence bi-Lipschitz equivalent to $d_{\Omega'}$ near every point, with
uniform constants. Therefore, it is readily verified that
$(\lambda_k)$ is a $(\delta,R)$-lattice for some $\delta>0$ and some
$R>1$. 
Notice that $\Omega$ is a group under pointwise multiplication, so  
that we may choose $T_+=\Omega$ with the natural action. In this case, 
$\Delta^{\vs}(t)=t^{\vs}=\prod_{j=1}^r t_j^{s_j}$ for every $\vs\in
\C^r$. 
Now, choose a positive $\phi\in C_c^\infty(\Omega')$ so that
$\sum_{k\in K} \phi(\,\cdot\, t_k)=1$ on $\Omega'$, and define
$\psi_k\coloneqq \Fc_\Nc^{-1}(\phi(\,\cdot\, t_k))$ for every $k\in
K$. 
Then,
\[
u\mapsto \norm{u}_{B^{\vs}_{p,q}(\Nc,\Omega)}\coloneqq\norm*{
\Delta^{\vs}_{\Omega'}(\lambda_k) \norm{u*\psi_k}_{L^p(\Nc)}}_{\ell^q(K)} 
\]
is a quasi-norm which defines the topology of
$B^{\vs}_{p,q}(\Nc,\Omega)$ for every $p,q\in ]0,\infty]$ and for
every $\vs\in \R^r$. Define $\Psi_t(u)\coloneqq
\Delta^{\vect{s_1}-(\vb+\vd)(1-1/p_1)}(t\,\cdot\,)_*u$ for every $u\in
\Sc_{\Omega,L}'(\Nc)$ and for every $t\in T_+$. Then, clearly 
\[
\norm{ \Psi_{t_k}(u)}_{B^{\vect {s_j}}_{p_j,q_j}(\Nc,\Omega)}=\norm{ u}_{B^{\vect {s_j}}_{p_j,q_j}(\Nc,\Omega)}
\]
for every $j=1,2$, for every $k\in K$, and for every $u\in B^{\vect {s_j}}_{p_j,q_j}(\Nc,\Omega)$.
In addition, it is readily seen that
\[
\lim_{k\to \infty} \norm{ f-\Psi_{t_k}(f)}_{B^{\vect {s_j}}_{p_j,q_j}(\Nc,\Omega)}=2^{1/q_j}\norm{ f}_{B^{\vect {s_j}}_{p_j,q_j}(\Nc,\Omega)}
\]
for every $j=1,2$ and for every $f\in \Sc_{\Omega,L}(\Nc)$. Now, by assumption there is a constant $C>0$ such that
\[
\norm{f}_{B^{\vect {s_2}}_{p_2,q_2}(\Nc,\Omega)}\meg C \norm{f}_{B^{\vect {s_1}}_{p_1,q_1}(\Nc,\Omega)}
\]
for every $f\in \Sc_{\Omega,L}(\Nc)$. Then, applying the same inequality to $f-\Psi_{t_k}(f)$ and passing to the limit for $k\to \infty$,
\[
\norm{f}_{B^{\vect {s_2}}_{p_2,q_2}(\Nc,\Omega)}\meg 2^{1/q_1-1/q_2}C\norm{f}_{B^{\vect {s_1}}_{p_1,q_1}(\Nc,\Omega)}
\]
for every $f\in \Sc_{\Omega,L}(\Nc)$. Since we may have assumed $C$ to be minimal, it is readily seen that $q_1\meg q_2$.
\end{proof}

\begin{cor}\label{cor:1}
Take $p_1,p_2,q_1,q_2\in ]0,\infty]$ and $\vect{s_1},\vect{s_2}\in
\R^r$ such that $\vect{s_1}\in \frac{1}{p_1}(\vb+\vd)+\frac{1}{2
q_1'}\vect{m'}+(\R_+^*)^r$ and such that $\widetilde
A^{p_1,q_1}_{\vect{s_1},0}(D)\subseteq A^{p_2,q_2}_{\vect{s_2}}(D)$. 
Then, $p_1\meg p_2$, $q_1\meg q_2$, and
$\vect{s_2}=\vect{s_1}+\left(\frac{1}{p_1}-\frac{1}{p_2}\right)(\vb+\vd)$. 
\end{cor}

Notice that, if $p_1\meg p_2$, $q_1\meg q_2$, and
$\vect{s_2}=\vect{s_1}+\left(\frac{1}{p_1}-\frac{1}{p_2}\right)(\vb+\vd)$, then $A^{p_1,q_1}_{\vect{s_1}}\subseteq A^{p_2,q_2}_{\vect{s_2}}$ and $\widetilde A^{p_1,q_1}_{\vect{s_1}}\subseteq \widetilde A^{p_2,q_2}_{\vect{s_2}}$ (cf.~\cite[Propositions 3.2 and 4.19]{CalziPeloso}). Nonetheless, $\widetilde
A^{p_1,q_1}_{\vect{s_1},0}$ need not embed into $ A^{p_2,q_2}_{\vect{s_2}}$.  

\begin{proof}
Take $\vect{s'}\in \N_{\Omega'}$ so that
$A^{p_2,q_2}_{\vect{s_2}+\vect{s'}}(D)=\widetilde
A^{p_2,q_2}_{\vect{s_2}+\vect{s'}}(D)$ (cf.~\cite[Corollary
5.11]{CalziPeloso}). Observe that the mapping $f\mapsto f*
I^{-\vect{s'}}_\Omega$ induces an isomorphism of $\widetilde
A^{p_1,q_1}_{\vect{s_1},0}(D)$ onto $\widetilde
A^{p_1,q_1}_{\vect{s_1}+\vect{s'},0}(D)$ by~\cite[Proposition
5.13]{CalziPeloso}, and a continuous mapping of
$A^{p_2,q_2}_{\vect{s_2}}(D)$ into
$A^{p_2,q_2}_{\vect{s_2}+\vect{s'}}(D)$ by~\cite[Corollary
3.27]{CalziPeloso}. Thus, $\widetilde
A^{p_1,q_1}_{\vect{s_1}+\vect{s'},0}(D)\subseteq \widetilde
A^{p_2,q_2}_{\vect{s_2}+\vect{s'}}(D)$, so that the conclusion follows
from Lemma~\ref{lem:3}. 
\end{proof}

\begin{teo}\label{prop:2}  
Assume that $n=0$, and take $k\in\N^*$ and
\[
\vs\in \sup\left(\frac 1 2 \vd+\frac 1 4 \vect{m'},\frac{k-1}{2k}
\vd+\frac{1}{4k} \vect{m'}, \frac{k-1}{2 k}\vd+\frac{1}{4 k}\vect m\right)+(\R_+^*)^r.
\]
Then, $\widetilde A^{2,2}_{\vs}(D)$ embeds continuously into
$A^{q,q}_{\vs-(1/2-1/q)\vd}(D)$ for every $q\in [2k,\infty]$. 
\end{teo}

This extends~\cite[Theorem 1.3]{BekolleSehbaTchoundja}, where the case
in which $\vs=\vect 0$ and $D$ is an irreducible symmetric tube domain is
considered. Notice that~\cite[Theorem 1.4]{BekolleSehbaTchoundja}
provides better results when $\vs=\vect 0$, $r=2$, and $m=3,4,5,6$. 

Notice, in addition, (cf.~\cite[Remark 1.5]{BekolleSehbaTchoundja})
that if $A^{p_1,q_1}_{\vect{s_1}}(D)$ embeds continuously into
$A^{p_2,q_2}_{\vect{s_2}}(D)$, then also $ A^{kp_1,kq_1}_{\vect{s_1}/k}(D)$
embeds continuously into $A^{ k p_2,kq_2}_{\vect{s_2}/k}(D) $ forvevery $k\in \N^*$.
Indeed, $f\in A^{kp_1,k q_1}_{\vect{s_1}/k}(D)$ if and only if
$f^k\in  A^{p_1,q_1}_{\vect{s_1}}(D)$, so that $f^k\in
A^{p_2,q_2}_{\vect{s_2}}(D)$, which is equivalent to $f\in A^{ k
p_2,kq_2}_{\vect{s_2}/k}(D)$.   
Unfortunately, it is not known if the weighted Bergman spaces
interpolate in full generality, so that this fact cannot be used to
deduce similar embedding for every $k\in [1,\infty[$. 

\begin{proof}
Define $\tau\colon \widetilde A^{2,2}_{\vs}(D)\cap A^{2,\infty}_{\vect
0}(D)\to L^2(\Omega', \Delta^{-2 \vs}\cdot\Hc^m)$ so that
$\tau(f)(\lambda)=\Fc(f_{e_\Omega})(\lambda) \ee^{\langle \lambda,
e_\Omega\rangle}$ for almost every $\lambda\in \Omega'$. Observe
that we may choose a norm on $\widetilde A^{2,2}_{\vs}(D)$ such that 
\[
\norm{f}_{\widetilde A^{2,2}_{\vs}(D)}^2=  \int_{\Omega'} \abs{\tau(f)(\lambda)}^2 \Delta_{\Omega'}^{-2 \vs}(\lambda)\,\dd \lambda
\]
for every $f\in \widetilde A^{2,2}_{\vs}(D)\cap A^{2,\infty}_{\vect 0}(D)$, so that we may extend $\tau$ to the whole of $\widetilde A^{2,2}_{\vs}(D)$. In addition,
\[
f(z)= \frac{1}{(2 \pi)^m} \int_{\Omega'} \tau(f)(\lambda) \ee^{\langle \lambda_\C, i z \rangle}\,\dd \lambda
\]
for every $z\in D$ and for every $f\in \widetilde
A^{2,2}_{\vs}(D)$.\footnote{This formula is clear when $f\in
\widetilde A^{2,2}_{\vs}(D)\cap A^{2,\infty}_{\vect 0}(D)$ thanks
to~\cite[Proposition 1.39]{CalziPeloso}, and then follows by
continuity in the general case, thanks to~\cite[Propositions 2.19 and 5.4]{CalziPeloso}.} It then follows that $\tau(f^k)=
(2\pi)^{-(k-1)m} \tau(f)^{*k}$ for every $f\in \widetilde
A^{2,2}_{\vs}(D)\cap \Ec(\Sc_{\Omega,L}(\Nc))$, where $\tau(f)^{*k}$
denotes the convolution of $k$ functions all equal to $\tau(f)$.  Now, 
\[
\begin{split}
\abs{[\tau(f)^{*k}](\lambda)}&\meg \int_{\Omega'^{k-1}} \abs{\tau(f)(\lambda-\lambda_1)\cdots \tau(f)(\lambda_{k-2}-\lambda_{k-1})}\,\dd (\lambda_1,\dots,\lambda_{k-1})\\
&\meg\left( \int_{\Omega'^{k-1}} (\abs{\tau(f)}^2 \Delta_{\Omega'}^{-2\vs})(\lambda-\lambda_1)\cdots (\abs{\tau(f)}^2\Delta_{\Omega'}^{-2\vs})(\lambda_{k-2}-\lambda_{k-1})\,\dd (\lambda_1,\dots,\lambda_{k-1})\right) ^{1/2}\\
&\qquad \times \left( \int_{\Omega'^{k-1}} \Delta_{\Omega'}^{2\vs}(\lambda-\lambda_1)\cdots \Delta_{\Omega'}^{2\vs}(\lambda_{k-2}-\lambda_{k-1})\,\dd (\lambda_1,\dots,\lambda_{k-1})\right) ^{1/2}\\
&= \abs{\tau(f)^2 \Delta_{\Omega'}^{-2\vs}}^{*k}(\lambda)^{1/2} (\Delta_{\Omega'}^{2 \vs})^{*k}(\lambda)^{1/2}
\end{split}
\]
for every $\lambda\in \Omega'$. In addition,~\cite[Corollary 2.21]{CalziPeloso} shows that
\[
(\Delta_{\Omega'}^{2 \vs})^{*k}=\frac{\Gamma_{\Omega'}(2\vs-\vd)^{k}}{\Gamma_{\Omega'}(2 k \vs-(k-1)\vd)} \Delta^{2 k\vs-(k-1)\vd}_{\Omega'}
\]
on $\Omega'$. Therefore, $\tau(f^k)\in L^2(\Omega', \Delta_{\Omega'}^{(k-1)\vd-2k \vs}\cdot \Hc^m)$ and
\[
\norm{\tau(f^k)}_{L^2(\Omega', \Delta_{\Omega'}^{(k-1)\vd-2k \vs}\cdot \Hc^m)}\meg \frac{\Gamma_{\Omega'}(2\vs-\vd)^{k/2}}{(2 \pi)^{(k-1)m}\Gamma_{\Omega'}(2k \vs-(k-1)\vd)^{1/2}} \norm{\tau(f)}_{L^2(\Omega', \Delta_{\Omega'}^{-2\vs}\cdot \Hc^m)}^k.
\]
Hence,
\[
f^k\in \widetilde A^{2,2}_{k \vs-(k-1)\vd/2}(D)=A^{2,2}_{k \vs-(k-1)\vd/2}(D)
\]
since $k \vs-(k-1)\vd/2\in \frac 1 4 \vect m+(\R_+^*)^r$, so that $f\in A^{2k,2k}_{\vs-(k-1)\vd/(2k)}(D)$. Hence, there is a continuous embedding $\widetilde A^{2,2}_{\vs}(D)\subseteq A^{2k, 2k}_{\vs-(k-1)\vd/(2k)}(D)$. Since there is a continuous embedding $\widetilde A^{2,2}_{\vs}(D)\subseteq A^{\infty,\infty}_{\vs-\vd/2}(D)$ by~\cite[Definition 5.3]{CalziPeloso}, the assertion follows.
\end{proof}

\section{Reverse Carleson and Sampling measures for
$A^{p,p}_\vs(D)$}\label{sec:6} 

Recall that a measure $\mi\in \cM_+(D)$ is $p$-sampling for
$A^{p_1,q_1}_\vs$ if this latter space embeds  as a closed subspace of
$L^p(\mi)$. As observe in~\cite{Luecking6}, replacing $\mi$ with a
suitable integral of positive measures (for instance, but not
necessarily, a disintegration of $\mi$ along $\rho$ with respect to
some of its image measures), it is also possible to extend this
definition to mixed-norm Lebesgue spaces. Even though we prefer to
avoid such technicalities, we present here a result in this spirit,
which shows how one may construct sampling measures out of
lattices. See~\cite[Theorem 3.22]{CalziPeloso} for a proof of a
stronger version of the following result. 

\begin{prop}
Take $p,q\in ]0,\infty]$, $\vs\in \R^r$, and $R_0>0$. Then, there is $\delta_0>0$ such that, for every $(\delta,R)$-lattice $(\zeta_{j,k},z_{j,k})_{j\in J,k\in K}$ on $D$, with $\delta\in ]0,\delta_0]$ and $R\in ]1,R_0]$, the mapping
\[
S\colon \Hol(D)\ni f\mapsto \big(\Delta_\Omega^{\vect s-(\vect b+\vect d)/p}(\rho(\zeta_{j,k},z_{j,k})) f(\zeta_{j,k},z_{j,k})  \big) \in \C^{J\times K}
\]
induces isomorphisms of $A^{p,q}_\vs$ and $A^{p,q}_{\vs,0}$ onto closed subspaces of $\ell^{p,q}(J,K)$ and $\ell^{p,q}_0(J,K)$, respectively. In addition, $A^{\infty,\infty}_{\vs-(\vect b+\vect d)/p}\cap S^{-1}(\ell^{p,q}(J,K))=A^{p,q}_\vs$ and $A^{\infty,\infty}_{\vs-(\vect b+\vect d)/p}\cap S^{-1}(\ell^{p,q}_0(J,K))=A^{p,q}_{\vs,0}$. 
\end{prop}

The next result provides a necessary condition for a measure $\mi$ for
which the function $M_R (\mi)$ is in some mixed-normed weighted
Lebesgue space on $D$ to be a sampling measure. 
It extends~\cite[Theorem 4.3]{Luecking2}, which deals with the case
in which $D$ is the unit disc in $\C$. 

\begin{prop}\label{prop:32}
Take $p_1,q_1,p\in ]0,\infty]$, with $p<\infty$, and $\vs\in
\frac{1}{2 q_1}\vect m+(\R_+^*)^r$ if $q_1<\infty$, while $\vs\in
\R_+^r$ if $q_1=\infty$. Let $p^*,q^*$, and $\vs^*$ be as in~\eqref{p*q*s'}.
Then, for every $R,C,C'>0$, there are $R',C''>0$ such that for every
$\mi\in\cM_+(D)$ such that  
\[
\norm{f}_{A^{p_1,q_1}_{\vs}(D)}\meg C \norm{f}_{L^{p}(\mi)}
\]
for every $f\in A^{p_1,q_1}_{\vs}(D)$, and such that
\[
\norm{M_R (\mi)}_{L^{p^*,q^*}_{\vs^*}(D)}\meg C',
\]
one has $p_1,q_1\meg p$ and
\[
M_{R'}(\mi)(\zeta,z)\Meg C'' \Delta_\Omega^{p[\vs-(\vb+\vd)/p_1]}(\rho(\zeta,z))
\]
for every $(\zeta,z)\in D$.
\end{prop}

Notice that, if $\norm{M_R (\mi)}_{L^{p^*,q^*}_{\vs^*}(D)}$ is finite, then $\mi$ is a $p$-Carleson measure for $A^{p_1,q_1}_\vs$, thanks to Proposition~\ref{prop:23}. Hence, a measure $\mi$ as above is indeed a $p$-sampling measure for $A^{p_1,q_1}_\vs$.  

\begin{proof}
Observe that, by inspection of the proof of Proposition~\ref{prop:23},
it is readily verified that for every $R'>0$ there is a constant
$C_{R'}>0$ such that  
\[
\norm{f}_{L^{p}(\mi)}\meg C_{R'}\norm{M_R(\mi)}_{L^{p^*,q^*}_{\vs^*}(D)}^{1/p}\norm{\chi_{B(\Supp{\mi},R')}f}_{L^{p_1,q_1}_{\vs}(D)}
\]
for every $f\in \Hol(D)$ and for every positive Radon measure $\mi$ on $D$, where 
\[
B(\Supp{\mi},R')\coloneqq \bigcup_{(\zeta,z)\in \Supp{\mi}} B((\zeta,z),R').
\]
Therefore, for every $(\zeta,z)\in D$ and for every $R'>0$,
\[
\norm{\chi_{D\setminus B((\zeta,z), 2 R')}f}_{L^{p}(\mi)}\meg C_{R'}\norm{M_R(\mi)}_{L^{p^*,q^*}_{\vs^*}(D)}^{1/p}\norm{\chi_{D\setminus B((\zeta,z),R')}f}_{L^{p_1,q_1}_{\vs}(D)}
\]
for every $f\in \Hol(D)$ and for every positive Radon measure $\mi$ on $D$. Now, take $\vect{s'}\in \R^r$ so that $B^{\vect{s'}}_{(\zeta,z)}\in A^{p_1,q_1}_{\vs}(D)$ for every $(\zeta,z)\in D$, and observe that there is a constant $C'''>0$ such that
\[
\norm{B^{\vect{s'}}_{(\zeta,z)}}_{A^{p_1,q_1}_{\vs}(D)}=C''' \Delta_\Omega^{\vs+\vect {s'}-(\vb+\vd)/p_1}(\rho(\zeta,z))
\]
for every $(\zeta,z)\in D$ (cf.~\cite[Proposition 2.41]{CalziPeloso}). In addition, for every $R'>0$ there is a constant $C^{(4)}_{R'}>0$ such that
\[
\frac{1}{C^{(4)}_{R'}}\Delta_\Omega^{\vect{s'}}(\rho(\zeta,z))\meg \abs{B^{\vect{s'}}_{(\zeta,z)}(\zeta',z')}\meg C^{(4)}_{R'}\Delta_\Omega^{\vect{s'}}(\rho(\zeta,z)) 
\]
for every $(\zeta,z),(\zeta',z')\in D$ such that $d((\zeta,z),(\zeta',z'))\meg R'$ (cf.~\cite[Theorem 2.47]{CalziPeloso}).
Therefore,
\[
\begin{split}
&  (C^{(4)}_{2R'})^{p}\Delta_\Omega^{p\vect{s'}}(\rho(\zeta,z))M_{2R'}(\mi)(\zeta,z) \\
& \quad \Meg 
\int_{B((\zeta,z),2R')} \abs{ B^{\vect{s'}}_{(\zeta,z)} }^{p}\,\dd \mi\\
&\quad  =   \norm{B^{\vect{s'}}_{(\zeta,z)}}_{L^{p}(\mi)}^{p}- \norm{\chi_{D\setminus B((\zeta,z),2 R')}  B^{\vect{s'}}_{(\zeta,z)} }^{p}_{L^{p}(\mi)}\\
&  \quad 
\Meg \left(\frac{C'''}{C}\Delta_\Omega^{\vs+\vect{s'}-(\vb+\vd)/p_1}(\rho(\zeta,z))\right)^{p}-C' C_{R'}^{p} \norm{\chi_{D\setminus B((\zeta,z),R')} B^{\vect{s'}}_{(\zeta,z)}}_{L^{p_1,q_1}_{\vs}(D)}^{p}
\end{split}
\]
for every $(\zeta,z)\in D$, for every $R'>0$, and for every $\mi$ as in the statement. Now, observe that, by homogeneity, setting 
\[
C^{(5)}_{R'}\coloneqq\norm{\chi_{D\setminus B((0,ie_\Omega),R')} B^{\vect{s'}}_{(0, ie_\Omega)}}_{L^{p_1,q_1}_{\vs}(D)},
\]
one has
\[
\norm{\chi_{D\setminus B((\zeta,z),R')} B^{\vect{s'}}_{(\zeta,z)}}_{L^{p_1,q_1}_{\vs}(D)}= C^{(5)}_{R'} \Delta_\Omega^{\vs+\vs^*-(\vb+\vd)/p_1}(\rho(\zeta,z))
\]
for every $(\zeta,z)\in D$ (cf.~Lemma~\ref{lem:5}). Therefore,
\[
M_{2 R'}(\mi)(\zeta,z)\Meg (C^{(4)}_{2R'})^{-p} \left( \left(\frac{C'''}{C}\right)^{p}- C' (C_{R'} C^{(5)}_{R'})^{p}  \right) \Delta_\Omega^{p[\vs-(\vb+\vd)/p_1]}(\rho(\zeta,z)) 
\]
for every $(\zeta,z)\in D$ and for every $R'>0$. 

Now, observe that the function $(\Delta_\Omega^{p[(\vect b+\vect d)/p_1-\vs]}\circ \rho) M_{R'}$ is both bounded from below and in $L^{p^*,q^*}_{\vs^*+p[\vs-(\vect b+\vect d)/p_1]  }$, thanks to Lemma~\ref{lem:1}. It then follows easily that $p^*=q^*=\infty$, that is, $p_1,q_1\meg p$, in which case $\vs^*=p[(\vect b+\vect d)/p_1-\vs]$. The proof is complete.  
\end{proof}

In the next result we establish a necessary and sufficient condition
in order for a $\nu_D$-measurable set $G$ to be a dominant (or sampling) set (recall
Definition~\ref{def:dom-set}), that is, for the measure $\chi_G(\Delta_\Omega^{p\vs-(\vect b+\vect d)}\circ \rho)\cdot\nu_D$
to be a $p$-sampling measure for $A^{p,p}_\vs$.    It extends~\cite[Main
Theorem]{Luecking}, which deals with the case in which $D$ is the unit
disc in $\C$. See also~\cite[Theorem 1]{Luecking4}, which deals with
weighted Bergman spaces on general homogeneous domains.

\begin{teo}\label{prop:31} 
Take $p\in ]0,\infty[$ and $\vs\in \frac{1}{2 p}\vect
m+(\R_+^*)^r$. Then,  for every $\nu_D$-measurable subset $G$ of $D$
the following conditions are equivalent: 
\begin{enumerate}
\item[{\em(1)}] there are $R,C>0$ such that, for every  $(\zeta,z)\in D$,
\[
\norm*{\chi_{B((\zeta,z),R)}}_{L^{p,p}_{\vs}(D)}\meg C
\norm*{\chi_{G\cap B((\zeta,z),R)}}_{L^{p,p}_{\vs}(D)};
\]
\item[{\em(2)}] there exists $C'>0$ such that, for every $f\in A^{p,p}_{\vs}(D)$,
\[
\norm*{f}_{A^{p,p}_{\vs}(D)}\meg C' \norm*{\chi_{G}f}_{L^{p,p}_{\vs}(D)} ;
\]  
\item[{\em(3)}]  there are $  R' ,C''>0$ such that $\nu_D(G\cap
B((\zeta,z),  R' ))\Meg C'' $ for every $(\zeta,z)\in D$. 
\end{enumerate}
\end{teo}

Notice that condition (3) depends neither on $\vs$, nor on $p$. 
Before we pass to the proof, we need to establish some lemmas.

\begin{lem}\label{lem:12}
Fix $R>0$, and define, for every $\eps>0$  and for every $f\in A^{p,p}_{\vs}(D)$,
\[
A_{f,\eps}\coloneqq \Set{(\zeta,z)\in D\colon \abs{f(\zeta,z)}\meg \eps \left(\dashint_{B((\zeta,z),R)} \abs{f}^p\,\dd \nu_D\right)^{1/p}  }.
\]
Then, there is a constant $C>0$ such that 
\[
\norm{\chi_{A_{f,\eps}} f  }_{L^{p,p}_{\vs}(D)}\meg C\eps \norm{f  }_{L^{p,p}_{\vs}(D)}
\]
for every $f\in A^{p,p}_{\vs}(D)$ and for every $\eps>0$.
\end{lem}

This extends~\cite[Lemma 2]{Luecking2}, which deals with the case in which $D$ is the unit disc in $\C$.

\begin{proof}
Observe that, by Fubini's theorem,
\[
\begin{split}
\norm{\chi_{A_{f,\eps}} f  }_{L^{p,p}_{\vs}(D)}^p&\meg \eps^p \int_{A_{f,\eps}} \dashint_{B((\zeta,z),R)} \abs{f}^p\,\dd  \nu_D\, \Delta_\Omega^{p\vs-(\vb+\vd)}(\rho(\zeta,z))\,\dd \nu_D(\zeta,z)\\
&\meg \eps^p\int_D \abs{f(\zeta',z')}^p \dashint_{B((\zeta',z'),R)}  \Delta_\Omega^{p\vs-(\vb+\vd)}(\rho(\zeta,z)) \,\dd \nu_D(\zeta,z)\,\dd  \nu_D(\zeta',z').
\end{split}
\]
Now, observe that, by homogeneity,
\[
\begin{split}
\dashint_{B((\zeta',z'),R)}  \Delta_\Omega^{p\vs-(\vb+\vd)}(\rho(\zeta,z)) \,\dd \nu_D(\zeta,z)&= C'\Delta_\Omega^{p\vs-(\vb+\vd)}(\rho(\zeta',z'))  
\end{split}
\]
for a suitable constant $C'>0$. The assertion follows.
\end{proof}

\begin{lem}\label{lem:13}
Take $R,\eps,p\in ]0,\infty[$. Then, for every $\delta>0$ there is $\lambda>0$ such that, if we define
\[
E_{f,\lambda}(\zeta,z)\coloneqq \Set{(\zeta',z')\in B((\zeta,z),R)\colon \abs{f(\zeta',z')}>\lambda \abs{f(\zeta,z)} },
\]
then
\[
\nu_D( E_{f,\lambda}(\zeta,z)  )\Meg (1-\delta) \nu_D(B((\zeta,z),R))
\]
for every $(\zeta,z)\in D$ and for every $f\in \Hol(B((\zeta,z),R))$ such that
\[
\abs{f(\zeta,z)}\Meg \eps \left( \dashint_{B((\zeta,z),R)} \abs{f}^p\,\dd \nu_D \right)^{1/p}.
\]
\end{lem}

This extends~\cite[Lemma 2]{Luecking3}, which deals with the case in which $D$ is the unit disc in $\C$.

\begin{proof}
Observe that, by homogeneity, we may reduce to proving the assertion for $(\zeta,z)=(0,i e_\Omega)$. Then, assume by contradiction that there are $\delta>0$ and a sequence $(f_j)_{j\in\N}$ of elements of $\Hol(D)$ such that
\[
\abs{f_j(0,i e_\Omega)}\Meg \eps \left( \dashint_{B((\zeta,z),R)} \abs{f_j}^p\,\dd \nu_D \right)^{1/p}
\]
and such that
\[
\nu_D(E_{f_j,2^{-j}})<\delta\nu_D(B((0,i e_\Omega),R))
\]
for every $j\in \N$. Observe that, up to multiplying each $f_j$ by a suitable constant, we may assume that  $\dashint_{B((\zeta,z),R)} \abs{f_j}^p\,\dd \nu_D=1$ for every $j\in \N$, so that $\abs{f_j(0,i e_\Omega)}\Meg \eps$ for every $j\in\N$. It is then readily verified that the sequence $(f_j)$ is bounded in $\Hol(B((0,i e_\Omega),R))$, so that we may assume that it converges locally uniformly to some $f$. Then, take $R'\in ]0,R[$ so that $\nu_D(B((0,i e_\Omega),R'))\Meg \frac \delta 2\nu_D(B((0,ie_\Omega),R))$, and observe that
\[
\nu_D(\Set{(\zeta,z)\in B((0,i e_\Omega),R')\colon \abs{f_j(\zeta,z)}\meg 2^{-j}  \abs{f_j(0,i e_\Omega)} })\Meg \frac \delta 2\nu_D(B((0,ie_\Omega),R))
\]
for every $j\in \N$.
Observe that $(f_j)$ converges uniformly to $f$ on $B((0,i e_\Omega),R')$ (cf.~\cite[Proposition 2.44]{CalziPeloso}), so that for every $j_0\in\N$ there is $j_1\in \N$ such that
\[
\abs{f_j-f}\meg 2^{-j_0}
\]
on $B((0,i e_\Omega),R')$for every $j\Meg j_1$. Hence,
\[
\begin{split}
&\Set{(\zeta,z)\in B((0,i e_\Omega),R')\colon \abs{f(\zeta,z)}\meg 2^{-j_0}+2^{-j}  \abs{f_j(0,i e_\Omega)} }\\
&\qquad \qquad\qquad \qquad\qquad \qquad\qquad \qquad\supseteq \Set{(\zeta,z)\in B((0,i e_\Omega),R')\colon \abs{f_j(\zeta,z)}\meg 2^{-j}  \abs{f_j(0,i e_\Omega)} }
\end{split}
\]
for every $j\Meg j_1$, so that
\[
\nu_D(\Set{(\zeta,z)\in B((0,i e_\Omega),R')\colon \abs{f(\zeta,z)}\meg 2^{-j_0}+2^{-j}  \abs{f_j(0,i e_\Omega)} })\Meg \frac\delta 2\nu_D(B((0,ie_\Omega),R)).
\]
By the arbitrariness of $j_0$, we then see that
\[
\begin{split}
\nu_D(\Set{(\zeta,z)\in B((0,i e_\Omega),R')\colon f(\zeta,z)=0})&=\lim_{\eta\to 0^+} \nu_D(\Set{(\zeta,z)\in B((0,i e_\Omega),R')\colon \abs{f(\zeta,z)}\meg \eta })\\
&\Meg  \frac\delta 2\nu_D(B((0,ie_\Omega),R)).
\end{split}
\]
Then, $f$ vanishes on a non-$\nu_D$-negligible set, so that it vanishes identically on the connected set $B((0,i e_\Omega),R)$ by holomorphy.\footnote{Observe that the Bergman distance on $D$ is induced by a complete Riemannian metric by~\cite[Proposition 2.44]{CalziPeloso}, so that $B((0,i e_\Omega),R)=\exp_{(0,i e_\Omega)}(B(0,R))$ by the Hopf--Rinow theorem (cf., e.g.,~\cite[the proof of Theorem 6.6]{Lang}).   Alternatively, every element of $B((0, i e_\Omega),R)$ may be connected to $(0,i e_\Omega)$ by a minimizing geodesic, whose image therefore lies entirely in $B((0,i e_\Omega),R)$.} Nonetheless, $\abs{f(0,i e_\Omega)}\Meg \eps$: contradiction.
\end{proof}

\begin{proof}[Proof of Theorem~\ref{prop:31}.]
By means of~\cite[Corollary 2.49]{CalziPeloso}, condition
(1) is readily seen to be equivalent to condition (3).  

(2) $\implies$ (1). This follows from Proposition~\ref{prop:32}.

(1) $\implies$ (2). With the notation of Lemma~\ref{lem:12}, choose $\eps>0$  so small that
\[
\norm{f}_{A^{p,p}_{\vs}(D)}\meg 2\norm{\chi_{D\setminus A_{f,\eps}}f}_{L^{p,p}_{\vs}(D)}
\]
for every $f\in A^{p,p}_{\vs}(D)$. In addition, with the notation of Lemma~\ref{lem:13}, choose $\lambda>0$ so that
\[
\norm{\chi_{B((\zeta,z),R)\setminus E_{f,\lambda}(\zeta,z)}}_{L^{p,p}_{\vs}(D)}\meg \frac{1}{2^{1/p} C'}\norm{\chi_{B((\zeta,z),R)}}_{L^{p,p}_{\vs}(D)}
\]
for every $f\in A^{p,p}_{\vs}(D)$, and for every $(\zeta,z)\in D\setminus A_{f,\eps}$ (cf.~also~\cite[Corollary 2.49]{CalziPeloso}), so that
\[
\norm{\chi_{G\cap E_{f,\lambda}(\zeta,z)}}_{L^{p,p}_{\vs}(D)}\Meg \frac{1}{2^{1/p} C'} \norm{\chi_{B((\zeta,z),R)}}_{L^{p,p}_{\vs}(D)}.
\]
Now, the definition of $E_{f,\lambda}(\zeta,z)$ implies that
\[
\norm{\chi_{G\cap B((\zeta,z),R)} f}_{L^{p,p}_{\vs}(D)}\Meg \norm{\chi_{G\cap E_{f,\lambda}(\zeta,z)} f}_{L^{p,p}_{\vs}(D)}\Meg \frac{\lambda^p}{2^{1/p} C'} \abs{f(\zeta,z)} \norm{\chi_{B((\zeta,z),R)}}_{L^{p,p}_{\vs}(D)}
\]
for every $f\in A^{p,p}_{\vs}(D)$, and for every $(\zeta,z)\in D\setminus A_{f,\eps}$.
Therefore, by means of Fubini's theorem we see that
\[
\begin{split}
\norm{f}_{A^{p,p}_{\vs}(D)}^p&\meg 2^p\norm{\chi_{D\setminus A_{f,\eps}}f}_{L^{p,p}_{\vs}(D)}^p\\
&\meg \frac{2^{p+1} C'^p}{\lambda^p} \int_G \abs{f(\zeta',z')}^p \int_{B((\zeta',z'),R)} \frac{\Delta_\Omega^{p\vs+\vd}(\rho(\zeta,z))}{\norm{\chi_{B((\zeta,z),R)}}_{L^{p,p}_{\vs}(D)}^p}\,\dd (\zeta,z) \,\Delta_\Omega^{p \vs+\vd}(\rho(\zeta',z'))\,\dd(\zeta',z') 
\end{split}
\]
for every $f\in A^{p,p}_{\vs}(D)$. To conclude, it suffices to observe that the function
\[
(\zeta',z')\mapsto \int_{B((\zeta',z'),R)} \frac{\Delta_\Omega^{p\vs+\vd}(\rho(\zeta,z))}{\norm{\chi_{B((\zeta,z),R)}}^p_{L^{p,p}_{\vs}(D)}}\,\dd (\zeta,z)=\frac{\nu_D(B((\zeta',z'),R))  }{\norm{\chi_{B((0,i e_\Omega),R)}}^p_{L^{p,p}_{\vs}(D)}}
\]
is constant, by homogeneity (cf.~Lemma~\ref{lem:5}).
\end{proof}

\begin{teo}\label{prop:30}
Take $p\in ]0,\infty[$, $\vs\in \frac{1}{2 p}\vect m+(\R_+^*)^r$,
$\eps,C',R_0>0$. Then, there are $R_1,C>0$ such that, for every $R\in
]0,R_1]$, for every $\mi\in\cM_+(D)$ such that 
\[
N(\mi)\coloneqq \norm{M_R(\mi)}_{L^{\infty,\infty}_{\vb+\vd-p \vs}(D)}<\infty,
\]
and such that
\[
\nu_D(G_\mi\cap B((\zeta,z),R_0))\Meg C'
\]
for every $(\zeta,z)\in D$,
where
\[
G_\mi\coloneqq \Set{(\zeta,z)\in D\colon \Delta_\Omega^{\vb+\vd-p\vs}
(\rho(\zeta,z))M_R(\mi)(\zeta,z)\Meg \eps N(\mi)}, 
\]
one has
\[
\norm{f}_{A^{p,p}_{\vs}(D)}\meg \frac{C R^{(2 n+2m)/p}}{N(\mi)} \norm{ f}_{L^{p}(\mi)}
\]
for every $f\in A^{p,p}_{\vs}(D)$
\end{teo}

This extends~\cite[Theorem 4.2]{Luecking2}, which deals with the case in which $D$ is the unit disc in $\C$.

Before we pass to the proof, we need another result.

\begin{prop}\label{prop:28}
Take $p\in ]0,\infty]$. Then, there are two constants $C,R_0>0$ such
that the following hold. For every
$\vect{s_1},\vect{s_2},\vect{s_3}\in \R^r$, for every $R,R'\in
]0,R_0]$, for every $\mi_1,\mi_2\in\cM_+(D)$ such that
\[
C_1\coloneqq \sup_{(\zeta,z)\in D}
\Delta_\Omega^{-\vect{s_1}}(\rho(\zeta,z)) M_R(\mi_1)(\zeta,z)\qquad
\text{and}\qquad C_2\coloneqq \sup_{(\zeta,z)\in D}
\Delta_\Omega^{-\vect{s_2}}(\rho(\zeta,z)) M_{R+R'}(\mi_2)(\zeta,z) 
\]
are finite, and for every $f\in \Hol(D)$,
\[
\int_{d((\zeta,z),(\zeta',z'))< R}\frac{\abs{f(\zeta,z)-f(\zeta',z')}^p}{\Delta_\Omega^{\vect{s_1}-\vect{s_3}}(\rho(\zeta',z'))} \,\dd (\mi_1\otimes \mi_2)((\zeta,z),(\zeta',z'))\meg \frac{C C_1 C_2 R^p }{R'^{p+2 n+2m}}   \int_D   \abs{f}^p (\Delta_\Omega^{\vect{s_2}+\vect{s_3}}\circ \rho)\,\dd \nu_D,
\]
\end{prop}

This extends~\cite[Theorem 2.3]{Luecking2}, which deals with the case in which $D$ is the unit disc in $\C$.

\begin{proof}
By~\cite[Lemmas 3.24 and 3.25]{CalziPeloso}, there are two constants $C_1,R_0>0$ such that for every $R,R'\in ]0,R_0]$, for every $f\in \Hol(D)$, and for every $(\zeta,z),(\zeta',z')\in D$ such that $d((\zeta,z),(\zeta',z'))\meg R$,
\[
\abs{f(\zeta,z)-f(\zeta',z')}^p\meg C'\frac{R^p}{R'^{p+2n+2m}} \int_{B((\zeta',z'),R+R')} \abs{f}^p\,\dd \nu_D.
\] 
Therefore, by means of~\cite[Corollary 2.49]{CalziPeloso}, we see that there is a constant $C''>0$ such that
\[
\abs{f(\zeta,z)-f(\zeta',z')}^p\Delta_\Omega^{\vect{s_3}}(\rho(\zeta',z'))\meg C''\frac{R^p}{R'^{p+2n+2m}} \int_{B((\zeta',z'),R+R')} \abs{f}^p(\Delta_\Omega^{\vect{s_3}}\circ \rho)\,\dd \nu_D.
\]
for every $R,R'\in ]0,R_0]$, for every $f\in \Hol(D)$, and for every $(\zeta,z),(\zeta',z')\in D$ such that $d((\zeta,z),(\zeta',z'))\meg R$.
Then, integrating in $(\zeta,z)$ with respect to $\mi_1$,
\[
\begin{split}
&\Delta_\Omega^{\vect{s_3}}(\rho(\zeta',z'))\int_{ B((\zeta',z'),R)} \abs{f(\zeta,z)-f(\zeta',z')}^p\,\dd \mi_1(\zeta,z)\\
&\qquad\meg C'' \frac{R^p}{R'^{p+2n+2m}} M_R(\mi_1)(\zeta',z') \int_{B((\zeta',z'),R+R')} \abs{f}^p(\Delta_\Omega^{\vect{s_3}}\circ \rho)\,\dd \nu_D\\
&\qquad\meg C'' \frac{R^p}{R'^{p+2n+2m}}C_1
\Delta_\Omega^{\vect{s_1}}(\rho(\zeta',z')) \int_{B((\zeta',z'),R+R')}
\abs{f}^p(\Delta_\Omega^{\vect{s_3}}\circ \rho)\,\dd \nu_D 
\end{split}
\]
so that, integrating in $(\zeta',z')$ with respect to the measure $\mi_2$,
\[
\begin{split}
&\int_{d((\zeta,z),(\zeta',z'))< R} \abs{f(\zeta,z)-f(\zeta',z')}^p\Delta_\Omega^{\vect{s_3}-\vect{s_1}}(\rho(\zeta',z')) \,\dd (\mi_1\otimes \mi_2)((\zeta,z),(\zeta',z'))\\
&\qquad \qquad\meg \frac{C_1 C_2 C'' R^p}{ R'^{p+2n+2m}}  \int_D   \abs{f}^p (\Delta_\Omega^{\vect{s_2}+\vect{s_3}}\circ \rho)\,\dd \nu_D,
\end{split}
\]
whence the result.
\end{proof}

\begin{proof}[Proof of Theorem~\ref{prop:30}.]
Apply Proposition~\ref{prop:28} with $\mi_1=\mi $, $\mi_2=\nu_D$,
$\vect{s_1}=\vect{s_3}=p \vs-(\vb+\vd)$,  and $\vect{s_2}=\vect 0$.
Then, we find $R_0>0$ and $C_1>0$ such that
\[
\int_{d((\zeta,z),(\zeta',z'))< R} \abs{f(\zeta,z)-f(\zeta',z')}^p
\,\dd (\nu_D\otimes \mi)((\zeta,z),(\zeta',z'))\meg
\frac{C_1R^{p}}{R_0^{p+2 n+2 m}}
N(\mi) \norm{f}_{A^{p,p}_{\vs}(D)}^p
\]
for every $R\in ]0,R_0]$ and for every $f\in \Hol(D)$.
Therefore,
\[
\begin{split}
&\left( \int_{d((\zeta,z),(\zeta',z'))< R} \abs{f(\zeta,z)}^p\,\dd (\nu_D\otimes \mi)((\zeta,z),(\zeta',z'))\right)^{1/\max(1,p)}\\
&\qquad\meg \left( \int_{d((\zeta,z),(\zeta',z'))< R}\abs{f(\zeta',z')}^p \,\dd (\nu_D\otimes \mi)((\zeta,z),(\zeta',z'))\right) ^{1/\max(1,p)}\\
&\qquad\qquad+ \left( \frac{C_1 R^{p}}{ R_0^{p+2 n+2 m}} N(\mi) \norm{f}_{A^{p,p}_{\vs}(D)}^p\right) ^{1/\max(1,p)}
\end{split}
\]
for every $f\in \Hol(D)$. Now, observe that Theorem~\ref{prop:31} implies that there is a constant $C_2>0$ such that
\[
\begin{split}
\int_{d((\zeta,z),(\zeta',z'))< R} \abs{f(\zeta,z)}^p\,\dd (\nu_D\otimes \mi)((\zeta,z),(\zeta',z'))
&= \int_D \abs{f(\zeta,z)}^p M_R(\mi)(\zeta,z)\,\dd \nu_D(\zeta,z) \\
&\Meg \eps N(\mi)\int_{G_\mi} \abs{f}^p( \Delta_\Omega^{p\vs-(\vb+\vd)}\circ \rho)\,\dd \nu_D\\
&\Meg C_2 \eps N(\mi) \norm{f}_{A^{p,p}_{\vs}(D)}^p,
\end{split}
\]
for every $f\in A^{p,p}_{\vs}(D)$,
while clearly
\[
\begin{split}
\int_{d((\zeta,z),(\zeta',z'))< R}\abs{f(\zeta',z')}^p \,\dd (\nu_D\otimes \mi)((\zeta,z),(\zeta',z'))&= \nu_D(B((0,i e_\Omega),R)  )\int_D \abs{f}^p\,\dd \mi,
\end{split}
\]
for every $f\in A^{p,p}_{\vs}(D)$.
Therefore, 
\[
\begin{split}
\left( C_2\eps N(\mi) \norm{f}_{A^{p,p}_{\vs}(D)}^p\right)^{1/\max(1,p)}
&\meg\left( \nu_D(B((0,i e_\Omega),R)  )\int_D \abs{f}^p\,\dd \mi\right) ^{1/\max(1,p)}\\
&\qquad+ \left( \frac{C_1 R^{p}}{R_0^{p+2 n+2 m}} N(\mi)\norm{f}_{A^{p,p}_{\vs}(D)}^p\right) ^{1/\max(1,p)}
\end{split}
\]
for every $f\in A^{p,p}_{\vs}(D)$. It then follows that
\[
C_3N(\mi)\norm{f}_{A^{p,p}_{\vs}(D)}^p\meg \nu_D(B((0,i e_\Omega),R)  )\int_D \abs{f}^p\,\dd \mi
\]
for every $f\in A^{p,p}_{\vs}(D)$, where
\[
C_3\coloneqq \left( \left( C_2\eps \right)^{1/\max(1,p)}-    \left( \frac{C_1 R^{p}}{R_0^{p+2 n+2 m}}  \right) ^{1/\max(1,p)}  \right)^{\max(1,p)}.
\]
Since there is $R_1\in ]0, R_0]$ such that $C_3$ is well defined and $>0$ whenever $R\in ]0,R_1]$, the assertion follows.
\end{proof}

\begin{deff}\label{def:1}
Given $\vs\in \R^r$ and $\mi\in\cM_+(D)$, we shall denote with
$W_{\vs}(\mi)$ the vague closure of the set of measures of the
form 
\[
\Delta_\Omega^{\vs}(\rho(\zeta,z))(\phi_{(\zeta,z)})_*(\mi),
\]
for every $(\zeta,z)\in D$, where $\phi_{(\zeta,z)}$ is an affine
automorphism of $D$ of the form $(\zeta',z')\mapsto (\zeta,\Rea
z+i\Phi(\zeta))\cdot (g\zeta', t\cdot z)$, with $t\in T_+$, $g\in
GL(E)$, and $t\cdot \Phi=\Phi\circ (g\times g)$.\footnote{Notice that
$\phi_{(\zeta,z)}$ is not uniquely determined by $(\zeta,z)$ unless
$n=0$, so that this definition may depend on the choice of the
automorphisms $\phi_{(\zeta,z)}$.} 
\end{deff}

\begin{teo}\label{prop:27}
Take $p,q\in ]0,\infty[$, with $q<p$, $\vs\in\R^r$, and
$\mi\in\cM_+(D)$ such that the following hold: 
\begin{enumerate}
\item[{\em(1)}] $\vs\in \frac{1}{2 p}\vect m+\frac{q}{p(p-q)}\vect{m'}+(\R_+^*)^r$;

\item[{\em(2)}]  $M_1(\mi) \in L^{\infty,\infty}_{\vb+\vd-p\vs}(D)$;

\item[{\em(3)}]   the support of   every element of $W_{p\vs-(\vb+\vd)}(\mi)$  is a set
of uniqueness for $A^{q,q}_{(p/q)\vs}(D)$.  
\end{enumerate}
Then, the canonical mapping $A^{p,p}_{\vs}(D)\to L^{p}(\mi)$ is an
isomorphism onto its image. 
\end{teo}

This extends the implication (b)$\implies$(a) in~\cite[Theorem 5]{Luecking5}, which deals with the case in which $D$ is the unit disc in $\C$. Notice that, as shown in~\cite[Theorem 5]{Luecking5}, condition~(3) holds for some $q<p$ if $A^{p,p}_{\vs}(D)$ embeds as a closed subspace of $L^p(\mi)$ and  $D$ is the unit disc in $\C$.

Before we pass to the proof, we need some lemmas.

\begin{lem}\label{lem:10}
Take $p_1,q_1,p\in ]0,\infty]$, with $p<\infty$, and $\vs\in
\frac{1}{2 q_1}\vect m+(\R_+^*)^r$ and let $p^*,q^*$ and $\vs^*$ be as
in~\eqref{p*q*s'}.  Let $\Ms$ be a set of Radon measures
on $D$ such that  
\[
\sup_{\mi\in M}\norm{M_R(\mi)}_{L^{p^*, q^*}_{\vs^*}(D)}<\infty.
\]
Then, 
\[
\lim_{\mi,\Ff}\,  \norm{f}_{L^{p}(\mi)}= \norm{f}_{L^{p}(\mi_0)}
\]
for every filter $\Ff$ on $\Ms$ which converges vaguely to some  Radon
measure $\mi_0$ on $D$, and for every $f\in A^{p_1,q_1}_{\vs,0}(D)$.  
\end{lem}

This extends~\cite[Theorem 1]{Luecking5}, which deals with the case in which $D$ is the unit disc in $\C$.

\begin{proof}
By inspection of the proof of Proposition~\ref{prop:23}, it is readily verified that there is a constant $C'>0$ such that
\[
\norm{\chi_{B((0,i e_\Omega), R')}f}_{L^{q_2}(\mi)}\meg C' \norm{\chi_{B((0,i e_\Omega), R'+1)} f}_{L^{q_1,q_2}_{\vs}(D)}
\]
for every $\mi\in M$, for every $f\in A^{p_1,q_1}_{\vs,0}(D)$, and for every $R'>0$. The   assertion follows easily. 
\end{proof}

\begin{lem}\label{lem:14}
Take $p_1,q_1,p\in ]0,\infty]$, $\vs\in \frac{1}{2 q_1}\vect
m+(\R_+^*)^r$ if $q_1<\infty$ and $\vs\in \R_+^r$ if $q_1=\infty$, and
$\vect{s_2},\vect{s_3}\in \R^r$. Let $p^*,q^*$ and $\vs^*$  be as
in~\eqref{p*q*s'}.
Take $\eps>0$, and define, for every $(\zeta,z)\in D$,
\[
U_\eps(\zeta,z)\coloneqq \Set{f\in \Hol(D)\colon \abs{f(\zeta,z)}\Meg\eps \Delta_\Omega^{(\vb+\vd)/p-\vs-\vect{s_2}}(\rho(\zeta,z))\norm{f B^{\vect{s_2}}_{(\zeta,z)}}_{A^{p_1,q_1}_{\vs}(D)}}.
\]
Take $\mi\in\cM_+(D)$ such that
\[
\norm{M_R(\mi)}_{L^{p^*,q^*}_{\vect {s'}}(D)}<\infty,
\]
and such that the support of every
element of $W_{p[\vs-(\vb+\vd)/p_1]}(\mi)$  is a set of uniqueness for $A^{p_1,q_1}_{\vect {s_1}}(D)$. 

Then, there is  a constant $C>0$ such that
\[
\norm{f B^{\vect{s_2}}_{(\zeta,z)}(\Delta_\Omega^{\vect{s_3}}\circ \rho) 
}_{L^{p}(\mi')}\Meg C\Delta_\Omega^{\vect{s_3}}(\rho(\zeta,z))\norm{f B^{\vect{s_2}}_{(\zeta,z)}}_{A^{p_1,q_1}_{\vs}(D)}
\]
for every $(\zeta,z)\in D$, for every $f\in U_\eps(\zeta,z)$ and for every $\mi'\in W_{p[\vs-(\vb+\vd)/p_1]}(\mi)$.
\end{lem}

This extends~\cite[Lemma 4]{Luecking5}, which deals with the case in which $D$ is the unit disc in $\C$.

\begin{proof}
\textsc{Step I.} We prove the assertion for $(\zeta,z)=(0,i e_\Omega)$.
Define
\[
\mi_{(\zeta,z)}\coloneqq \Delta_\Omega^{p[\vs-(\vb+\vd)/p_1]}(\rho(\zeta,z))(\phi_{(\zeta,z)})_*(\mi)
\]
for every $(\zeta,z)\in D$, where $\phi_{(\zeta,z)}$ is as in Definition~\ref{def:1}.
Observe that it will suffice to prove the assertion with $W_{p[\vect{s_1}-(\vb+\vd)/p_1]}(\mi)$
replaced by $\Set{\mi_{(\zeta,z)}\colon (\zeta,z)\in D}$, thanks to Lemma~\ref{lem:10}.
Then assume, by contradiction, that there are a sequence $(f_j)$ of
elements of $U_\eps(0,i e_\Omega)$, and a sequence $((\zeta_j,z_j))$
of elements of $D$ such that 
\[
\norm{f_j B^{\vect{s_2}}_{(0,i e_\Omega)}}_{A^{p_1,q_1}_{\vs}(D)}=1
\]
for every $j\in \N$, while
\[
\lim_{j\to \infty }\norm{f_j B^{\vect{s_2}}_{(0,i e_\Omega)} (\Delta_\Omega^{\vect{s_3}}\circ \rho) }_{L^{q_2}(\mi_{(\zeta_j,z_j)})}=0.
\]
Observe that 
\[
\norm{M_R(\mi_{(\zeta',z')})}_{L^{p^*,q^*}_{\vs^*}(D)}= \norm{M_R(\mi)}_{L^{p^*,q^*}_{\vs^*}(D)}
\]
for every $(\zeta',z')\in D$ (cf.~Lemma~\ref{lem:5}), so that
$W_{p[\vect {s_1}-(\vb+\vd)/p_1]}(\mi)$ is bounded, hence compact and
metrizable, in the vague topology. Therefore, we may assume that
$(\mi_{(\zeta_j,z_j)})  $ converges vaguely to some (positive Radon)
measure $\mi'$ on $D$. Analogously, we may assume that $(f_j)$
converges locally uniformly to some $f\in A^{p_1,q_1}_{\vs}(D)$, so
that $\abs{f(0,i e_\Omega)}\Meg \eps$. Let $(\psi_k)_{k\in K}$ be a
partition of the unity on $D$ whose elements belong to $C_c(D)$, and
observe that 
\[
\lim_{j\to \infty }\norm{\psi_k^{1/p} f_j  }_{L^{p}(\mi_{(\zeta_j,z_j)})}= \norm{\psi_k^{1/p} f }_{L^{p}(\mi')}
\]
by the previous remarks. Therefore, by Fatou's lemma,
\[
\begin{split}
0&= \lim_{j\to \infty }\norm{f_j B^{\vect{s_2}}_{(\zeta,z)}(\Delta_\Omega^{\vect{s_3}}\circ \rho)  }_{L^{p}(\mi_{(\zeta_j,z_j)})}^{p}\\
&= \lim_{j\to \infty }\sum_{k\in K}\norm{\psi_k^{1/p} f_j B^{\vect{s_2}}_{(\zeta,z)}(\Delta_\Omega^{\vect{s_3}}\circ \rho)  }_{L^{p}(\mi_{(\zeta_j,z_j)})}^{p}\\
&\Meg \sum_{k\in K}\norm{\psi_k^{1/p} f B^{\vect{s_2}}_{(\zeta,z)}(\Delta_\Omega^{\vect{s_3}}\circ \rho) }_{L^{p}(\mi')}^{p}\\
&= \norm{f B^{\vect{s_2}}_{(\zeta,z)}(\Delta_\Omega^{\vect{s_3}}\circ \rho) }_{L^{p}(\mi')}^{p}.
\end{split}
\]
Since the support of $\mi'\in W_{p[\vs-(\vb+\vd)/p_1]}(\mi)$ is a set
of uniqueness for $A^{p_1,q_1}_{\vs}(D)$, this implies that $f=0$,
which is absurd, since $\abs{f(0,i e_\Omega)}\Meg \eps$. 

\textsc{Step II.} We now prove the assertion for general $(\zeta,z)\in D$ and for $\mi'=\mi_{(\zeta',z')}$, $(\zeta',z')\in D$. Define $\psi_{(\zeta,z),(\zeta',z')}\coloneqq  \phi_{(\zeta',z')}\circ \phi_{\phi^{-1}_{(\zeta',z')}(\zeta,z)}$.
Then, take $f\in U_\eps(\zeta,z)$, and observe that $f\circ \psi_{(\zeta,z),(\zeta',z')}\in U_\eps(0,i e_\Omega)$, since $B^{\vect{s_2}}_{(\zeta,z)}\circ \psi_{(\zeta,z),(\zeta',z')}= \Delta_\Omega^{\vect{s_2}}(\rho(\zeta,z)) B^{\vect{s_2}}_{(0,i e_\Omega)}$ (use Lemma~\ref{lem:5} again).
Applying \textsc{step I} to  $f\circ \psi_{(\zeta,z),(\zeta',z')}$ then yields the result, by means of another application of Lemma~\ref{lem:5}.
\end{proof}

\begin{lem}\label{lem:15}
Take $p\in ]1,\infty[$, $\vect{s_1},\vect{s_2},\vect{s_3}\in \R^r$,
and  $\mi_1,\mi_2\in\cM_+(D)$ such that the
following hold: 
\begin{enumerate}
\item[{\em(1)}] the mapping $M_1(\mi_j)\in L^{\infty,\infty}_{\vb+\vd-\vect{s_j}}(D)$ for $j=1,2$;

\item[{\em(2)}] $\frac{1}{p'}\vect{s_1}+\frac 1 p \vect{s_2}=\vb+\vd-\vect{s_3}$;

\item[{\em(3)}] $\frac{1}{p'}\vect{s_1}\in \frac{1}{2 p'}\vect m+\frac{1}{2 p}\vect{m'}+(\R_+^*)^r$;

\item[{\em(4)}] $\frac{1}{p}\vect{s_2}\in \frac{1}{2 p}\vect m+\frac{1}{2 p'}\vect{m'}+(\R_+^*)^r$;
\end{enumerate}
Then, the mapping
\[
T\colon C_c(D) \ni f \mapsto  \int_D f(\zeta,z) \abs{B_{(\zeta,z)}^{\vect{s_3}}}\,\dd \mi_1(\zeta,z) 
\]
induces a continuous linear mapping of $L^p(\mi_1)$ into $L^p(\mi_2)$.
\end{lem}

\begin{proof}
Define
\[
T'\colon C_c(D) \ni f \mapsto  \int_D f(\zeta,z) \abs{B_{(\zeta,z)}^{\vect{s_3}}}\,\dd \mi_2(\zeta,z).
\]
Observe that our assumptions and~\cite[Theorem 2.47 and Corollary
2.49]{CalziPeloso} imply that for every $\vs',\vect{s_5}\in \R^r$
there is a constant $C_1>0$ such that 
\[
T(\Delta_\Omega^{p'\vs'}\circ \rho)(\zeta,z)\meg C_1 \norm{B^{\vect{s_3}}_{(\zeta,z)}}_{A^{1,1}_{\vect{s_1}+p'\vs'}(D)}
\]
and
\[
T(\Delta_\Omega^{p\vect{s_5}}\circ \rho)(\zeta,z)\meg C_1 \norm{B^{\vect{s_3}}_{(\zeta,z)}}_{A^{1,1}_{\vect{s_2}+p\vect{s_5}}(D)}
\]
for every $ (\zeta,z)\in D$. In addition, by~\cite[Proposition 2.41]{CalziPeloso}, there are two constants $C_2,C_3>0$ such that
\[
\norm{B^{\vect{s_3}}_{(\zeta,z)}}_{A^{1,1}_{\vect{s_1}+p'\vs'}(D)}=C_2 \Delta_\Omega^{\vect{s_1}+\vect{s_3}+p'\vs'-(\vb+\vd)}(\rho(\zeta,z))
\]
and
\[
\norm{B^{\vect{s_3}}_{(\zeta,z)}}_{A^{1,1}_{\vect{s_2}+p\vect{s_5}}(D)}=C_3 \Delta_\Omega^{\vect{s_2}+\vect{s_3}+p\vect{s_5}-(\vb+\vd)}(\rho(\zeta,z))
\]
for every $(\zeta,z)\in D$ if and only if the following conditions are satisfied:
\begin{itemize}
\item[(i)] $p'\vs'\in \frac 1 2 \vect m-\vect{s_1}+(\R_+^*)^r$;

\item[(ii)] $\vect{s_3}\in \vb+\vd-\frac1 2 \vect{m'}-(\R_+^*)^r$;

\item[(iii)] $ p' \vs'\in \vb+\vd-\frac 1 2 \vect{m'}-\vect{s_1}-\vect{s_3}-(\R_+^*)^r$;

\item[(iv)] $p\vect{s_5}\in \frac 1 2 \vect m-\vect{s_2}+(\R_+^*)^r$;

\item[(v)] $ p \vect{s_5}\in \vb+\vd-\frac 1 2 \vect{m'}-\vect{s_2}-\vect{s_3}-(\R_+^*)^r$.
\end{itemize}
In addition, 
\[
\vect{s_1}+\vect{s_3}+p'\vs'-(\vb+\vd)= p'\vect{s_5} \qquad \text{and} \qquad \vect{s_2}+\vect{s_3}+p\vect{s_5}-(\vb+\vd)=p \vect{s_4}
\]
if and only if
\[
\vect{s_5}=\vect{s_4}+\frac 1 p (\vb+\vd-\vect{s_2}-\vect{s_3}) \qquad \text{and} \qquad \frac{1}{p'}\vect{s_1}+\frac 1 p \vect{s_2}=\vb+\vd-\vect{s_3}.
\]
If these conditions are satisfied, then conditions (iv) and (v) become:
\begin{itemize}
\item[(iv$'$)] $p \vect{s_4}\in \vect{s_3}-(\vb+\vd)+\frac 1 2 \vect m+(\R_+^*)^r$;

\item[(v$'$)] $ p\vect{s_4}\in -\frac{1}{2}\vect{m'}-(\R_+^*)^r$.
\end{itemize}
By our assumptions, we may take $\vect{s_4}, \vect{s_5}$ such that all the preceding conditions hold, so that the assertion follows by means of Schur's lemma (cf., e.g.,~\cite[Lemma of I.2]{Grafakos}).
\end{proof}

\begin{lem}\label{lem:11}
Take $p,q\in ]0,\infty[$, with $q<p$, and take $\vect{s_1}, \vect{s'_2},\vect{s_3}\in \R^r$ such that the following hold:
\begin{enumerate}
\item $\vect{s_1}-\vect{s_3}\in \left( \frac{1}{2 q}-\frac{1}{2 p}  \right)\vect m+\frac{1}{2 p}\vect{m'}+(\R_+^*)^r$;

\item $\vect{s_1}+\vect{s_2}-\vect{s_3}\in \frac{1}{q}(\vb+\vd)-\frac{1}{2 p}\vect m-\left( \frac{1}{2 q}-\frac{1}{2p}  \right)\vect{m'}-(\R_+^*)^r$.
\end{enumerate}
For every $\eps>0$ and for every $f\in \Hol(D)$, define
\[
B_{\eps,f}^{q,\vect{s_1},\vect{s_2}}\coloneqq \Set{ (\zeta,z)\in D\colon \abs{f(\zeta,z)}\meg \eps  \Delta_\Omega^{(\vb+\vd)/q-\vect{s_1}-\vect{s_2}}(\rho(\zeta,z))\norm{f B^{\vect{s_2}}_{(\zeta,z)}}_{A^{q,q}_{\vect{s_1}}(D)}}.
\] 
Then, there is a constant $C>0$ such that
\[
\norm{\chi_{B_{\eps,f}^{q,\vect{s_1},\vect{s_2}}} f}_{L^{p,p}_{\vect{s_3}}(D)}\meg C\eps \norm{f}_{A^{p,p}_{\vect{s_3}}(D)}
\]
for every $\eps>0$ and for every $f\in A^{p,p}_{\vect{s_3}}(D)$.
\end{lem}

This extends~\cite[Lemma 2]{Luecking5}, which deals with the case in which $D$ is the unit disc in $\C$.

\begin{proof}
It will suffice to prove that the operator
\[
T\colon f \mapsto \left[(\zeta,z)\mapsto
\Delta_\Omega^{\vb+\vd-q_2(\vect{s_1}+\vect{s_2})}(\rho(\zeta,z))
\int_D f(\zeta',z')
\abs{B^{q_2\vect{s_2}}_{(\zeta,z)}(\zeta',z')}\Delta_\Omega^{q_2\vect{s_1}-(\vb+\vd)}(\rho(\zeta',z'))\,\dd
\nu_D(\zeta',z')   \right] 
\]
induces a continuous linear mapping  of
$L^{p/q,p/q}_{q_2\vect{s_3}}(D)$ into itself. This follows from
Lemma~\ref{lem:15}. 
\end{proof}

\begin{proof}[Proof of Theorem~\ref{prop:27}.]

Set $\ell\coloneqq p/q$ and $\vect{s''}\coloneqq \ell \vs$.
Observe that Lemma~\ref{lem:14} implies that there is a constant $C'>0$ such that
\[
\norm{f B^{\vect{s'}}_{(\zeta,z)}
}_{L^{q}(\mi)}\Meg C'\norm{f B^{\vect{s'}}_{(\zeta,z)}}_{A^{q,q}_{\vect {s''}}(D)}
\]
for every $(\zeta,z)\in D$ and for every $f\in \Hol(D)$ such that
\[
\abs{f(\zeta,z)}>\eps \Delta_\Omega^{(\vb+\vd)/q-\vect{s''}-\vect{s'}}(\rho(\zeta,z))\norm{f B^{\vect{s'}}_{(\zeta,z)}}_{A^{q,q}_{\vect {s''}}(D)},
\]
that is, for every $f\in \Hol(D)$ and for every $(\zeta,z)\in D\setminus B_{f,\eps}^{q,\vect {s''},\vect{s'}}$, with the notation of Lemma~\ref{lem:11}.
In addition, by~\cite[Proposition 3.2]{CalziPeloso}, there is a constant $C''>0$ such that
\[
\norm{f}_{A^{\infty,\infty}_{\vect{s''}-(\vb+\vd)/q}(D)}\meg C''\norm{f}_{A^{q,q}_{\vect {s''}}(D)}
\]
for every $f\in A^{q,q}_{\vect {s''}}(D)$, so that
\[
\norm{f B^{\vect{s'}}_{(\zeta,z)}
}_{L^{q}(\mi)}\Meg \frac{C'}{C''} \Delta_\Omega^{\vect{s'}+\vect {s''}-(\vb+\vd)/q}(\rho(\zeta,z))\abs{f(\zeta,z)}
\]
for every $f\in \Hol(D)$ and for every $(\zeta,z)\in D\setminus B_{f,\eps}^{q,\vect {s''},\vect{s'}}$. By Lemma~\ref{lem:11}, we may choose $\eps$ so small that
\[
\norm{f}_{A^{p,p}_{\vs}(D)}\meg 2 \norm{\chi_{D\setminus B_{f,\eps}^{q,\vect {s''},\vect{s'}}}f}_{L^{p,p}_{\vs}(D)}
\]
for every $f\in A^{p,p}_{\vs}(D)$, provided that
\[
\vect{s'}+ \vect{s''}-\vs\in \frac{1}{q}(\vb+\vd) -\frac{1}{2p} \vect{m}-\left( \frac{1}{2 q}-\frac{1}{2 p}  \right)\vect{m'}-(\R_+^*)^r.
\]
Therefore,
\[
\norm{f }_{A^{p,p}_{\vs}(D)}\meg \frac{2}{C' C''} \norm*{(\zeta,z)\mapsto \norm{f B^{\vect{s'}}_{(\zeta,z)} 
}_{L^{q}(\mi)}  }_{L^{p,p}_{\vs-\vect{s'}-\vect{s''}+(\vb+\vd)/q}(D)}
\]
for every $f\in A^{p,p}_{\vs}(D)$.  Hence, it will suffice to show that the linear mapping
\[
T\colon f\mapsto \left[(\zeta,z)\mapsto  \int_D f \abs{ B^{q\vect{s'}}_{(\zeta,z)} }\,\dd \mi  \right]
\]
induces a continuous linear mapping from $L^{p /q}(\mi)$ into $L^{p/q}((\Delta_\Omega^{p[\vs-\vect{s'}-\vect{s''}+(1/q-1/p)(\vb+\vd) ]  }\circ \rho)\cdot \nu_D )$. 
By Lemma~\ref{lem:15}, this is the case if the following conditions are satisfied:
\begin{itemize}
\item[(i)] $\frac{p}{\ell'}\vs\in \frac{1}{2 \ell'}\vect m+\frac{1}{2 \ell}\vect{m'}+(\R_+^*)^r $;

\item[(ii)] $q \vect{s'}+\frac{p}{\ell'}\vs\in \vb+\vd- \frac{1}{2 \ell}\vect m+\frac{1}{2 \ell'}\vect{m'}-(\R_+^*)^r $.
\end{itemize}
This is the case if $\vect{s'}$ is sufficiently small.
\end{proof}

Finally, we provide a sufficient condition for   a measure
$\mi\in\cM_+(D)$ to be a reverse Carleson, but not necessarily
Carleson, measure for the weighted Bergman space $A^{p,q}_\vs(D)$. A
necessary condition  (for sampling measures only)   has already been provided 
in Proposition~\ref{prop:32}. 

\begin{teo}\label{prop:29}
Take $p_1,q_1,  p \in ]0,\infty]$ with $  p <\infty$, $R_0>1$, and
$\vs\in \frac{1}{2 q_1}\vect m+(\R_+^*)^r$ if $q_1<\infty$, while
$\vs\in \R_+^r$ if $q_1=\infty$. Define $  p^* \coloneqq
p_1/(  p -p_1)_+$ and $  q^* \coloneqq q_1/(  p -q_1)_+$. 

Then, there are $\delta_0>0$ and a constant $C>0$ such that the
following hold. For every $(\delta,R)$-lattice
$(\zeta_{j,k},z_{j,k})_{j\in J,k\in K}$ on $D$, with $\delta\in
]0,\delta_0]$ and $R\in ]1,R_0]$, for every Borel partition
$(B_{j,k})_{(j,k)\in J\times K}$ of $D$ such that 
\[
B((\zeta_{j,k},z_{j,k}),\delta)\subseteq B_{j,k}\subseteq B((\zeta_{j,k},z_{j,k}),R\delta)
\]
for every $(j,k)\in J\times K$, and for every  $\mi\in\cM_+(D)$ such that
\[
C'\coloneqq \norm*{  \bigg( \frac{\Delta_\Omega^{  p [\vs-(\vect
b+\vd)/p_1]}(h_k)}{\mi(B_{j,k})} \bigg)_{j,k}
}_{\ell^{  p^*, q^* }(J,K)}  <\infty,
\]
one has
\[
\norm{f}_{A^{p_1,q_1}_{\vs}(D)}\meg  C C'^{1/  p} \delta^{(2n+m)/p_1+m/q_1}  \norm{f}_{L^{  p}(\mi)}
\]
for every $f\in A^{\infty,\infty}_{\vs-(\vb+\vd)/p_1}(D)$, where $h_k\coloneqq \rho(\zeta_{j,k},z_{j,k})$ for every $(j,k)\in J\times K$.
\end{teo}

\begin{oss}
We point out  that the condition $C'<\infty$ is sufficient, but far from
being necessary.   As a matter of fact, using Theorem~\ref{prop:31} one may easily construct reverse Carleson measures which vanish on any given compact subset of $D$. Moreover, the condition
$f\in A^{\infty,\infty}_{\vs-(\vb+\vd)/p_1}(D)$ can be relaxed (in the spirit of~\cite[Theorem 3.23]{CalziPeloso})  but not
eliminated, as observed  in~\cite{Luecking}.
\end{oss}

\begin{proof}
By~\cite[Theorem 3.23]{CalziPeloso}, we may take $\delta_0>0$ and $C_1>0$ so that, if we define 
\[
S_-\colon \Hol(D) \ni f \mapsto \left( \Delta_\Omega^{\vs-(\vb+\vd)/p_1}(h_k) \min_{\overline
B((\zeta_{j,k},z_{j,k}),R\delta)}\abs{f} \right)_{j,k}\in
\R_+^{J\times K}, 
\]
then
\[
\frac{1}{C_1}\norm{f}_{A^{p_1,q_1}_{\vs}(D)}\meg\delta^{(2n+m)/p_1+m/q_1}
\norm{S_- f}_{\ell^{p_1,q_1}(J,K)}
\meg C_1\norm{f}_{A^{p_1,q_1}_{\vs}(D)}
\]
for every $f\in A^{\infty,\infty}_{\vs-(\vb+\vd)/p_1}(D)$ and for every lattice as in the statement.
Then, take $f\in A^{\infty,\infty}_{\vs-(\vb+\vd)/p_1}(D)$, and observe that, by H\"older's inequality,
\[
\begin{split}
&\norm*{  \bigg( \frac{\Delta_\Omega^{  p  [\vs-(\vb+\vd)/p_1]}(h_k)}{\mi(B_{j,k})} \bigg)_{j,k}  }^{1/  p}_{\ell^{  p^*,q^*}(J,K)} \norm{f}_{L^{q_2}(\mi)}\\
&\qquad\Meg \norm*{  \bigg( \frac{\Delta_\Omega^{\vs-(\vb+\vd)/p_1}(h_k)}{\mi(B_{j,k})^{1/  p}} \bigg)_{j,k}  }_{\ell^{  p p^*, p q^*}(J,K)} \norm*{\big(\Delta_\Omega^{(\vb+\vd)/p_1-\vs}(h_k) \mi(B_{j,k})^{1/  p}  (S_- f)_{j,k} \big)_{j,k}  }_{\ell^{  p}(J\times K)}\\
&\qquad\Meg \norm{S_- f}_{\ell^{p_1,q_1}(J,K)} \\
&\qquad\Meg \frac{1}{C_1\delta^{(2n+m)/p_1+m/q_1} }  \norm{f}_{A^{p_1,q_1}_{\vs}(D)},
\end{split}
\]
whence the result.
\end{proof}

We conclude this section with some rather `pathological' examples of non-Carleson reverse Carleson measures for the spaces $A^{p,q}_s(\C_+)$ which fail to satisfy the preceding sufficient condition.

\begin{oss}
Take $q\in ]0,\infty]$ and $p\in ]0,\infty[$. In addition, fix $a<b$ in $\R$ and $\nu\in \cM_+(\R_+^*)$ such that $\nu(]0,1])=\infty$. Define
\[
\mi\coloneqq (\chi_{[a,b]}\cdot \Hc^1)\otimes \nu,
\]
and observe that $\mi$ is a positive Radon measure on $\C_+\cong\R\times \R_+^*$. Take a non-zero $f\in H^q$, and let us prove that $\norm{f}_{L^p(\mi)}=\infty$. Indeed, the function 
\[
f^*\colon \R\ni x\mapsto \lim_{y\to 0^+} f(x+i y)\in \C
\]
is well-defined and non-zero almost everywhere (cf.~\cite[Corollary to Theorem 11.1]{Duren2}). Hence, by Fatou's lemma,
\[
\liminf_{y\to 0^+} \int_a^b \abs{f(x+i y)}^p\,\dd x\Meg\int_a^b \abs{f^*(x)}^p\,\dd x>0,
\]
so that clearly
\[
\int_{\C_+}\abs{f}^p\,\dd \mi=+\infty.
\]
By the arbitrariness of $f$, it is clear that $\mi$ is (trivially) reverse $p$-Carleson for $H^q$.
\end{oss}

In order to extend the preceding example to the spaces $A^{p,q}_s(\C_+)$ for $s>0$, we need the following lemma.

\begin{lem}\label{lem:4}
Take a non-zero $u\in \Sc'(\R)$. If the Fourier transform of $u$ is supported in a half-line, then $\Supp{u}=\R$.
\end{lem}

\begin{proof}
Up to replacing $u$ with $\ee^{i \xi\,\cdot\,} u(\eps\,\cdot\,)$ for
some $\xi\in \R$ and some $\eps\in \Set{-1,1}$, we may assume that the
Fourier transform of $u$ is supported in $\R_+$. Now, take a non-zero
$g\in \Sc(\R)$ whose Fourier transform is supported in $\R_+$, and
observe that $g u$  is non-zero, and has a Fourier transform
supported in $\R_+$, thanks to~\cite[Theorems XIII and XIV of Chapter
VI]{Schwartz}.    
Now, fix $\phi\in C^\infty_c(\R)$ such that $\int_\R \phi(x)\,\dd x=1$ and
$\Supp{\phi}\subseteq [-1,1]$, and define $\phi_j\coloneqq
2^{j}\phi(2^{j}\,\cdot\,)$ for every $j\in \N$. Then, it is easily verified that $(g u)*\phi_j$belongs to $\Sc(\R)$ and that its Fourier transform is supported in $\R_+$ for every $j\in \N$. In addition, $(g u)*\phi_j\to g u$ in $\Sc'(\R)$,
so that  $(g u)*\phi_j\neq 0$ if $j$ is sufficiently
large. Therefore,~\cite[Corollary to Theorem 11.1, and Theorem
11.9]{Duren2} imply that $(g u)*\phi_j$ is non-zero almost everywhere
if $j$ is sufficiently large. 
Now, assume by contradiction that $\Supp{u}\neq \R$, and take $a,b\in
\R$ such that $a<b$ and $u$ vanishes on $]a,b[$. Then, clearly $(g
u)*\phi_j$ vanishes on $]a+2^{-j},b-2^{-j} [$, which is not empty if
$b-a>2^{1-j}$, that is, if $j$ is sufficiently large: contradiction. 
\end{proof}

\begin{oss}
Take  $p_1,q_1\in ]0,\infty]$ and $s,p\in ]0,\infty[$. In addition, fix $a<b$ in $\R$ and $\nu\in \cM_+(\R_+^*)$ such that $\nu(]0,1])=\infty$. Define
\[
\mi\coloneqq (\chi_{[a,b]}\cdot \Hc^1)\otimes \nu,
\]
and observe that $\mi$ is a positive Radon measure on $\C_+\cong\R\times \R_+^*$. Take a non-zero $f\in A^{p_1,q_1}_s$, and let us prove that $\norm{f}_{L^q(\mi)}=\infty$. 
Take $(g^{(\eps)})$ as in~\cite[Lemma 1.22]{CalziPeloso}, and observe that $g^{(\eps)}f\in A^{1,1}_{s+(1/p_1-1)_+}$ for every $\eps>0$ (cf.~\cite[Proposition 3.2]{CalziPeloso}). Hence,~\cite[Theorem 1.7]{BekolleBonamiGarrigosRicci} implies that $(g^{(\eps)}f)_y$ converges to some non-zero $f_0$ in $\Sc'(\R)$, for $y\to 0^+$, and that the support of the Fourier transform of $f_0$ is contained in $\R_+$. Therefore, Lemma~\ref{lem:4} implies that $\Supp{f_0}=\R$, so that clearly 
\[
0<\liminf_{y\to 0^+} \int_a^b \abs{(g^{(\eps)}f)(x+i y)}^q\,\dd x\meg \liminf_{y\to 0^+} \int_a^b \abs{f(x+i y)}^q\,\dd x.
\]
Hence,
\[
\int_{\C_+}\abs{f}^q\,\dd \mi=+\infty.
\]
By the arbitrariness of $f$, it is clear that $\mi$ is (trivially) reverse $p$-Carleson for $A^{p_1,q_1}_s$.
\end{oss}


\begin{thebibliography}{99}  

 
\bibitem{AbateSaracco}
Abate, M., Saracco, A., Carleson Measures and Uniformly Discrete Sequences in Strongly Pseudoconvex Domains, \emph{J.\ Lond.\ Math.\ Soc.} \textbf{83} (2011), p.~587--605.

\bibitem{Abate}
Abate, M., Carleson Measures and Toeplitz Operators, in \emph{Metrical and dynamical aspects in complex analysis}, Ed. L.\ Blanc-Centi, Lecture Notes in Mathematics 2195, Springer, Berlin, 2017, pp.~141--157.

\bibitem{AbateRaissy}
Abate, M., Raissy, J., Skew Carleson Measures in Strongly Pseudoconvex Domains, \emph{Compl.\ Anal.\ Oper.\ Th.} \textbf{13} (2019), p.~405--429.

\bibitem{AbateMongodiRaissy}
Abate, M., Mongodi, S., Raissy, J., Toeplitz Operators and Skew Carleson Measures for Weighted Bergman Spaces on Strongly Pseudoconvex Domains, \emph{J.\ Operator Th.} \textbf{84} (2020), p.~339--364.




\bibitem{ACMPS} Arcozzi N., Chalmoukis N., Monguzzi A., Peloso
  M.\ M., Salvatori M., The Drury--Arveson space on the Siegel
  upper half-space and a von Neumann type inequality, preprint
  2021 arXiv:2103.05067.

\bibitem{AMPS} Arcozzi N., Monguzzi A., Peloso M.\ M., Salvatori M.,
Paley--Wiener theorems on the Siegel upper half-space,  \emph{J.\
  Fourier Anal.\ Appl.} (2019), 1958--1986. 
  
\bibitem{ArcozziRochbergSawyer3}
Arcozzi, N., Rochberg, R., Sawyer, E., Carleson measures for analytic Besov spaces, \emph{Rev.\ Mat.\ Iberoamericana} \textbf{18} (2002), p.~443--510. 
  
\bibitem{ArcozziRochbergSawyer1}
\bysame, Carleson Measures and Interpolating Sequences for Besov Spaces on Complex Balls, \emph{Mem.\ Amer.\ Math.\ Soc.} \textbf{182} (2006), p.~vi+163.

\bibitem{ArcozziRochbergSawyer5}
\bysame, Some Problems on Carleson Measures for Besov--Sobolev Spaces, \emph{Topics in Complex Analysis and Operator Theory}, Proceedings of the Winter School held in Antequera, Malaga, Spain (February 5-9 2006), p.~141--148.

\bibitem{ArcozziRochbergSawyer2}
\bysame, Carleson Measures for the Drury--Arveson Hardy Space and other Besov--Sobolev Spaces on Complex Balls, \emph{Adv.\ Math.} \textbf{218} (2008), p.~1107--1180.

\bibitem{ArcozziRochbergSawyer4}
\bysame, Capacity, Carleson Measures, Boundary Convergence, and Exceptional Sets, "Perspective in Harmonic Analysis and Applications" in honor of V.G. Maz'ya 70-th birthday, \emph{AMS Proceedings of Symposia in Pure and Applied Mathematics} (2008). 

\bibitem{AstengoCowlingDiBlasioSundari}
Astengo, F., Cowling, M., Di Blasio, B., Sundari, M., \emph{Hardy's Uncertainty Principle on Certain Lie Groups}, {J.\ London Math.\ Soc.} {62} (2000), p.~461--472.


\bibitem{BekolleBonamiGarrigosRicci}
B\'ekoll\'e, D., Bonami, A., Garrig\'os, G., Ricci, F., \emph{Littlewood--Paley Decompositions Related to Symmetric Cones and Bergman Projections in Tube Domains}, {P.\ Lond.\ Math.\ Soc.} {89} (2004), p.~317--360.


\bibitem{BekolleSehba}
B\'ekoll\'e, D., Sehba, B.\ F., Some Carleson Measures for the Hilbert--Hardy Space
of Tube Domains over Symmetric Cones, \emph{Eur. J. Math.} \textbf{5} (2019), p.~585--610.

\bibitem{BekolleSehbaTchoundja}
B\'ekoll\'e, D., Sehba, B.\ F., Tchoundja, E.\ L., The Duren--Carleson Theorem in Tube
Domains over Symmetric Cones, \emph{Integr.\ Equ.\ Oper.\ Theory}
\textbf{86} (2016), p.~475--494.  



\bibitem{CalziPeloso}
Calzi, M., Peloso., M.\ M., Holomorphic Function Spaces on Homogeneous
Siegel Domains, \emph{Diss. Math.} to appear,  arXiv:2009.11083. 


\bibitem{T&C}
\bysame,  Toeplitz and Ces\`aro Operators on Homogeneous Siegel Domains,
 preprint 2021,   arXiv:2105.06348. 

\bibitem{Carleson1}
Carleson, L., An Interpolation Problem for Bounded Analytic Functions, \emph{Amer.\ J.\ Math.} \textbf{80} (1958), p.~921--930.

\bibitem{Carleson2}
\bysame, Interpolation by Bounded Analytic Functions and the Corona Problem, \emph{Ann.\ of Math.} \textbf{76} (1962), p.~547--559.

\bibitem{CimaWogen}
Cima, J.\ A., Wogen, W.\ R., A Carleson measure theorem for the Bergman space on the ball, \emph{J. Oper. Th.} \textbf{7} (1982), p.~157--165.

\bibitem{Duren}
Duren, P.\ L., Extension of a Theorem of Carleson, \emph{Bull.\ Amer.\ Math.\ Soc.} \textbf{75} (1969), p.~143--146.


\bibitem{Duren2}
\bysame, \emph{Theory of $H^p$ Spaces}, Academic Press, 1970.

\bibitem{FarautKoranyi}
Faraut, J., Kor\'anyi, A., \emph{Analysis on Symmetric Cones}, Clarendon Press, 1994.


\bibitem{Folland}
Folland, G.\ B., \emph{Harmonic Analysis in Phase Space}, Princeton University Press, 1989.

\bibitem{FricainHR}
Fricain, E., Hartmann, A., Ross, W.\ T., A survey on reverse Carleson  measures, in \emph{Harmonic analysis, function theory, operator   theory, and their applications}, Theta Ser.\ Adv.\ Math., 19, Theta,  Bucharest, 2017, p.~91--123.



\bibitem{Gindikin}
Gindikin, S.\ G., Analysis in Homogeneous Domains, \emph{Russ.\ Math.\ Surv.} \textbf{19} (1964), p.~1--89.


\bibitem{Grafakos}
Grafakos, L., \emph{Classical Fourier Analysis}, 2nd ed., Springer, 2008.



\bibitem{HuLvZhu}
Hu, Z., Lv, X., Zhu, K., Carleson Measures and Balayage for Bergman Spaces of Strongly Pseudoconvex Domains, \emph{Math.\ Nachr.} \textbf{289} (2015), p.~1237--1254.

\bibitem{Lang}
Lang, S., \emph{Fundamentals of Differential Geometry}, Springer-Verlag, 1999.

\bibitem{Luecking}
Luecking, D.\ H., Inequalities on Bergman Spaces, \emph{Illinois J. Math.} \textbf{25} (1981), p.~1--11.

\bibitem{Luecking3}
\bysame, Equivalent Norms on $L^p$ Spaces of Harmonic Functions, \emph{Mh.\ Math.} \textbf{96} (1983), p.~133--141.

\bibitem{Luecking7}
\bysame, A Technique for Characterizing Carleson Measures on Bergman Spaces, \emph{Proc.\ Amer.\ Math.\ Soc.} \textbf{87} (1983), p.~656--660.

\bibitem{Luecking4}
\bysame, Closed Ranged Restriction Operators on Weighted Bergman Spaces, \emph{Pac.\ J.\ Math.} \textbf{110} (1984), p.~145--160.

\bibitem{Luecking2}
\bysame, Forward and Reverse Carleson Inequalities for Functions in Bergman Spaces and their Derivatives, \emph{Am.\ J.\ Math.} \textbf{107} (1985), p.~85--111.

\bibitem{Luecking6}
\bysame, Dominating Measures for Spaces of Analytic Function, \emph{Illinois J.\ Math.} \textbf{32} (1988), p.~23--39.

\bibitem{Luecking8}
\bysame, Embedding Derivatives of Hardy Spaces into Lebesgue Spaces, \emph{Proc.\ London Math.\ Soc.} \textbf{63} (1991), p.~595--619.

\bibitem{Luecking5}
\bysame, Sampling Measures for Bergman Spaces on the Unit Disk, \emph{Math.\ Ann.} \textbf{316} (2000), p.~659--679.

  \bibitem{MPS} A.\ Monguzzi, M.\ M.\ Peloso, M.\ Salvatori,
    Sampling in spaces of entire functions of exponential type
    in $\C^{n+1}$, preprint 2021, arXiv:2105.08458. 
    
\bibitem{Murakami}
Murakami, S., \emph{On automorphisms of Siegel Domains}, Springer-Verlag, 1972.

\bibitem{NanaSehba}
Nana, C., Sehba, B.\ F., Carleson Embeddings and Two Operators on
Bergman Spaces of Tube Domains over Symmetric Cones, \emph{Integr.\ Equ.\ Oper.\ Theory} \textbf{83} (2015), p.~151--178.

\bibitem{Power}
Power, S.\ C., Vanishing Carleson Measures, \emph{Bull.\ London Math.\ Soc.} \textbf{12} (1980), p.~207--210.


\bibitem{Satake}
Satake, I., On Classification of Quasi-Symmetric Domains, \emph{Nagoya Math.\ J.} \textbf{62} (1976), p.~1--12.


\bibitem{Schwartz}
Schwartz, L., \emph{Th\'eorie des distributions}, Hermann, 1966.


\bibitem{Stegenga}
Stegenga, D.\ A., Multipliers of the Dirichlet Space, \emph{Illinois J.\ Math.} \textbf{24} (1980), p.~113--139.

\bibitem{Vinberg}
Vinberg, E.\ B., The Theory of Convex Homogeneous Cones, \emph{Trans.\ Moscow Math.\ Soc.} \textbf{12} (1965),p.~340--403.

\bibitem{Xu}
Xu, Y., \emph{Theory of Complex Homogeneous Bounded Domains}, Science Press, 2000.

\bibitem{Wu}
Wu, Z., Carleson measures and multipliers for Dirichlet spaces,
\emph{J.\ Funct.\ Anal.} \textbf{169} (1999),  p.~148--163.  



\bibitem{Zhu2}
\bysame, \emph{Spaces of holomorphic functions in the unit ball},
Graduate Texts in Mathematics, 226. Springer-Verlag, New York,
2005. 

\end{thebibliography}
\end{document}